\documentclass{amsart}
\usepackage{amsmath}
\usepackage{amsfonts}
\usepackage{amssymb}
\usepackage{amsthm}
\usepackage{setspace}
\usepackage{graphicx}
\usepackage{wasysym}
\usepackage{listings}
\usepackage{mathrsfs}
\usepackage{tikz}

\def\bfa{{\bf a}}
\def\bfb{{\bf b}}
\def\Skl{\operatorname{1^\nu-Set}}
\def\hSkl{\operatorname{P1^\nu-Set}}

\def\skeletona{K-hypersurface}
\def\Hdim{\operatorname{Hdim}}
\def\HSpec{\operatorname{HSpec}}
\def\ConSpan{\operatorname{Conv}}
\def\hgt{\operatorname{hgt}}
\def\condeg{\operatorname{d}_{\operatorname{conv}}}
\def\skeletonb{K-variety}
\def\skeletonbs{K-varieties}
\def\vsemifield0{$\nu$-semifield$^\dagger$}
\def\vsemifields0{$\nu$-semifields$^\dagger$}
\def\vdomain0{$\nu$-domain$^\dagger$}
\def\vdomains0{$\nu$-domains$^\dagger$}

\def\scrL{\mathscr L}

\def\onenu{1^\nu}
\def\Plr{\mathcal Plr}

\def\Frac{{\operatorname{Frac}}}
\def\FF{{\langle F \rangle}}

\theoremstyle{plain}
\newtheorem{thm}{Theorem}[subsection]

\newtheorem{exampl}[thm]{Example}
\newtheorem{lem}[thm]{Lemma}
\newtheorem{cor}[thm]{Corollary}
\newtheorem{prop}[thm]{Proposition}
\newtheorem{constr}[thm]{Construction}
\theoremstyle{remark}
\newtheorem{rem}[thm]{Remark}
\newtheorem{note}[thm]{Note}

\theoremstyle{definition}
\newtheorem{defn}[thm]{Definition}
\newtheorem{nota}[thm]{Notation}
\newtheorem{exmp}[thm]{Example}

\numberwithin{equation}{section}

\DeclareMathOperator{\PCon}{\mathcal P}

\renewcommand{\Im}{\operatorname{Im\,}}

\def\la{\lambda}

\def\Skel{\operatorname{1}_{\operatorname{loc}}}
\def\Ker{\mathcal Kern}
\def\dotplusss{+}
\def\rat{\operatorname{rat}}
\def\Corn{\operatorname{C}_{\operatorname{loc}}}
\def\semiring0{semiring$^\dagger$}
\def\Semirings0{Semiring$^\dagger$}
\newcommand{\etype}[1]{\renewcommand{\labelenumi}{(#1{enumi})}}
\def\eroman{\etype{\roman}}
\def\Net{\mathbb N}
\def\semialg0{semi-algebra$^\dagger$}
\def\semirings0{semirings$^\dagger$}
\def\domain0{domain$^\dagger$}
\def\domains0{domains$^\dagger$}
\def\field0{semifield$^\dagger$}
\def\semifield0{semifield$^\dagger$}
\def\Semifield0{Semifield$^\dagger$}
\def\semifields0{semifields$^\dagger$}
\def\Semifields0{Semifields$^\dagger$}
\def\Real{\mathbb R}
\def\R{\mathbb R}
\def\Q{\mathbb Q}
\def\tG{\mathcal{G}}

\def\ptM{M}

\def\tSS{S}
\def\tS{\mathcal{S}}
\def\tT{\mathcal{T}}

\def\FunStM{\operatorname{Fun} (S,\ptM)}

\def\FunSF{\operatorname{Fun} (S,F)}

\def\one{\mathbf{1}}
\def\zero{\mathbf{0}}
\def\rone{\one_R}

\def\rzero{\zero_R}

\def\skel{{\operatorname{1-loc}}}
\def\vrp{\varphi}

\def\a{\alpha}

\def\nucong{\cong_\nu}
\def\tlvrp{\widetilde \varphi}
\def\Cong{\Omega}
\newcommand{\dss}[1]{\quad {#1} \quad }

\def\fone{\one_F}

\def\FunSR{\operatorname{Fun} (\tS,R)}
\def\Fun{{\operatorname{Fun}}}
\def\FunF{\Fun (F^{(n)},F)}
\def\Rplus{(\mathbb R,+)}

\begin{document}
\pagenumbering{arabic}

\title{Kernels in tropical geometry and a Jordan-H\"{o}lder theorem}%
\author{Tal Perri, Louis Rowen}%
\email {talperri@zahav.net.il} \email{rowen@math.biu.ac.il}
\subjclass[2010]  {Primary: 14T05, 12K10 ; Secondary: 16Y60 }
\begin{abstract}
A correspondence exists between affine tropical varieties and
algebraic objects, following the classical Zariski correspondence
between irreducible affine varieties and the prime spectrum of the
coordinate algebra in affine algebraic geometry. In this context the
natural analog of the polynomial ring over a field is the polynomial
semiring over a semifield, but one obtains homomorphic images of
coordinate algebras via congruences rather than ideals, which
complicates the algebraic theory considerably.

In this paper, we pass to the semifield $F(\la_1, \dots, \la_n)$ of
fractions of the polynomial semiring, for which there already exists
a well developed theory of \textbf{kernels}, also known as
\textbf{lattice ordered subgroups}; this approach enables us to
switch the structural roles of addition and multiplication and makes
available much of the extensive theory of chains of homomorphisms of
groups. The parallel of the zero set now is the \textbf{1-set}.
(Idempotent semifields correspond to lattice ordered groups, and the
kernels to normal convex $l$-subgroups.)

These notions are refined in the language of supertropical algebra
to $\nu$-\textbf{kernels} and $\onenu$-\textbf{sets},   lending
more precision to tropical varieties when viewed as sets of common
roots of polynomials. The $\nu$-kernels corresponding to
(supertropical) hypersurfaces are the $\onenu$-sets of corner
internal rational functions.  The $\nu$-kernels corresponding to
``usual'' tropical geometry are the regular, corner-internal
$\nu$-kernels.

 In analogy to Hilbert's celebrated Nullstellensatz which
provides a correspondence between radical ideals and zero sets, we
develop a correspondence between $\onenu$-sets and a well-studied
class of $\nu$-kernels of the rational semifield called
\textbf{polars}, originating from the theory of lattice-ordered
groups. This correspondence becomes simpler and more applicable
when restricted to a special kind of $\nu$-kernel, called
\textbf{principal}, intersected with the $\nu$-kernel generated by
$F$. We utilize this theory to study tropical roots of finite sets
of tropical polynomials.

 For our main application, we   develop algebraic  notions such as composition series and convexity degree,
 along with notions having a geometric interpretation, like reducibility and hyperdimension,
 leading to a tropical version of the
Jordan-H\"{o}lder theorem for the relevant class of $\nu$-kernels.
\end{abstract}

\maketitle
\setcounter{tocdepth}{1}{\small \tableofcontents}

\section{Introduction}

This paper is a combination of \cite{Kern}  and \cite{KS}. The
underlying motive of tropical algebra is that the valuation group
of the order valuation (and related valuations) on the Puiseux
series field is the ordered group $(\Real, +)$ or $(\Q, +)$
(depending on which set one uses for powers in the Puiseux
series), which can also be viewed as the max-plus algebra on
$(\Real, +)$ or $(\Q, +)$. This leads one to a procedure of
tropicalization, based on valuations of Puiseux series, which
takes us from polynomials over Puiseux series to ``tropical''
polynomials over the max-plus algebra. One of the main goals of
tropical geometry is to study the ensuing varieties.
Traditionally, following Zariski, in the affine case, one would
pass to the ideal of the polynomial algebra generated by the
polynomials defining the variety. However, this does not work well
in the tropical world, since one winds up with ``too many''
varieties for an intuitive theory of dimension. For example,
intuitively one would like the tropical line to be an irreducible
variety of dimension 1. But the intersection of the tropical lines
$\la _1 + \la_2 + 1$ and $\la _1 +\la_2 + 2$ is the diagonal ray
emanating from $(2,2)$, so one must deal with rays, and the
tropical line is the union of three rays. The customary solution
to this dilemma is to introduce the ``balancing condition,'' given
in rather general terms in Yu~\cite{Y}, but this is tricky in
higher dimensions and does not relate directly to the algebraic
structure.

Another promising direction, which we follow, is to develop an
analog of the celebrated Zariski correspondence, in order to
understand tropical geometry in terms of the algebraic structure
of various \semirings0 refining the polynomial algebra over the
max-plus algebra. Such a correspondence is one of the traditional
ways of treating algebraic geometry in terms of universal algebra,
as described in~\cite[Theorem~1.1]{UnivAlg} and in considerable
generality in \cite{Alglogic}, and was adapted to tropical algebra
by means of corner roots, as described in~\cite{IntroTropGeo}; a
structural description of roots in the ``supertropical language''
was given in \cite{SuperTropicalAlg}.

Let $\mathbb{K} = \mathbb{C}\{\{t\}\}$ be the field of Puiseux
series on the variable $t$, which is the set of formal series of
the form $f = \sum_{k = k_0}^{\infty} c_k t^{k/N}$ where $N \in
\mathbb{N}$, $k_0 \in \mathbb{Z}$, and $c_k \in \mathbb{C}$. Then
$\mathbb{K}$ is an algebraically closed field equipped with the
\textbf{Puiseux valuation} $val : \mathbb{K}^{\ast} =
\mathbb{C}\{\{t\}\} \setminus \{0\} \rightarrow \mathbb{Q} \subset
\mathbb{R}$ defined by
\begin{equation}
val(f) = -\min_{c_k \neq 0}\{k/N\}.
\end{equation}
%
%

The \textbf{tropicalization} of a Laurent polynomial  $f = \sum
c_{u}x^{u} \in \mathbb{K}[x_1^{\pm 1},...,x_n^{\pm 1}]$ is defined
to be $trop(f) : \mathbb{R}^{(n)}  \rightarrow \mathbb{R}$ given by
\begin{equation}
trop(f)(w) = max(val(c_u) + w \cdot u)
\end{equation}
where $u \in \mathbb{N}$ is the power vector and $\cdot$ is the
scalar product.

Given a Laurent polynomial  $f   \in \mathbb{K}[x_1^{\pm
1},...,x_n^{\pm 1}]$, one defines its \textbf{tropical hypersurface}
as
\begin{equation}
T(trop(f)) = \{ w\in \mathbb{R}^{(n)}  : \text{the maximum in
$trop(f)$ is achieved at least twice}. \}
\end{equation}

 To develop tropical algebraic geometry further, one has the choice either of working directly
at the level of polynomials over Puiseux series, and then
tropicalizing, or first tropicalizing and then utilizing a Zariski
correspondence at the tropical level.
 The latter is
attractive from the point of view of being able to work directly
with a simpler concept, and one is led to continue the algebraic
study of tropical geometry by means of corner roots of
polynomials, developing algebraic geometry over the ordered group
$(\Q,+)$; some relevant references are
\cite{Ca,GG,SuperTropicalAlg,Mac}. One tricky issue is that we
need to consider intersections of tropical hypersurfaces.

There has also been recent interest in algebraic geometry over
monoids, and much of the basic algebraic material can be found in
\cite{CHWW}. But  many problems arise in formulating algebraic
geometry directly over the max-plus algebra, not the least of which
is the failure of the max-plus algebra to reflect the uncertainty
involved in taking the value of the sum of two polynomials. This was
dealt with in \cite{Layered,SuperTropicalAlg}, by refining the
max-plus algebra to a ``supertropical \semifield0'' $F$ and, even
more generally, to a ``layered'' \semifield0, i.e., semifield
without 0. \cite{GG} also treats tropical algebraic geometry in
terms of valuations.

When developing tropical affine geometry, it is natural to try to
obtain the Zariski-type correspondence between tropical varieties
and coordinate semirings. This approach is used for our proposed
definition of tropical (affine) variety, in parallel to classical
affine algebraic geometry. (Then the usual constructions of schemes
and sheaves can be obtained for ``semiringed spaces'' by noting that
negation is not needed in the proofs; the schemes correspond to the
congruences defining tropical varieties, and the sheaves to the
localizations of the coordinate semirings.) We do not go into the
details here, for lack of space.)

 Our approach to affine tropical geometry also focuses on a
Zariski correspondence parallel to the Zariski correspondence in
classical algebraic geometry, which  is obtained from the passage
between ordered groups and \semifields0. Since tropical geometry
often is understood in terms of corner roots of polynomials over
the target of the order valuation, one should investigate the
\semiring0 $R =  F[\la_1, \dots, \la_n]$ of polynomials over a
\semifield0 $F$, as well as its homomorphic images.

There is a natural homomorphism $\psi$ from $F[\la_1, \dots, \la_n]$
to the semiring  $\FunF$ of functions from $F^{(n)}$ to $F$,
obtained by viewing any polynomial as a function. Since polynomials
over the max-plus algebra could have inessential terms which do not
affect their values, two different polynomials could have the same
image under $\psi$  (such as  $5\la^2 +\la +7$ and $5\la^2  +7$).
But their corner roots must be the same, so it is convenient to work
with the image $\psi(F[\la_1, \dots, \la_n ])$, which we denote as
$\overline{F[\la_1, \dots, \la _n]}$. This is particularly apt when
one views the coordinate semiring of a variety as the set of
polynomial functions restricted to that variety. Viewed in this way,
all coordinate semirings are cancellative, and one can define
morphisms of varieties in terms of \semiring0 homomorphisms.   (The
information lost by passing to $\overline{F[\la_1, \dots, \la _n]}$
can be accounted for by \textbf{region fractions},
cf.~Definition~\ref{regionfrac}. These arise in the decomposition
given in Theorem~\ref{thm_HP_expansion}, described presently.)

In general \semiring0 theory, such homomorphisms are not attained
by means of (prime) ideals, but rather via congruences, which are
defined as subsets of $R\times R$ rather than of $R$. Since
congruences replace ideals in the study of homomorphic images, one
is led to the study of congruences of \semirings0. Developing
ideas of \cite{BeE}, Joo and Mincheva
 \cite{JoM} have developed an elegant theory of  prime congruences,
 to be discussed in Remark~\ref{primecon}.

Several authors, \cite{GG, MacRin, Mac} have explored ``bend
congruences.'' Even so, the theory of congruences is considerably
more complicated than the study of ideals of rings, since it
involves substructures of $R\times R$ rather than $R$.
Furthermore, the polynomial algebra over a \semifield0 need not be
Noetherian, cf.~\cite{IzhakianRowenIdeals}, although  \cite{JoM}
shows that chains of prime congruences are bounded.

 The main
innovation of this paper, put forth in 2013 by T.~Perri in his
doctoral dissertation  \cite{AlgAspTropMath}, is to  switch the
roles of multiplication and addition (which is natural enough,
since we started with the max-plus algebra). Although
mathematically equivalent, this switch enables us to view
$\overline{F[\la_1, \dots, \la _n]}$ as a lattice-ordered
(multiplicative) monoid and pass to its group of fractions,
focusing on the group structure. This leads us further into the
classical theory of lattice-ordered monoids and groups \cite{B,
CG,Ho,Ko,M,Sa}
  and their
corresponding idempotent \semifields0.

Put another way, although the polynomial \semiring0 is not a
\semifield0, it is a cancellative \semiring0 when its elements are
viewed as functions,  so we can  view polynomials as functions and
then pass to the \semifield0 of fractions, which is called the
\semifield0 of \textbf{rational functions} over~ $F$. (The
information thereby lost is compensated by  Homomorphisms of
idempotent \semifields0 have been studied long ago in the literature
\cite{Hom_Semifields,Semifields_LO_Groups,Cont_Semifields}, where
the homomorphic images are described in terms of what they call
\textbf{(\semifield0) kernels}, which as noted in
\cite{Hom_Semifields} are just the convex (normal)  $l$-subgroups of
the corresponding lattice-ordered groups. Since \semifield0 kernels
are subgroups which can be described algebraically, cf.~
\cite[Theorem~3.2]{Hom_Semifields}, they are more amenable to the
classical structure theorems of group theory (Noether isomorphism
theorems, Jordan-H\"{o}lder   theorem, etc.) than congruences. For
the supertropical theory, we need the $\nu$-analog of kernels, which
we call $\nu$-\textbf{kernels}.

The parallel notion to a zero set in algebraic geometry now is the
set of points which when substituted into a function give the
value 1 instead of 0, so we call these sets
\textbf{$\onenu$-sets}. Our ultimate goal is a 1:1 correspondence
between  sets of corner roots (which in the supertropical language
are ``ghost roots'') of polynomials and $\onenu$-sets of rational
functions.

 A $\nu$-kernel
is called \textbf{principal} if it is generated by one element.
One big advantage of the use of $\nu$-kernels is that the product
of finitely many principal $\nu$-kernels is a principal
$\nu$-kernel,
cf.~Theorem~\ref{cor_max_semifield_principal_kernels_operations},
and thus varieties defined as finite intersections of
hypersurfaces still can be described in terms of principal
$\nu$-kernels. This enables us to study tropical varieties that
are not necessarily hypersurfaces.


We come upon various correspondences, but for some reason to be
explained below, we need to intersect down to the $\nu$-kernel
generated by the constant functions, which we
 call $\langle F \rangle$, in order to obtain the correspondence given in
Corollary~\ref{corlatticegeneratedbyCIkernels}:

\medskip

 The lattice of (tangible) simultaneous corner loci with respect to a finite set of polynomials corresponds
to the lattice of principal (tangible)   $\nu$-kernels of $\langle
F \rangle$.

\medskip
This result, a consequence of Theorem~\ref{thmcireglattice}, to be
described soon, enables one to transfer the geometric theory to
$\nu$-kernels of \vsemifields0. The preparation for its proof
requires a careful description of various different kinds of
$\nu$-kernels, each with its particular geometric properties, and
takes the bulk of this paper.

 The application of $\nu$-kernels to
tropical geometry, given in \cite{AlgAspTropMath}, is related to
the bend congruences of~\cite{GG}, as discussed in
Remark~\ref{bend}.

To obtain our results algebraically, we need a slight modification of the results about \semifields0 in
the literature. As explained in
\cite{IzhakianKnebuschRowen2012Algebraicstructures} and elaborated
in \cite{Layered,IzhakianKnebuschRowen2011CategoriesII}, there are
other \semirings0~$R$ covering the usual idempotent max-plus
algebra, which provide better tools for examining the algebraic
structure arising from valuations on Puiseux series, and these
\semirings0 are no longer embeddable into \semifields0. We list the
relevant structures in increasing level of refinement:

\begin{itemize}
\item The max-plus algebra
\item Supertropical algebra
\item Layered tropical algebras
\end{itemize}

In each of the latter three cases, although the multiplicative
monoid is not ordered, there is an underlying partially ordered
monoid, so we can modify the results about (semifield)
$\nu$-kernels in the literature to suit our purposes. Furthermore,
in the supertropical and layered situations, there is a ``ghost
map'' $\nu$ from $R$ to  the set of ``ghost elements'' $\tG$,
which enables us to compare elements in $R$ via their images in
the ordered monoid~$\tG$. We treat only the supertropical setting
here, since we feel that the  theory can be described more
concisely in that context. A more thorough description using the
layered setting would require even more technical detail.

Let us proceed with a more detailed overview of this paper. Since
both \vsemifields0 and their $\nu$-kernels may be unfamiliar to
many researchers, we review  ordered monoids and their \semirings0
in \S\ref{back}.  After an introduction to the supertropical
theory in \S\ref{supertrop}, the basic notions of \vsemifield0
$\nu$-kernels (and their supertropical analogs) are presented in
\S\ref{ker0}.

 The pertinence to tropical geometry is given in
\S\ref{supertrop1}. Corner roots are replaced by  kernel roots,
those points whose value at each function in the $\nu$-kernel is
$\nu$-congruent to 1; these are called $\onenu$-sets, which are
the kernel-theoretic analogues of corner loci. The basic
Zariski-type correspondence between $\onenu$-sets and
$\nu$-kernels is given in
 Theorem~\ref{prop_correspondence}. But different $\nu$-kernels could
 correspond to the same $\onenu$-set, cf.~Example~\ref{nonuniq}.
 This ambiguity is overcome by intersecting with the $\nu$-kernel $\langle F
 \rangle$. Indeed, one still obtains all the $\onenu$-sets, by
 Proposition~\ref{intersects}, but now  Proposition~\ref{prop_generator_of_skel2}
 yields uniqueness in the correspondence, eventually yielding
the desired 1:1 correspondence between principal $\onenu$-sets and
principal $\nu$-subkernels of $\langle \mathscr{R} \rangle$ in
Theorem~\ref{cor0}.

The road to Theorem~\ref{cor0} is somewhat arduous. The interplay
of corner roots of polynomials and $\nu$-kernel roots comes in
Section \ref{supertrop16}. In applying the theory to sets of
polynomials defining tropical varieties, one quickly sees that
there still are too many $\nu$-kernels; for example, some of the
varieties defined by $\nu$-kernels do not satisfy basic tropical
conditions such as the balancing condition. Example~\ref{ci10}
also displays a certain form of pathology.  Thus, there is no way
to obtain such basic notions as dimension without specifying which
kinds of $\nu$-kernels to yield the tropical varieties arising
from hypersurfaces and their intersections. These $\nu$-kernels,
called \textbf{corner internal} $\nu$-kernels, lie at the heart of
our theory, being the ones that provide the transition between
ghost roots and kernel roots:

A rational function $f$ is  \textbf{corner internal} if $f = \frac
hg $ such that every ghost root of $g+h$ is a $\nu$-kernel root of
$f$. A \textbf{corner internal} $\nu$-kernel,
cf.~Definition~\ref{defn_corner_admissible}, is a $\nu$-kernel
generated by a corner internal rational function.
\medskip

There are two main ways of passing from hypersurfaces to corner
internal rational function, one given in the basic ``hat
construction'' of \S\ref{hat1}, leading to Theorem~\ref{cor11},
stating  that the correspondences $f \mapsto \hat f$ and $h
\mapsto \underline{h}$ induce  1:1 correspondences between
hypersurfaces and $\onenu$-sets of corner internal rational
functions.

 Another, subtler way, of describing corner internal
rational functions is given in Definition~\ref{hat2} and
Theorem~\ref{cor_CI_map1}. In \S\ref{sym} we specify the
\textbf{regular}   rational functions, which locally are not the
identity, and thus have tropical significance, distinguishing
between two general types of nontrivial principal $\onenu$-sets in
$F^{(n)} $:
\begin{itemize}\item $\onenu$-sets not containing a region of
dimension $n$; \item $\onenu$-sets that do contain a region of
dimension $n$.\end{itemize}

The $\nu$-kernels defining ``traditional'' tropical hypersurfaces
are the regular corner internal $\nu$-kernels.

 We further narrow down the $\nu$-kernels of interest, finally showing how one obtains a strong Zariski correspondence
  using
 regular principal $\nu$-kernels. 
There is a basic question as to over which \semifields0   we
should define our $\nu$-kernels. Although most of the results are
for $\nu$-kernels of $F(\lambda_1, \dots, \lambda _n)$, the
rational fractions of the polynomial \semiring0 over an arbitrary
supertropical \semifield0 $F$, in \S\ref{compl1} we add the
condition that the underlying \semifield0 $F$ be complete.  At
times we need to assume that $F$ is complete, archimedean, and
divisible, in order to obtain principal $\nu$-kernels.

Most of \S\ref{subsection:bounded_kernel0} involves intersecting
down to $\langle F \rangle$. In Theorem~\ref{thmcireglattice} we
formulate the major result:

\medskip

All (regular) principal $\nu$-kernels  of $\langle F \rangle$ are
generated by (regular) corner internal principal $\nu$-kernels.

\medskip

See Example~\ref{tropline1} as to how this result applies to the
familiar tropical line. Its proof requires more machinery, and
also requires us to take $F$ to be complete under the order
topology, say $F = \mathscr{R}$, the supertropical version of the
real numbers (viewed as a max-plus structure).

 The \textbf{wedge decomposition} is a tool introduced
in Definition~\ref{defn_decomposition}, and in
Theorem~\ref{thm_every_gen_is_reducible} it is seen that an
intersection of principal $\nu$-subkernels of $\langle \mathscr{R}
\rangle$ has a corresponding wedge decomposition.

But we still need more preparation to prove
Theorem~\ref{thmcireglattice}. Our classes of $\nu$-kernels are
further subdivided in~\S\ref{HOdecomp} into ``HO-kernels'' and
another kind of $\nu$-kernel which can be bypassed when
restricting from $\overline{F(\Lambda)}$ to~$\langle F \rangle$.
HO-kernels are decomposed as products of ``HS-kernels'' and
``region $\nu$-kernels'' in Theorem~\ref{thm_HP_expansion}:

 \medskip

 Every
principal $\nu$-kernel $\langle f \rangle$ of $F(\la_1,...,\la_n)$
can be written as the intersection of finitely many principal
$\nu$-kernels
$$\{K_i : i=1,...,q\} \ \text{and} \ \{N_j : j = 1,...,m \},$$
whereas each $K_i$ is the product of an HS-kernel and a region
$\nu$-kernel, and each $N_j$ is a product of ``bounded from below''
$\nu$-kernels and (complementary) region $\nu$-kernels.

This enables us to reduce to finite decompositions and finally
prove Theorem~\ref{thmcireglattice}.

%
%
%
%

 As a major application, the decomposition theory we have just described also
 gives rise 
 to an algebraic description of basic geometric tropical concepts such
as dimension, which is seen 
to match the intuitive geometric notion. Towards this end, we define
a \textbf{hyperplane $\nu$-kernel}, or  \textbf{HP-kernel}, for
short, to be a principal $\nu$-kernel of $F(\la_1,...,\la_n)$
generated by a single Laurent monomial. More generally, a
\textbf{hyperspace-fraction $\nu$-kernel}, or \textbf{HS-kernel}, is
a principal $\nu$-kernel of $F(\la_1,...,\la_n)$ generated by a
rational function  $f \sim_{\FF} \sum_{i=1}^{t} |f_i|$
(cf.~Definition~\ref{defn_abstract_similarity_of_generators}) where
the $f_i $ are non-proportional non-constant Laurent monomials, and
corresponds to a hyperplane in polyhedral geometry.

On the other hand, we also encounter \textbf{order $\nu$-kernels}.
Imposition of an order $\nu$-kernel intersects our variety with a
half-space, so does not reduce the dimension. Furthermore, one
could have an infinite chain of ``parallel'' order $\nu$-kernels.
For these reasons,
our theory of dimension bypasses order $\nu$-kernels and turns to HS-kernels. 

The \textbf{height} $\hgt(L) $  of an HS-kernel $L$ is given
in~Definition~\ref{height0}, and the \textbf{convex  dimension} then
is given in~Definition~\ref{condim}.
 This dimension is catenary, in the sense of Theorem~\ref{caten}:

 \medskip

If $R$ is a region $\nu$-kernel and $L$ is an HS-kernel  of
$\overline{F(\Lambda)}$, then
$$\condeg(\overline{F(\Lambda)}/LR) = \condeg(\overline{F(\Lambda)}) - \condeg(L).$$

 \medskip

In particular, $\condeg(\overline{F(\Lambda)}/LR) = n -
\condeg(L).$ There is a well-known   general argument for
Jordan-H\"{o}lder type theorems for finite chains in modular
lattices, recalled in Theorem~\ref{Schreier},
 leading us to
Theorem~\ref{prop_kernel_descending_chain_properties}:

 \medskip

If a $\nu$-kernel $L $ is irreducible  in the lattice of
$\nu$-kernels finitely generated by HP-kernels of
$\overline{F(\Lambda)}$, then $\hgt(L) $ is its convex  dimension.
 Moreover, every factor of a  descending chain of HS-kernels of maximal length is an  HP-kernel.
 \medskip

Finally, let us turn to the Nullstellensatz. Several versions  for a
tropical Hilbert Nullstellensatz have been put forth, including
\cite{BeE,JoM}. Our objective here rather is to understand the
tropical spirit of the Nullstellensatz in a more general context.
 Just as the classical Hilbert
Nullstellensatz shows that affine varieties correspond to radical
ideals of the polynomial algebra, here affine tropical varieties
correspond to certain $\nu$-kernels    called \textbf{polars},
defined intrinsically in terms of the natural orthogonality relation
that $f\perp g $ for $f,g \in F(X)$ if their minimum always takes on
the value 1.

\medskip

 Polars are best  understood as $\nu$-kernels   in   the
completion of the natural image of $F(\lambda_1, \dots, \lambda _n)$
inside the \semiring0 of functions from some subset $S\subseteq
F^{(n)}$ to $F$, but we can take the restriction back to $F(\la_1,
\dots, \la_n)$. Unfortunately, finitely generated polars need not be
finitely generated as $\nu$-kernels, which is another reason for us
to restrict to bounded $\nu$-kernels.

We discuss some of their basic properties in \S\ref{Polar}. By
Theorem~\ref{prop_polar_k_kernel_completion_equivalence}, the polars
of $\overline{\mathscr{R}(\Lambda)}$ are precisely those that come
from $\onenu$-sets, thereby leading up to the fundamental
correspondence of Theorem~\ref{thm_polar}:

\medskip

 There is a $1:1$ correspondence from the polars of $\overline{\mathscr{R}(\Lambda)}$ to the
$\onenu$-sets in $\mathscr{R}^{n}$, given by $B \mapsto \Skel(B)$
and $Z \mapsto \Ker(Z)$, which  restricts to a correspondence
between the principal polars (as polars) of
$\overline{\mathscr{R}(\Lambda)}$ and the principal $\onenu$-sets
in~$\mathscr{R}^{n}$.

In this way we reduce $\nu$-kernels to finitely generated
$\nu$-kernels, which then are principal, and we obtain the
appropriate Zariski correspondence in its entirety.

\section{Algebraic Background}\label{back}

We start by reviewing some familiar notions that are needed
extensively in our exposition.
%
%
%
%
%
%

\subsection{Semi-lattice ordered monoids, groups, and semirings} $
$

 The passage to the max-plus algebra in tropical mathematics
can be viewed via ordered groups and, more generally, ordered
monoids and semirings, so we start with them, drawing on the
review given in
\cite{IzhakianKnebuschRowen2012Algebraicstructures}. This material
is well known, and largely can be found in \cite{OrderedGroups3}.
Recall that a monoid $(M,\cdot, 1)$ is just a semigroup with an
identity, i.e., unit element, denoted as $1$. We work solely with
\textbf{Abelian} monoids, in which the operation is commutative.
Our semigroups without identity are written in additive notation,
and monoids are written in multiplicative notation.

\begin{defn}\label{semilat} A (set-theoretic) \textbf{semi-lattice} $L$ is a partially ordered set with
an (associative) binary ``sup'' operation $\vee$, which means:
\begin{equation} a ,b \le a \vee b, \quad\text{and} \quad  \text{if } a,b \le c \text{ then }
 a \vee b \le  c.\end{equation}

\end{defn}
\begin{defn}
An element $a$ of  a semigroup $(\tSS,+)$ is \textbf{idempotent} if
$a+a = a$. A \textbf{band} is a semigroup in which
 every element is idempotent. A band $(\tSS,+)$ is \textbf{bipotent} if
 $a+b \in \{a,b\}$ for each pair of elements.
 \end{defn}

\begin{lem}\label{band0} Any
 semi-lattice can be viewed as a  band, where we define $a+b = a
 \vee b.$
\end{lem}
\begin{proof} To check associativity of addition we have
$(a+b)+c = \sup \{ a,b,c \} = a+ (b+c).$ Furthermore $a+a = a \vee
a = a.$
\end{proof}

\begin{defn}
 A \textbf{semi-lattice ordered monoid}
is a monoid $\ptM$ that is also a  semi-lattice with respect to
the operation
 $\vee$ and
satisfies the following property:

\begin{equation}\label{ogr} g(a \vee b) =  ga
\vee  gb  \end{equation} for all elements $a,b,g \in \ptM$.

 A \textbf{partially ordered monoid}
is a monoid $\ptM$ with a partial order satisfying
\begin{equation}\label{ogr1} a \le b \quad \text{implies}\quad ga
\le gb,\end{equation} for all elements $a,b,g \in \ptM$.

An \textbf{ordered monoid}  is a partially ordered monoid for
which the given partial order is a total order.

 A \textbf{lattice ordered monoid}
is a monoid $\ptM$ that is also a  lattice with respect to the
operations
 $\vee$ and $\wedge$ satisfying \eqref{ogr} for both $\vee$ and
 $\wedge.$
\end{defn}

 It follows
  readily  that if $a \le b $ and $g \le h,$ then $ga \le gb
\le hb.$ Note that the group $(\Real,+\;)$ is ordered, under this
definition, but $(\Real,\cdot \; )$ is not.

Every semi-lattice ordered monoid is  partially ordered with
respect to the   partial order given by $a\le b$ iff $a\vee b =
b.$

In case a semi-lattice ordered monoid $\ptM$ is a group $\tG$, one
defines the \textbf{dual semi-lattice} $(\tG,\wedge)$ via
\begin{equation}\label{dua} a\wedge b = (a^{-1} \vee
b^{-1})^{-1}.\end{equation} Now $\tG$ becomes a lattice, and we call
$(\tG, \vee, \wedge)$ a \textbf{lattice ordered group} if $\tG$ is a
lattice ordered monoid.

\begin{lem}\label{onesuf} Suppose   $(\tG, \vee)$ is a semi-lattice ordered Abelian group with respect to a given partial order $\le.$
Then $(\tG,\wedge)$ is indeed a semi-lattice, with respect to
which $\tG$ also becomes a lattice ordered group satisfying
$(a\vee b)^{-1} = a^{-1}\wedge b^{-1}.$\end{lem}
\begin{proof} $a \le b$ iff $b^{-1} = (ab)^{-1}a \le (ab)^{-1}b = a^{-1},$
so we get the dual partial order $\ge$. Hence $(\tG,\wedge)$ is a
semi-lattice with respect to $\ge$, since $$a\wedge b  = (a^{-1}
\vee b^{-1})^{-1} = \sup \left\{ a^{-1}, b^{-1} \right\} ^{-1} =
\inf \left\{ (a^{-1})^{-1}, (b^{-1})^{-1} \right\} = \inf\{ a,b
\}.$$

Then $g(a\wedge b) = g(a^{-1} \vee b^{-1})^{-1} = (g^{-1}(a^{-1}
\vee b^{-1}))^{-1} = (g^{-1}a^{-1} \vee g^{-1} b^{-1})^{-1} = ga
\wedge gb,$ as desired.

The last assertion follows from \eqref{dua} taking $a^{-1},
b^{-1}$ instead of $a,b.$
\end{proof}

The duality in Lemma~\ref{onesuf} shows that it suffices to consider
only $\vee$ (or only $\wedge$). 
Every partially ordered group that is a semi-lattice is a
semi-lattice ordered group, seen by using $g^{-1}$ instead of $g$ in
\eqref{ogr}.

In any ordered group we define the open interval of distance $b$
around $a$ given by $\{ g: ab^{-1} < g < ab\},$ giving rise to  the
\textbf{order topology}.

Although idempotence pervades the theory, it turns out that what
is really crucial for many applications is the following
well-known fact:

  \begin{lem}\label{idpar} In any band, if $a+b +c= a$,
then $a+c = a.$ \end{lem}
\begin{proof} $ a = a+b+c = (a+b+c) +c = a+c$. \end{proof}

Let us call   a semigroup \textbf{proper} if it satisfies the
condition of Lemma~\ref{idpar}.

 \begin{rem}[{\cite[Theorem~4.3]{Cont_Semifields}}]\label{prop1} A proper semigroup cannot
have additively invertible elements other than~$\zero,$ since if
$a+c = \zero$, then $a = a+ \zero = a + a + c,$ implying $a = a+c
= \zero.$
\end{rem}

%
%

 A submonoid $\tT$ of an Abelian monoid $\ptM$ is called
\textbf{cancellative} if $a b = a c$ for $a \in \tT$ and $b,c \in
M$ implies $b=c.$ In this case, when $\tT = \ptM,$ we say that the
Abelian monoid $\ptM$ is \textbf{cancellative}.

A monoid $\ptM$ is \textbf{power-cancellative} (called
\textbf{torsion-free} in \cite{CHWW})
 if $a^n   =
b^n $ for some $n\in \Net$ implies $a=b.$ Any ordered, cancellative
monoid is power-cancellative and infinite.

\subsubsection{Divisibility}$ $

 We say that a monoid $\ptM$ is $\Net$-\textbf{divisible} (also called \textbf{radicalizible}
 in the tropical literature, but that terminology conflicts with  \cite{KA_Semifields} and
 radical theory) if for
each $a\in M$ and $m \in \Net$ there is $b\in M$ such that $b^m =
a.$ For example, $(\Q, +)$ is $\Net$-divisible. There is a standard
construction of the \textbf{divisible hull} of a cancellative monoid
$\ptM$, given in \cite{OrderedGroups3}, which is semi-lattice
ordered when $\ptM$ is semi-lattice ordered; namely,
\begin{equation}\label{div2} \root m \of a \vee \root n \of b = \root {mn} \of {a^n \vee
b^m}.\end{equation} This also was discussed
 in~\cite[Remark~2.3]{Layered}.

\subsection{Semirings without zero} $ $

 \begin{defn}\label{semir1} A
\textbf{\semiring0} (semiring without zero, at times called a
hemiring) is a set~$R:=(R,+,\cdot,1)$ equipped with
 binary operations $+$ and~$\cdot \;$ and distinguished element $\rone$ such that:
\begin{enumerate}\eroman
    \item $(R, + )$ is an Abelian semigroup;
    \item $(R, \cdot \ , \rone )$ is a monoid with identity element
    $\rone$;
    \item Multiplication distributes over addition.
      \item $R$ contains elements $r_0$ and $r_1$ with $r_0+r_1 = \rone.$

\end{enumerate}
\end{defn}


Note that (iv) is automatic if $(R, + )$ is a band, since then
$\rone+\rone = \rone.$ For the purposes of this paper, a
\textbf{\domain0} is a commutative \semiring0 whose multiplicative
monoid is cancellative.

\begin{defn} A
\textbf{\semifield0} is a \domain0 in which every element is
(multiplicatively) invertible.
\end{defn}

 (In other words, its multiplicative
monoid is an Abelian group. In \cite{Cont_Semifields} commutativity
is not assumed, but we make this assumption to avoid distraction
from our applications.) In particular, the max-plus algebras
$(\Net,+)$, $(\Q,+)$, and $(\Real,+)$ are \semifields0, whose
multiplication now is given by  $+$. We also have the
\textbf{trivial \semifield0} $\{ 1\}.$
 In the literature it is customary to write the
   operations as $\oplus$ and $\odot$,
   but we  use the usual algebraic notation of
   $+$ and $\cdot$ for   addition and
multiplication respectively, to emphasize the structural aspects of
the theory.

Any \semifield0 without negatives is proper, by \cite
[Proposition~20.37]{Hom_Semifields}. 

\medskip {\it Digression}: The customary definition of semiring
\cite{golan92} also requires the existence of a zero element:

 A \textbf{semiring} is a \semiring0 $R$
with a zero element  $\rzero $ satisfying
  $$a + \rzero = a, \quad  a \cdot \rzero
 =\rzero = \rzero \cdot a, \quad   \forall a\in R.$$
(Note that in the definition of \semiring0 one could then take $r_0
= \rzero$ and $r_1 = \rone$.) We use \semifields0 instead of
semifields since
  the zero element  usually can be adjoined formally, and often is
  irrelevant, and the concepts are easier to define when we do not
  need to
treat  the zero element separately. Also, the language of
\semifields0 is more
  appropriate to geometry over tori, which are direct products of groups.


A \semiring0 $R$ is \textbf{idempotent} if the semigroup $(R,+)$ is
a band. A \semiring0 is \textbf{bipotent} if  the semigroup $(R,+)$
is bipotent.
 Thus, the max-plus
algebra, viewed as a \semiring0, is {bipotent}.

Our basic structures are idempotent \semifields0, denoted as $F$
or $\mathbb S$ throughout. (Usually $F$ is the underlying
\semifield0, contained in $\mathbb S$.) Let us recall an idea of
Green from the theory of idempotent semigroups.

\begin{prop}[{\cite[\S4]{Semifields_LO_Groups}}]\label{bip}
 Any semi-lattice ordered Abelian monoid $\ptM $ becomes an idempotent (commutative) \semiring0, which we denote as $R$,  via the usual
max-plus procedure; we define the new multiplication on~$R$ to be
the operation given originally on $\ptM$, and addition on~$R$ is
defined as in Lemma~\ref{band0}:
\begin{equation}\label{add0} a+b : = a \vee b\end{equation}
(viewed in $\ptM$). Conversely, any idempotent commutative
\semiring0 $R$ becomes a semi-lattice ordered Abelian monoid by
reading \eqref{add0} backwards.

\end{prop}
\begin{proof}  Distributivity follows
from \eqref{ogr}. Furthermore $a+a = a \vee a = a.$

Conversely,  if $a \le b$ and $b \le a$ we have    $a = a+b=b$, so
$\le$ is antisymmetric; we need to show that $a+b$ is the sup of
$a$ and $b$. In other words, if $a \le c$  and $b \le c$ then $a+b
\le c.$ But we are given $a+c = c$ and $b+c = c,$ so $$a+b + c =
(a+c) + (b+c) = c+ c = c,$$ as desired.
 \end{proof}

  When $\ptM$ is cancellative, then $R$ is a \domain0. When $\ptM$ is a group, then $R$ is a \semifield0.
(This could all be formulated categorically.)
%

\begin{lem}[{\cite[Property 2.6]{Prop_Semifields}}]\label{torfree} Every \domain0 is
torsion-free as a monoid.\end{lem}
\begin{proof} If $a^n = 1,$ then $a(a^{n-1} +a^{n-2} + \dots + 1) = (1 +a^{n-1} +a^{n-2} + \dots + a),$ implying $ a = 1.$
\end{proof}

\begin{defn}
 A \semiring0
$R$ is \textbf{ordered} if $(R,+)$ and $(R,\cdot)$ are both
ordered monoids.
\end{defn}

\begin{rem}\label{rem_torsion_free}
The multiplicative group of every ordered \semifield0 is a
torsion-free group, i.e., all of its elements not equal to $1$ have
infinite order.
\end{rem}
%

Here is another important property:

\begin{defn} A \semiring0 $R$ is \textbf{Frobenius} if it satisfies the identity
$$(a+b)^n = a^n + b^n, \qquad \forall a,b \in R.$$
\end{defn}

It is clear that any bipotent \semiring0 is Frobenius.

\begin{defn}\label{defn_cond_complete}
A lattice $P$ is said to be \textbf{complete} if all of its
bounded subsets have both a supremum and an infimum. A sublattice
of $P$ is said to be \textbf{completely closed} if all of its
subsets have both a supremum and an infimum in $P$.

 An idempotent
\semiring0  is said to be \textbf{complete} if its underlying
lattice (with $\vee$ as $\dotplusss $ and $\wedge$ as $\cdot$) is
complete.
\end{defn}

\subsection{Localization}$ $

Since we are mainly interested in (proper) idempotent
\semifields0, we need a method of passing from \semirings0 to
\semifields0.

\begin{rem}\label{loc0}
There is a well-known localization procedure with respect to
multiplicative subsets of Abelian monoids, described in detail  in
\cite{B}. Namely, for any  submonoid $S$ of a monoid $\ptM$, we
define an equivalence on $S \times \ptM$ by saying $(s_1, a_1) \sim
(s_2,a_2)$ iff there is $s\in S$ such that $ss_2 a_1 = s s_1 a_2;$
then the \textbf{localization} $S^{-1}\ptM $  is the set of
equivalence classes $\{ [ (s,a)]: s \in S, a \in \ptM\}$, written as
$\frac as.$  $S^{-1}\ptM $  is a monoid via the operation
$$ \frac as \frac {a'}{s'} = \frac {aa'}{ss'}.$$

There is a homomorphism $\ptM \mapsto S^{-1}\ptM $ given by $a
\mapsto \frac a1.$ This map is 1:1 precisely when $a \ne b$
implies $\frac a1 \ne \frac b1;$ in other words, when the
submonoid $S$ of $\ptM$ is cancellative.

 If
the monoid $\ptM$ itself is cancellative, then localizing with
respect to all of $\ptM$ yields its \textbf{group of fractions}.
In this case $\frac { a_1}{s_1} = \frac { a_2}{s_2}$ iff $s_2 a_1
= s_1 a_2.$

If the monoid $\ptM$ is ordered, then $S^{-1}\ptM $ is also ordered,
by putting $ \frac as \le \frac {a'}{s'}$ iff $as's \le a'ss$ for
some $s\in S.$

When $R$ is a commutative \semiring0 and $S$ is a submonoid of
$(R,\cdot),$ we endow $S^{-1}R $ with addition given by $$ \frac as
+ \frac {a'}{s'} = \frac {as' +a's}{ss'}.$$ When $(R,  \cdot)$ is
ordered, this is compatible with the monoid order on $S^{-1}R $, in
the sense of Proposition~\ref{bip}.

Clearly, the localization of an idempotent (resp.~proper,
resp.~Frobenius) \semiring0 is idempotent (resp.~proper,
resp.~Frobenius).

Furthermore, if $R$ is a \domain0, and $S = R,$ then    we call
$S^{-1}R$ the \textbf{\semifield0 of fractions} of $R$.
\end{rem}

%
%
%

%

\begin{rem} These considerations are compatible with localization and the divisible hull. For example,
one can define $\vee$ on the monoid $S^{-1} \ptM$ via
$$ \frac as \vee \frac bs = \frac{a\vee b} s.$$
For the divisible hull, we would use the analog of
Equation~\ref{div2}.\end{rem}

\begin{prop} Suppose $R,R'$ are \domains0, with a \semiring0 homomorphism $\varphi: R
\to R',$ and suppose $S$ is a submonoid of $R$. Then there is a
homomorphism $ \tilde \varphi :    S^{-1}R \to \tilde \varphi
(S)^{-1}R',$ given by $ \tilde \varphi ( \frac as) =   \frac
{\varphi (a)}{\varphi (s)}.$
\end{prop}
\begin{proof} This is standard. First we need to show that $\tilde \varphi$ is well-defined:
 If $\frac {a_1}{s_1} = \frac {a_2}{s_2}$ then $s  s_2 a_1 = ss _1
 a_2$ so $\varphi(s )\varphi( s_2 )\varphi( a_1 ) = \varphi( s)\varphi(s _1)\varphi(
 a_2)$, implying $\varphi( s_2 )\varphi( a_1 ) =  \varphi(s _1)\varphi(
 a_2)$, and thus $\tilde \varphi(\frac {a_1}{s_1}) = \tilde \varphi(\frac {a_2}{s_2}).$

Now $ \tilde \varphi \left(\frac {a s_1}{ s s_1}\right) =
\left(\frac {a_1}{s_1} \frac {a_2}{s_2} \right) =   \tilde
\varphi\left(\frac {a_1 a_2}{s_1 s_2} \right) = \frac
{\varphi(a_1a_2)}{\varphi(s_1s_2)} =  \tilde \varphi \left(\frac
{a_1}{s_1}\right) \tilde \varphi \left(\frac {a_2}{s_2}\right).$
\end{proof}
%

\subsection{Congruences}$ $

 A~\textbf{congruence} $\Cong $ of
an algebraic structure $A$ is an equivalence relation $\equiv$
that preserves
 all the relevant operations and relations;  we call  $\equiv$ the \textbf{underlying  equivalence} of $\Cong
 $. One can view a congruence more formally as a subalgebra of $A\times A.$
For ease of notation we write $a \equiv b$ when $(a,b) \in \Cong.$
 In the case of \semirings0,
 $\Cong$ being a congruence means that $a_i \equiv b_i$ for $i=1,2$ implies $a_1 + a_2 \equiv b_1 + b_2$ and $a_1 a_2 \equiv b_1  b_2$.

\begin{rem}\label{cong10} We recall some key
results of~\cite[\S 2]{J}:

 \begin{enumerate} \eroman  \item Given a congruence  $\Cong$ of an algebraic
 structure $R$, one can endow the set $$A/ \Cong := \{ [a]: a \in A \}$$ of equivalence
 classes  with the same (well-defined) algebraic structure, and the map
 $a \mapsto [a]$ defines an onto homomorphism $A\to A/\Cong $.

\item  In the opposite direction, for any homomorphism $\vrp:A \to
A',$  one can define a congruence $\Cong_\vrp $ on $A$ by saying
that $$a \equiv_\vrp b \dss{\text{ iff }} \vrp(a) = \vrp (b).$$ We
call $\Cong_\vrp $ the \textbf{congruence of} $\vrp.$ Then $\vrp$
induces a 1:1 homomorphism $\tlvrp :A/\Cong_\vrp  \ \to A',$ via
$\tlvrp ([a]) = \vrp(a)$, for which $\vrp$ factors through
$$A \to A/\Cong _\vrp \to A',$$ as indicated in
\cite[p.~62]{J}. Thus the homomorphic images of $A$ correspond to
the congruences defined on~$A$.
\end{enumerate}
 \end{rem}

\begin{defn}\label{def:primeCOng}  A congruence $\Cong$ on an Abelian monoid $\ptM$ is \textbf{cancellative}
 when the monoid
$\ptM/\Cong$ is cancellative.  The congruence $\Cong$ on an
Abelian monoid $\ptM$ is \textbf{power-cancellative} when the
monoid $\ptM/\Cong$ power-cancellative, i.e., if
 $a_1 ^k  \equiv  a_2^k$ for some $k \ge 1$ implies $a_1 \equiv
a_2$. \end{defn}

 For the moment, let $\mathcal C (\ptM)$ denote a given set of cancellative
 and  power-cancellative congruences on $\ptM$. A congruence  $\Cong \in \mathcal C (\ptM)$
 is \textbf{irreducible} (called \textbf{intersection indecomposable} in \cite{JoM}) if it is not the proper intersection of two
 congruences in $\mathcal C (\ptM)$.

%
%
%
%

 \begin{rem}\label{primecon} \textbf{Prime congruences} are defined by Joo and Mincheva
 in \cite{JoM},
 and are seen \cite[Theorem~2.11]{JoM} to be precisely the
 irreducible cancellative congruences.  \cite{JoM} goes on to
 classify the prime congruences of the polynomial semirings over
 the max-plus semifields  $\mathbb R$ and $\mathbb Z$, and also
 over the 2-element semifield.
 \end{rem}

\begin{lem}\label{ext}
 Any congruence $\Cong$ on a \semiring0 $R$ extends to its localization $S^{-1}R$ by a cancellative submonoid $S$, by putting $\frac a s \equiv \frac {a'}{s'}$
 when $as' \equiv a's$.
 \end{lem}
\begin{proof} It is standard and easy that $\equiv$ extends to the given
equivalence on $S^{-1}R$, so we need check merely that the given
operations are preserved: If $\frac {a_1}{ s_1} \equiv \frac
{a_1'}{s_1'}$ and $\frac {a_2}{ s_2} \equiv \frac {a_2'}{s_2'}$,
then $a_i s_i' \equiv a_i' s_i$ for $i = 1,2$, from which it
follows that ${a_1 s_1' }a_2 s_2' \equiv a_1' s_1 a_2' s_2$ and
thus
$$\frac {a_1}{ s_1}  \frac {a_2}{ s_2} =
 \frac {a_1 a_2}{s_1 s_2} \equiv \frac {a_1' a_2'}{s_1' s_2'}$$

Likewise,
$$\frac {a_1}{ s_1} + \frac {a_2}{ s_2} =
\frac{a_1 s_2 + a_2 s_1}{s_1s_2} \equiv \frac{a_1' s_2' + a_2'
s_1'}{s_1's_2'}=\frac {a_1'}{ s_1'} + \frac {a_2'}{ s_2'}.$$
\end{proof}

 \subsection{The function monoid and \semiring0}$ $

 The supertropical structure   permits us to detect
 corner roots of tropical polynomials
 in terms of the algebraic structure, by means of ghosts.
 To see this most clearly, we introduce a key structure taken from universal algebra.
%
%

 \begin{defn}\label{funcsem} The \textbf{function monoid}   $\FunStM$ is the set of
functions from a set $S$ to a monoid~$\ptM$.

The \textbf{function \semiring0}   $\FunSR$ is the set of functions
from a set $S$ to a \semiring0~$R$.
\end{defn}

When $\ptM$ is a monoid, $\FunStM$ becomes a monoid under the
componentwise operation, $fg(\bfa) = f(\bfa)g(\bfa)$, and is
cancellative when $\ptM$ is cancellative. If $\ptM$ is
semi-lattice ordered, then so is $\FunStM$, where we define $(f
\vee g )(\bfa) = f(\bfa) \vee g(\bfa)= f(\bfa)+g(\bfa).$

 When $R$ is a \semiring0, $\FunSR$ becomes a \semiring0 under
componentwise operations, and is idempotent (resp.~proper) when
$R$ is idempotent (resp.~proper). When $R$ is a \semifield0,
$\FunSR$ becomes a \semifield0.

 Note that the \semiring0 $\Fun (\tS,R)$ need  not be
 bipotent even if $R$ is bipotent, since some of the evaluations  of $f+g$ might come
from $f$ and others from $g$.
 As in Proposition~\ref{bip}, we identify the \semiring0 structure of $\Fun (\tS,R)$
 with the semilattice operation $\label{dua1} f\vee g = f+g$. When
 $R$ is a \semifield0 $F$,
we can also define
\begin{equation}
\label{dua2} f\wedge g= (f^{-1} + g^{-1})^{-1}.\end{equation}

\begin{rem}\label{ord1} Viewing $F$ as an ordered group, one sees easily that $ (f\wedge
g )(\bfa) = \min \{ f(\bfa), g(\bfa) \},$ and this is how we think
of it, although \eqref{dua2} is easier to handle formally.
\end{rem}

 Often  $\tS$ is taken
to be $F^{(n)},$  the Cartesian product of $n$ copies of $F$. But
using $\tS \subseteq F^{(n)}$    enables one to lay the foundations
of algebraic geometry.

For any subset $\tS' \subseteq \tS,$ there is a natural homomorphism
$$ \FunSF \to \Fun (\tS',F),$$ given by restriction of
functions.


 \subsection{Archimedean monoids}$ $

\begin{defn}
A  partially ordered monoid   $(G, \cdot)$ is called
\textbf{archimedean} if $a^{\mathbb{Z}} \leq b$ implies that $a =
1$.

 A \semiring0 $(R, + , \cdot)$ is \textbf{archimedean}
if it is archimedean as a  partially ordered monoid.
\end{defn}

This reduces to the usual definition when $G$ is totally ordered,
but we need this more general  version, because of the
next example.

\begin{exampl}
The \semifield0 $\Net$ is ordered and archimedean, as are
 $(\Q,+)$ and $(\R,+)$.
 \end{exampl}

\begin{rem} If $R$ is archimedean then $\FunSR$ also is
 archimedean.\end{rem}

\section{The tropical environment}\label{supertrop}

We would like to fit this all  into tropical mathematics,  mainly
for affine varieties defined by a finite number of polynomials.
Tropical mathematics involves using the Puiseux valuation to pass
from the field of Puiseux series to its value group $(\mathbb Q,
+)$ (or $(\mathbb R, +)$), which in turn is viewed as the max-plus
algebra. But this transition does not reflect the valuation theory
for Puiseux series having lowest order terms of the same degree;
their sums do not necessarily have lowest order terms of this
degree. Thus we should consider alternative structures in which $a
\vee a \ne a$ -- i.e., monoids which are ordered by sets which are
not quite semi-lattices.

 \subsection{The standard supertropical \semifield0}
$ $

We are ready to bring in the algebraic structure that reflects the
properties of the Puiseux valuation.

\begin{defn} A \textbf{\vdomain0} is a quadruple
$(R, \tT, \nu, \tG)$ where  $R$ is a \semiring0 and $\tT \subset R$
is a cancellative multiplicative submonoid  and $\tG \triangleleft
R$ is endowed with a partial order, together with an idempotent
homomorphism $\nu : R \to \tG$, with $\nu |_\tT$ onto, satisfying
the conditions:
$$  a+b =  a  \quad \text{whenever} \quad \nu (a) > \nu (b). $$
$$ a+b = \nu (a) \quad \text{whenever} \quad \nu (a) = \nu (b). $$
 $\tT$ is called  the \textbf{tangible
submonoid} of $R$.    $ \tG$ is called the \textbf{ghost ideal}.
%
\end{defn}

(In practice, $\nu$ is induced from an m-valuation of a field, for
example the  Puiseux valuation.)
%
The  condition $\nu |_\tT$ onto is introduced to enable us to move
back and forth between tangible and ghost elements.

\begin{rem}\label{notq}
 This definition is quite close to that of
\cite{Layered}, where $\tT$ would be the tangible elements and $\tG$
the ghost ideal $(1+1) R$. The main differences are as follows:

\begin{itemize}
\item Here the target $\tG$ need not be ordered, so that we can
handle semirings of polynomials directly.
%

\item  We do not require a zero element.
\end{itemize}
\end{rem}

We write $a^\nu$ for $\nu(a),$ for $a\in R.$
We write $a \nucong
  b$ if $a^\nu = b^\nu,$ and say that $a$ and $b$ are
  $\nu$-\textbf{equivalent}. Likewise we write $a \ge_\nu b$ (resp.~ $a >_\nu
  b$) if $a^\nu \ge  b^\nu$ (resp.~ $a^\nu >  b^\nu$).

\begin{rem}
Any   \semiring0 homomorphism $\varphi: R \to R' =(R',\tT',
\tG',\nu')$ of \vdomains0  satisfies $$\varphi (a^\nu) = \varphi
(a+a) =  \varphi (a)+\varphi(a) = \varphi(a)^{\nu'}$$ for all $a$ in
$R$.
\end{rem}

 Strictly speaking, the  supertropical \vdomains0\  can never   be \semifields0\ since the ghost elements are not
 invertible. Accordingly, we frame the next definition:

 \begin{defn}\label{vsemi2}
 A \vdomain0 $R$ is \textbf{$\nu$-bipotent} if $a+b \in \{a,b, a^\nu\}$ for all $a,b \in R.$
 A \textbf{$\nu$-\semifield0} $\mathbb S$
 is a
  \vdomain0 for which the tangible
submonoid $\tT$ is an Abelian group.
\end{defn}
\begin{exampl}\label{degen} In the  case where $F = \tT = \tG$,
with $\nu = 1_F$,  we are back to the \semifield0 theory. We call
this the \textbf{degenerate case} Although we are not interested
in this case, it shows us how to generalize to the $\nu$-theory.
Namely, we  must ``lift'' the usual theory from $\tG$ to $F$.
\end{exampl}

\begin{exampl}\label{mathscrR}
Any supertropical \semifield0 in the sense of \cite{SuperTropicalAlg} is a \vsemifield0, in the sense of this paper.

The most important example, which we denote as $\mathscr{R},$ is
defined by taking the tangible submonoid to be $(\mathbb R,+)$,
together with another ghost copy.
\end{exampl}

Here is our main example, but we will need our more general
definition when we consider functions, and in particular,
polynomial functions. 

\begin{exampl}\label{stropi} Given a group $\ptM$ and an ordered group $\tG$
with a group isomorphism $\nu: \ptM \to \tG,$ we write $a ^\nu$
for $\nu(a).$ The \textbf{standard supertropical} monoid $R$ is
the disjoint union $ \tT \cup \tG$ where $\tT$ is taken to be
$\ptM$, made into a group by starting with the given products on
$\ptM$ and $\tG$, and defining $ a b^\nu $ and $ a^\nu b $ to be $
(ab)^\nu$ for $a,b \in \ptM$.

We extend $v$ to the \textbf{ghost map} $\nu: R\to \tG$ by taking
$\nu|_\ptM = v$ and  $\nu_\tG$ to be the identity on $ \tG$. 

We make $R$ into a \semiring0, called the \textbf{standard
supertropical \semifield0}, by defining $$a+b =
\begin{cases} a \text{ for } a >_\nu b;\\b \text{ for } a
<_\nu b;\\a^\nu \text{ for } a\nucong b.
\end{cases}$$

In order for Condition (iv) of Definition~\ref{semir1} to hold, we
need the valuation $\nu$ to be nontrivial, in the sense that there
exists $a \in R$ such that $ a <_\nu  \one$.

$R$ is $\nu$-bipotent. $R$ is never additively cancellative, since
$$a + a^\nu = a^\nu = a^\nu + a^\nu.$$
%

\end{exampl}

The standard supertropical \semifield0 $R$  is a cover of the
max-plus algebra of $\tG$,
  in which we can ``resolve'' additive
 idempotents, in the sense that $a+a = a^\nu$ instead of $a+a =
 a.$

Although we are interested in $\ptM = \tT,$ the ordered group
 $\tG$    itself can be viewed as an idempotent
\semiring0 with unit element $\rone^\nu$ (instead of $\rone$).
Thus, we obtain $\nu$-versions of the previous notions by lifting
from $\tG$.

 The ghost ideal $\tG$ is  to be treated much the same
way that one  treats the 0 element in commutative algebra. The
standard supertropical
 \semifield0 works well with linear algebra, as seen for example
  in \cite{IzhakianRowenMatrix}, providing many of the analogs to
 the classical Cayley-Hamilton-Frobenius theory, but our interest
 here will be in its geometric significance.

\subsubsection{$\nu$-Localization}$ $

 If $R=(R,\tT, \tG,\nu)$ is a  \vdomain0, then we call $F: = \tT ^{-1}R $ the
\textbf{$\nu$-\semifield0 of fractions} $\Frac _\nu R$ of $R$.

\begin{lem}\label{vsemi20}
  $\Frac_\nu R$
 is a \vsemifield0 in the obvious way, where $\tT^{-1}\tT$
is the group of tangible elements.
\end{lem}
\begin{proof} Define $\nu( \frac rs) = \frac{r^\nu}s.$
\end{proof}

%
%


\subsection{The (*)-operation}$ $

Strictly speaking, since a $\nu$-\semifield0  is not a \semifield0,
we need to modify our definition of inverse. Considering $\tG$ as a
group in its own right, with unit element $\onenu,$ we define $a^*$
to be the inverse of $a$ in the appropriate group to which it
belongs ($\tT$ or $\tG$). Specifically, we have:

\begin{rem}\label{aut1}  Any  supertropical \semifield0 $(F,\tT,
\tG,\nu)$  has a multiplicative monoid automorphism $(*)$ of order
2 given by
$$a^* =   a^{-1}, \qquad (a^\nu)^*  = (a^{-1})^\nu, \qquad a \in \tT.$$

 But $(*)$
reverses the partial order induced by $\nu$, and thus does not
preserve addition.\end{rem}

\begin{rem}\label{wedge1}
There is   ambiguity in defining $a \wedge b$ in a supertropical
\semifield0, since we can only compare $\nu$-values. So we define
\begin{equation}\label{starry}a\wedge b = (a^*+b^*)^*.\end{equation}
But now $a\wedge a = a^\nu$.
\end{rem}

\subsubsection{The $\nu$-norm}$ $

 Since the
element $\one_F$ plays an important role, the following notion will
be useful.

\begin{defn} Suppose $F= (F,\tT, \tG,\nu)$ is a  supertropical \semifield0.
 The $\nu$-\textbf{norm} $|a|$ of an element
$a \in F$ is $a + a^*.$ \end{defn}

In the    degenerate case of Example~\ref{degen}, $|a|=a + a^{-1}.$
 In the max-plus \semifield0\ this is the parallel of the
usual absolute value of $a$.
\begin{rem}\label{norm1}  (i) $|a|^2 =  \one^\nu  + |a^2|\ge _\nu \one$.

(ii) $|a| \nucong \one$  iff  $ a \nucong \one $. \end{rem}

 \subsection{The function  \vdomain0}$ $

 This supertropical structure also permits us to detect
 corner roots of tropical polynomials
 in terms of the algebraic structure by means of ghosts, once we provide a $\nu$-structure for
  the function \semiring0 of Definition~\ref{funcsem}.

\begin{exampl}\label{vloc31}
If $(R,\tT, \tG,\nu)$ is a \vdomain0, then $\Fun (\tS,R) $ is a
\vdomain0   has the induced map $\tilde \nu: \Fun (\tS,R)\to \Fun
(\tS,\tG)$ given by
$$\tilde  \nu(f)(\bfa) =  \nu(f(\bfa)).$$
Here the ghosts are $\Fun (\tS,\tG)$.


 We also extend the $(*)$-map to $\Fun (\tS,F)$:

  For $F$ a
supertropical \semifield0 and $f \in \Fun (\tS,F)$, define $f^*$
via $f^*(a) = f(a)^*$ for all $a\in F.$\end{exampl}

By Remark~\ref{aut1}, the map $f \mapsto f^*$ defines a monoid
automorphism of $\Fun (\tS,F)$ of order 2. In this way, we have the
following lattice operations on $\Fun (\tS,F)$:

\begin{equation}\label{dua01} f\vee g = f+g ; \qquad f\wedge g= (f^* + g^*)^*.\end{equation}

\begin{defn}\label{pos}  For any  semitropical
\semifield0 $F = (F, \tT, \tG, \nu)$ and any subset $A \subseteq
\FunSF$, $A^{+}$ denotes $ \{ f \in A : f >_\nu 1 \}$.

\end{defn}

\begin{rem}   If $R$ is a $\nu$-\domain0, then so is $R^+$.
\end{rem}

\subsubsection{Polynomials and Laurent polynomials}$ $

%
%

$\Lambda = \{ \la_1, \dots, \la _n \}$ always denotes a finite
commuting set of indeterminates; often $n=1$, and we have a single
indeterminate $\la$.

 Given a
 \semiring0 $R$, we have the polynomial \semiring0 $R[\Lambda]$.
Just as in \cite{SuperTropicalAlg}, we view
polynomials in $ R[\Lambda]$ as functions. 
More precisely, for any subset $\tS \subseteq R^{(n)},$ there is a
natural \semiring0 homomorphism
\begin{equation}\label{eq:polymap1} \psi:  R[\Lambda]  \to \Fun
(\tS, R),
\end{equation}
 obtained by viewing a polynomial as a function on $\tS$, and we write
 .$\overline{R[\Lambda]}$ for $\psi(R[\Lambda] )\subset \Fun
(\tS, R).$ (Likewise for $\tT [\Lambda]$  and $ \tG [\Lambda].$

To ease notation, we still write a typical element of
$\overline{R[\Lambda]}$ as $\left(\sum a_i \la_1^{i_1}\dots
\la_n^{i_n}\right)$. Any m-valuation $\nu : R \to \tG$ extends to
a map $\tilde \nu : \overline{R[\Lambda]} \to
\overline{\tG[\Lambda]} $ given by
$$\tilde \nu \left(\sum a_i \la_1^{i_1}\dots  \la_n^{i_n}\right) =
\sum \nu(a_i) \la_1^{i_1}\dots  \la_n^{i_n}.$$

\begin{rem}\label{vsemi23} If $R= (R,\tT, \tG,\nu)$ is a  \vdomain0, then so is
$(\overline{R[\Lambda]} ,\overline{\tT[\Lambda]}, \overline{\tG
[\Lambda]},\nu)$ (where $\la^\nu = \rone^\nu \la).$
\end{rem}

 This definition is
 concise, but we should note the difficulty that two polynomials,
 one tangible, and one non-tangible, can describe the same
 function, for example  $\la ^2 + 6$ and $(\la +3)^2 = \la ^2 + 3 \la ^\nu + 6$.
 One can remedy this formally by identifying polynomials agreeing on a dense subset of a
 variety.
%

 When $R$ is a \vsemifield0, the same
 analysis is applicable to the \vdomain0 of Laurent polynomials $\overline{R[\Lambda,
 \Lambda^{-1}]}$, since the homomorphism $\la_i \mapsto a_i$
 then sends $\la_i^{-1} \mapsto a_i^{-1}$.
But the algebraic structure of choice in this paper is the following
(and its $\nu$-analog described below):

\begin{defn}\label{ratfrac} When $(F,\tT, \tG,\nu)$ is a supertropical \vsemifield0, $\overline{F(\Lambda)}$
denotes $\Frac_\nu \overline{F[\Lambda]},$ the $\nu$-\semifield0
of fractions of~$\overline{F[\Lambda]}$ obtained by localizing at
$\overline{\tT[\Lambda]}$. We call $\overline{F(\Lambda)}$ the
$\nu$-\textbf{\semifield0 of rational functions} (over $F$); this
is our main subject of investigation.
\end{defn}

Viewed in $\Fun (F^{(n)},F)$, $\overline{F(\Lambda)}$ is contained
in $\{ fg^* : f,g\in \overline{F[\Lambda]}\},$ and they behave
similarly. We use the former ($\overline{F(\Lambda)}$) since it is
intuitively clearer, although probably the latter may be
technically superior.
%

 (Likewise,  we could also define
  $\overline{F[\Lambda]}_{\rat}$  where the exponents of the indeterminates $\la_i$ are taken to
  be arbitrary rational numbers; we
  could define substitution homomorphisms when $(F,\tT, \tG,\nu)$ is  power-cancellative and  divisible.)

These can all be
  viewed as elementary sentences in the appropriate language, so
 model theory  is applicable to polynomials
  and their (tropical) roots, to be considered shortly.

\begin{rem}\label{props} Here are the semilattice operations on
$\overline{F(\Lambda)}$:

 Suppose $f = \frac hg$ and  $f'= \frac {h'}{g'}$. Then
\begin{equation}\label{dua1} f\vee f' = f+f' = \frac{hg'+gh'}{gg'}; \qquad
f\wedge f' = ff'(f+f')^*  =  \frac{hh'}{gg'}{gg'}(hg'+gh')^* =
{hh'}(hg'+gh')^*.\end{equation} When $hg'+gh'$ is tangible, we have
\begin{equation}
\label{dua22} f\wedge f'= \frac{hh'}{hg'+gh'}.\end{equation}
\end{rem}

\subsection{Frobenius and archimedean properties}$ $

%
%

\begin{defn} A  \vdomain0 $R$ is  \textbf{$\nu$-Frobenius} if it satisfies the identity
$$(a+b)^n \nucong a^n + b^n, \qquad \forall a,b \in R.$$
$R$ is  $\nu$-\textbf{archimedean} if   if $a^{\mathbb{Z}}
\leq_{\nu} b$ implies that $a \nucong 1$.
\end{defn}

\begin{prop}\label{Fr1}
 \label{idpar1} If $R$ is a $\nu$-Frobenius \vdomain0, then so is $ \FunSR$.  \end{prop}
\begin{proof} Pointwise verification. \end{proof}

\begin{prop}\label{prop_semifield_of_fractions_is_archimedean}
If  the  idempotent \semifield0   $F$ is $\nu$-archimedean, then
$Fun(F^{(n)},F)$ is $\nu$-archimedean. 
\end{prop}
\begin{proof}
Let $f,g \in Fun(F^{(n)},\tG)$ such that $f^{\mathbb{Z}} \leq_\nu
g$. If $\bfa \in F^{(n)}$ then by our assumption $f(\bfa)^{k}
\leq_{\nu} g(\bfa)$ for all $k \in \mathbb{Z}$. Since $F$ is
$\nu$-archimedean and $f(\bfa),g(\bfa) \in F$, we have $f(\bfa) =
1$. Since this holds for any $\bfa \in F^{(n)}$ we have
$f(F^{(n)})\nucong 1$, implying $f \nucong 1$.
\end{proof}

\begin{lem}\label{rem_uniqueness_of_root}
If a supertropical semifield $F$ is $\nu$-Frobenius, then   $k$-th
roots in $ Fun(F^{(n)},F)$ are unique.\end{lem} \begin{proof} It is
enough to check this pointwise in $F$. Suppose $a^k = b^k$. Then
$(a+b)^k \nucong a^k + b^k \in \tG,$ implying $a+b \in \tG$. If
moreover, $a^k$ is tangible then so is $a$ and $b,$ so $a=b$. But if
$a^k \in \tG$, then $a$ and $b$ are both ghost, so $a=b$.
\end{proof}

\begin{rem}\label{log0}
Let $(\mathcal{S}, \cdot,  +  )$ be a divisible \semifield0. By
Lemma ~\ref{rem_uniqueness_of_root}, we can uniquely define any
rational power of the elements of $\mathcal{S}$. In this way,
$\mathcal{S}$ becomes a vector space over $\mathbb{Q}$, rewriting
the multiplicative operation~$\cdot$ on $\mathcal{S}$ as addition
and defining
\begin{equation}\label{eq_vector_space_translation}
(m/n) \cdot \alpha = \alpha^{\frac{m}{n}}.
\end{equation}

Furthermore, when $\mathcal{S}$ is complete, we can define real
powers as limits of rational powers, and  $\mathcal{S}$ becomes a
vector space over $\mathbb{R}$. In this way we can apply linear
algebra techniques to $(\mathcal{S},\cdot)$.
\end{rem}

 \section{Kernels and $\nu$-kernels}\label{ker0}

 Although the theory of congruences applies generally in universal
 algebra and in particular for  a \semiring0~$R$, congruences are difficult to work with since they
 involve subsets of $R\times R$ rather than $R$ itself. When $R$
 is a \semifield0 $\mathcal{S}$, one can get around this by switching the
 roles of addition and multiplication. Although we are interested in
 the $\nu$-structure, we present the definitions without $\nu$, in
 order to get to the underlying ideas more quickly.

 \subsection{$\nu$-Congruences and  $\nu$-kernels of \vsemifields0}$ $

\begin{defn} A \textbf{kernel} of a \semiring0 $\mathcal{S}$ is a
subgroup ${K}$ which is \textbf{convex} in the sense that if  $a,b
\in {K}$ and $\a, \beta \in \mathcal{S}$ with $\a + \beta =\one_{\mathcal{S}},$
then $\a a + \beta b \in {K}.$\end{defn}

\begin{note} Usually one takes $\mathcal{S}$ to be a \semifield0,
which is our assumption through the remainder of this subsection,
but we frame the definition a bit more generally in order to be
able later to handle \vsemifields0, which technically are not
\semifields0.  In some of the literature \semifields0 are not
required to be commutative, and then $K$ is required to be a
normal subgroup; we forego this generality.
\end{note}

 Following \cite[Proposition
(1.1)]{Prop_Semifields}, we have the following key correspondence.

\begin{prop}\label{basicprop} If $\Cong$ is a congruence
on a \semifield0 $\mathcal{S}$, then ${K}_\Cong = \left\{ a  \in
\mathcal{S} : a \equiv 1\right\}$ is a kernel. Conversely, any
kernel $K$ of $\mathcal{S}$ defines a congruence $\rho_K$
according to \cite[Definition~3.1]{Hom_Semifields}, i.e., $a
\equiv b$ iff $\frac a b \equiv 1.$ If $\mathcal{S}$ is the
\semifield0 of the lattice-ordered group $G$, then the \semifield0
$\mathcal{S}/\rho_K$ is the \semifield0 of the lattice-ordered
group~$G/K.$
\end{prop}

 A quick proof that $\Cong $ is indeed a congruence
is given in \cite[Proposition~1.1]{Cont_Semifields}.

We need to make a slight modification in order to take the
$\nu$-structure into account. Throughout, $R= (R,\tT, \tG,\nu)$ is
a \vdomain0. Any congruence of~$\tG$ also can be viewed as a
congruence of $R,$ but we also want to bring $\tT$ into play. The
next observation is due to Izhakian.

\begin{rem}\label{vsemi20}
Any congruence $\Cong$ satisfies the property that if $(a,b) \in
\Cong$ then $(a^\nu, b^\nu) = (\onenu,\onenu) (a,b) \in \Cong.$
\end{rem}

 \begin{defn}\label{vsemi24} A $\nu$-\textbf{congruence} is a  congruence $\Cong$
 for which $(a,b) \in \Cong$ iff $(a^\nu, b^\nu) \in \Cong.$ We write $a_1
\equiv_\nu  a_2$ when  $a_1^\nu \equiv  a_2^\nu.$

 A $\nu$-\textbf{kernel} of a \vsemifield0 $F$ is a
subgroup $K$ which is $\nu$-\textbf{convex} in the sense that if
$a,b \in K$ and $\a, \beta \in F$ with $\a + \beta \nucong \fone,$
then $\a a + \beta b \in K.$
\end{defn}
%
%

\begin{rem}
There is a natural correspondence between $\nu$-congruences
(resp.~$\nu$-kernels) of $R$ and   congruences (resp.~kernels) of
$\tG$, given as follows:

Any $\nu$-congruence $\Cong = \{(a,b): a,b \in R\}$ of $R$ defines
a congruence $\Cong^\nu = \{(a^\nu ,b^\nu ): (a,b) \in \Cong \}$
of $\tG$. Conversely, if $\Cong_\nu $ is a congruence of $\tG$,
then $\nu^{-1}(\Cong_\nu )$ is a $\nu$-congruence of $R$.

Any $\nu$-kernel ${\mathcal K}$ of a $\nu$-\semifield0 $F$ defines
a kernel ${\mathcal K}^\nu$ of $\tG$. Conversely, if ${\mathcal K}
$ is a kernel of $\tG$, then $\nu^{-1}({\mathcal K})$ is a
$\nu$-kernel of $F$.

 If
${\mathcal K} = {\mathcal K}_\Cong$, then $\nu^{-1}({\mathcal K})
= {\mathcal K}_{\nu^{-1}(\Cong)}$.
\end{rem}

Thus, we will quote results from the literature about kernels and
use their analogs for $\nu$-kernels.

\begin{thm}
Given a $\nu$-congruence $\Cong$ on a \vsemifield0 $F$,
corresponding to the $\nu$-kernel $K$ of $F$, the factor
\semiring0 $\bar F = F/\Cong$ is a \vsemifield0 with respect to
the idempotent map $\bar \nu: \bar F \to \overline{\tG}=
\tG/\nu(K)$ defined by $\bar \nu (Ka) = Ka^\nu.$\end{thm}
\begin{proof} If $a \nucong b$ then $ ab^* \in K,$ implying $Ka
= Kb.$ This implies $\bar \nu$ is well-defined, and its target is
the  \semifield0 $\overline{\tG}$.
\end{proof}

Writing $\bar a$ for the image of $a$ in $\bar R,$ we write $a
\equiv_\nu b$ to denote that $\bar a \nucong \bar b$, or
equivalently $\overline{a} \overline{b}^* \nucong \bar \fone.$
\begin{rem}\label{bfacts1}
\begin{enumerate} \eroman
\item  Given a $\nu$-congruence $\Cong$ on a
$\nu$-\semifield0~$F$, we define $K_\Cong = \left\{ a  \in F  : a
\equiv_\nu 1\right\}$. Conversely,  given a $\nu$-kernel $K$ of
$F$, we define the $\nu$-congruence $\Cong $ on $F$ by  $a \equiv
b$ iff $ a b^* \equiv_\nu \fone.$

 \item \cite[Corollary~1.1]{Cont_Semifields}, \cite[Property~2.4]{Prop_Semifields} Any  $\nu$-kernel $K$ is $\nu$-convex,
  in the sense that if $a \le_\nu  b \le_\nu  c$ with $a,c \in K$,
then $b \in K.$ This is  seen by passing to the factor \semiring0
$F/\Cong$  ($\Cong$ as in  (i)) and applying Lemma~\ref{idpar}.

 \item \cite[Proposition~2.3]{Cont_Semifields}. If $|a| \in {K}$, a $\nu$-kernel,
then $a\in {K}$.

 \item \cite{Cont_Semifields} The product $K_1   K_2 = \{ ab \ : \
a \in K_1, b \in K_2 \}$   of two $\nu$-kernels  is    a
$\nu$-kernel, in fact the smallest $\nu$-kernel containing $K_1
\cup K_2$. (This follows at
 once from (ii).)

 \item The intersection of  $\nu$-kernels is a $\nu$-kernel. Thus, for any set $S \subset F$ we
can define the $\nu$-kernel $\langle S \rangle$  \textbf{generated
by} $S$ to be the intersection of all $\nu$-kernels containing
$S$.

\item  \cite[Theorem~3.5]{Cont_Semifields}. Any $\nu$-kernel
generated by a finite set $\{ s_1, \dots, s_m\}$ is in fact
generated by the single element $\sum _{i=1}^m (|s_i|).$ (Follows
easily from (iii).)

\item The $\nu$-kernel generated by $a \in F$ is just the set of
finite sums $  \{ \sum_i  b_i  a^i : \ b_i \in F, \ \sum b_i
\nucong 1\}.$

 \item  \cite[Theorem~3.8]{Hom_Semifields}. If $K$ is a $\nu$-kernel of a \vsemifield0
$F$ and the \vsemifield0 $F/K$ is $\nu$-idempotent, then $K$ is a
sub-\vsemifield0 of $F$. (This is because for $a,b \in K$ the
image of $a+b$ is $\nu$-equivalent to $1K + 1K = 1K.$)

\item Let $K$ be a $\nu$-kernel of a \vsemifield0 $F$. For every
$a \in \tT$, if $a^{n} \in K$ for some $n \in \mathbb{N}$ then $a
\in K$. (Indeed, passing to $F/K$,  it suffices to note by the
analog of Lemma~\ref{torfree} that if $a^n=1$ then $a =1.$)

\item Any $\nu$-kernel of a $\nu$-kernel is a $\nu$-kernel.
\end{enumerate}
\end{rem}



%

%

%
%

\begin{prop}\label{lem_kernels_algebra}\label{prop_convex_l_subgroups_distributive_lattice}\cite[Theorem (2.2.5)]{OrderedGroups3}
The lattice of $\nu$-kernels of an idempotent \vsemifield0 is a
complete distributive lattice and satisfies the infinite
distributive law.\end{prop}

\begin{thm}\label{thm_kernels_hom_relations}\cite[Theorems (3.4) and (3.5)]{Hom_Semifields}
Let $\phi : F_1 \rightarrow F_2$ be a \vsemifield0 homomorphism.
Then the following hold:
\begin{enumerate}
 \item For any $\nu$-kernel $L$ of \ $F_1$, $\phi(L)$ is a $\nu$-kernel of \ $\phi(F_1)$.
  \item For a $\nu$-kernel $K$ of \ $\phi(F_1)$, $\phi^{-1}(K)$ is a $\nu$-kernel of \ $F_1$.\end{enumerate}
In particular, $\phi^{-1}(1)$ is a $\nu$-kernel.\end{thm}

\begin{cor}\label{cor_subdirect_decomposition_for_semifield_of_fractions}
There is an injection $\mathbb{S}/( K_1 \cap K_2 \cap \dots  \cap
K_t) \hookrightarrow \prod_{i=1}^{t} \mathbb{S}/K_i$, for any
$\nu$-kernels $K_i$  of a  \vsemifield0~$\mathbb{S}$, induced by
the map $f \mapsto (fK_i).$
\end{cor}

\begin{flushleft}We have the   fundamental isomorphism theorems.\end{flushleft}

\begin{thm}\cite{KA_Semifields}\label{thm_nother_1_and_3}
Let $\mathcal{S}$ be an idempotent \vsemifield0 and $K,L$ be
$\nu$-kernels of $\mathcal{S}$.
\begin{enumerate}
  \item If $\mathcal{U}$ is a sub-\vsemifield0 of  $\mathcal{S}$, then $\mathcal{U} \cap K$ is a $\nu$-kernel of $\mathcal{U}$, and $K$ a $\nu$-kernel of the sub-\vsemifield0  $\mathcal{U}  K = \{u \cdot k \ : \ u \in \mathcal{U}, \ k \in K \}$ of \ $\mathcal{S}$,
   and one has the isomorphism $$\mathcal{U}/(\mathcal{U} \cap K) \cong \mathcal{U} K/K.$$
  \item $L \cap K$ is a $\nu$-kernel of $L$ and $K$ a $\nu$-kernel of $L  K$, and the group isomorphism $$L/(L \cap K) \cong L   K/K$$   is a \vsemifield0 isomorphism.
  \item If $L \subseteq K$, then $K/L$ is a $\nu$-kernel of \ $\mathcal{S}/L$ and one has the \vsemifield0 isomorphism $$\mathcal{S}/K \cong (\mathcal{S}/L)/(K/L).$$
\end{enumerate}
\end{thm}

\begin{thm}\label{cor_qoutient_corr}
Let $L $ be a $\nu$-kernel of  a \vsemifield0 $F$. Every
$\nu$-kernel of $F/L$ has the form $K/L$ for some uniquely
determined $\nu$-kernel $K \supseteq L$, yielding a lattice
isomorphism
$$ \{ \text{Kernels of} \ F/L \} \rightarrow  \{ \text{Kernels of} \ F \ \text{containing} \ L \}$$
given by $K/L \mapsto K$.
\end{thm}
\begin{proof}
From the theory of groups, there is such a bijection for
subgroups. We  note    by Theorem \ref{thm_kernels_hom_relations}
that the homomorphic image and pre-image of a $\nu$-kernel are
$\nu$-kernels.
\end{proof}

\begin{defn}
A \vsemifield0 which contains no $\nu$-kernels but the trivial
ones, $\{1,1^\nu\}$ and itself, is called \textbf{simple}.
\end{defn}

\begin{rem}\label{cor_max_ker_simple_corr}
A $\nu$-kernel $K$ of a  \vsemifield0 $F$ is a maximal
$\nu$-kernel if and only if the \vsemifield0 $F/K$ is simple.
\end{rem}

\subsubsection{Principal $\nu$-kernels}

\begin{defn}\label{defn_generation_of_kernels}
A $\nu$-kernel $K$ is said to be \textbf{finitely generated}  if
$K = \langle S \rangle$ where $S$ is a finite set.  If $K =
\langle a \rangle$ for some $a \in F$, then $K$ is called a
\textbf{principal $\nu$-kernel}.  \end{defn}

 $\PCon(K )$ denotes the set
of principal $\nu$-subkernels  of a $\nu$-kernel $K.$ $\PCon(K )$
turns out to be a lattice, and one of the keys of our
  study.

\begin{lem}\label{lem_generators}\cite[Property 2.3]{Prop_Semifields}
Let $K$ be a $\nu$-kernel  of an idempotent \vsemifield0
$\mathcal{S}$. Then for $a,b \in \mathcal{S}$,
\begin{equation}
|a| \in K \ \ \text{or} \ \ |a|+ b \in K \ \ \Rightarrow a \in K.
\end{equation}
\end{lem}

\begin{prop}\label{prop_principal_ker}\cite[Proposition (3.1)]{Prop_Semifields}
\begin{equation}
\langle a \rangle = \{ x \in \mathcal{S} \ : \ \exists n \in
\mathbb{N} \  \ \ \text{such that} \ \ \  a^{-n} \leq x \leq a^{n}
\}.
\end{equation}
\end{prop}

\begin{cor}\label{prop_absolute_generator}\cite{Prop_Semifields}
 $
\langle a \rangle = \langle a^k \rangle, $ for any element $a$ of
$\mathcal{S}$ and any $0 \ne k \in \mathbb Z $.
\end{cor}

\begin{cor}\label{prop_absolute_generator1}\cite{Prop_Semifields}
$
\langle a \rangle = \langle |a| \rangle,
$
for any element $a$ of    $\mathcal{S}$.
\end{cor}

A direct consequence of Proposition \ref{prop_principal_ker} and
Corollary \ref{prop_absolute_generator1} is:

\begin{cor}\label{cor_principal_ker_by_order}
For any $a \in \mathcal{S}$,
\begin{align*}
\langle a \rangle = \ & \{x \in \mathcal{S} \ : \ \exists n \in
\mathbb{N} \ \text{such that} \  |a|^{-n} \leq_\nu   x \leq_\nu
|a|^{n} \}.
\end{align*}
\end{cor}

\begin{defn}\label{defn_generation_of_kernels1} A \vsemifield0 is said to be
\textbf{finitely generated} if it is finitely generated as a
$\nu$-kernel. If $\mathcal{S} = \langle a \rangle$ for some $a \in
\mathcal{S}$, then $a$ is called a \textbf{generator} of
$\mathcal{S}$ and $\mathcal{S}$ is said to be a \textbf{principal
\vsemifield0}.

\end{defn}

\begin{cor}\label{cor_positive_generator_of_a_semifield}
Every nontrivial principal \vsemifield0\ $\mathcal{S}$   has a
generator $a
> 1$.
\end{cor}
\begin{proof}
Let $u \in \mathcal{S} \setminus \{1\}$ be a generator of
$\mathcal{S}$. By Lemma \ref{lem_generators}, the element $|u|$ is
also a generator of $\mathcal{S}$ which yields that the element $a
= |u|^2 = u^2 + u^{-2} + 1 \ge 1$ is a generator of $\mathcal{S}$
too, by Proposition~\ref{prop_principal_ker}. But $a \ne 1,$ so $a
>1.$
\end{proof}

\begin{thm}\label{thm_semi_with_a_gen}
If an idempotent \vsemifield0\ $\mathcal{S}$ has a finite number
of generators $a_1, \dots, a_n$, then $\mathcal{S}$ is a principal
\vsemifield0, generated by $a  = |a_1|   + \dots + |a_n| .$
\end{thm}
\begin{proof}
 By
Lemma~\ref{lem_generators}, $a_1,...,a_n$ are contained in the
$\nu$-kernel $\langle a \rangle \subseteq \mathcal{S}$; hence,
$\mathcal{S} = \langle a \rangle$ as desired.
\end{proof}

%

\begin{prop}\label{rem_order_simple}
Every   idempotent $\nu$-archimedean \vsemifield0\ $\mathcal{S}$
is simple.
\end{prop}
\begin{proof}
We may assume that $\mathcal{S} \neq \{ 1 \}$. Take $a \in
\mathcal{S}$ such that $a > 1$.   Since $\mathcal{S}$ is
$\nu$-archimedean, for every $b \in \mathcal{S}$ there exists $m
\in \mathbb{N}$ such that $a^{-m} \leq _\nu b \leq_\nu a^{m}$, so
$b \in \langle a \rangle$ by Proposition \ref{prop_principal_ker}.
Thus $\langle a \rangle = \mathcal{S}$ and our claim is proved.
\end{proof}

 \begin{nota} By Proposition~\ref{rem_order_simple}, $\mathcal{S} =
 \langle \alpha \rangle$ for each $\alpha \ne \{1\}$.

 The $\nu$-kernel $\langle \mathcal{S} \rangle$ in $\mathcal{S}(\Lambda)$ is
 somewhat bigger, containing all functions $\a f + \beta g$ where
 $f+g $ is the constant function $1$. This has a significant role,
 discussed in \S\ref{subsection:bounded_kernel} below, where we describe $\mathcal{S}(\Lambda)$
 more explicitly.
\end{nota}

\begin{prop}\label{prop_maximal_kernels_in_semifield_of_fractions_part1}
For any $\gamma_1,....,\gamma_n \in \mathcal{S}$ the $\nu$-kernel
$\langle \frac{\la_1}{\gamma_1},...,\frac{\la_n}{\gamma_n}
\rangle$ is a maximal $\nu$-kernel of $\mathcal{S}(\Lambda)$.
\end{prop}
\begin{proof} The quotient is isomorphic to $\mathcal{S}$, which is simple.
\end{proof}

Here are more properties of $\nu$-kernels  of an idempotent
\vsemifield0\ and their generators.

\begin{prop}\label{prop_algebra_of_generators_of_kernels}\cite[Theorem 2.2.4(d)]{OrderedGroups3}
For any $V,W \subset \mathcal{S}$, any $\nu$-kernel  $K$ of
$\mathcal{S} $ and $f,g \in \mathcal{S}$ the following statements
hold:
\begin{enumerate}\eroman
  \item $\langle V \rangle   \langle W \rangle = \langle V \cup W \rangle = \langle \{|x| \dotplusss  |y| : x \in V, y \in W \} \rangle$.
  \item $\langle V \rangle \cap \langle W \rangle = \langle \{|x| \wedge |y| : x \in V,\ y \in W \} \rangle$.
  \item $\langle K, f \rangle \cap \langle K ,g \rangle  = K  \langle |f| \wedge |g| \rangle$.
\end{enumerate}
\end{prop}

\begin{lem}\label{rem_sum_prod_absolute_values}
\begin{equation}\label{eq_sum_prod_absolute_values}
|g||h| \in K \ \ \Leftrightarrow \ \ |g| \dotplusss  |h| \in K.
\end{equation}
for any $\nu$-kernel $K$ of an idempotent \vsemifield0.
\end{lem}
\begin{proof}
  $|g| \leq |g||h|$
and $|h| \leq |g||h|$  since $|g|, |h| \geq 1$, and thus
$$|g|+  |h| = \sup(|g|,|h|) \leq |g||h|.$$

On the other hand,
$$(|g| \dotplusss  |h|)^2 = |g|^2 \dotplusss  |g||h| \dotplusss  |h|^2 \geq |g||h|,$$
so, by Corollary \ref{cor_principal_ker_by_order}, $\langle |g||h|
\rangle = \langle |g| + |h| \rangle$ and
\eqref{eq_sum_prod_absolute_values} follows.
\end{proof}

Thus, we have the following fact, essentially garnered from
\cite{OrderedGroups3}.

\begin{thm}\label{cor_max_semifield_principal_kernels_operations}
Let $\mathcal{S}$ be an idempotent \vsemifield0. Then the
intersection and product of two principal $\nu$-kernels are
principal $\nu$-kernels. Namely, for every $f,g \in \mathcal{S}$
\begin{equation}\label{eq_max_principal_operations}
\langle f \rangle \cap \langle g \rangle = \langle  |f|   \wedge
 |g|  \rangle \ ; \ \langle f \rangle   \langle g \rangle =
\langle  |f|  |g|  \rangle.
\end{equation}
\end{thm}
\begin{proof}
Taking $V = \{f\}$ and $W = \{g\}$ in Proposition
\ref{prop_algebra_of_generators_of_kernels} yields the first
equality, whereas for the second equality,
Lemma~\ref{rem_sum_prod_absolute_values}  yields
$$\langle f \rangle \langle g \rangle = \langle  |f|  \dotplusss  |g|  \rangle,$$
 and again we apply the proposition.
\end{proof}

\begin{cor}\label{cor_principal_kernels_sublattice}
The set of principal $\nu$-kernels of an idempotent \vsemifield0\
forms a sublattice of the lattice of $\nu$-kernels (i.e., a
lattice with respect to intersection and  multiplication).
\end{cor}

\begin{cor}\label{prop_frac_semifield_finitely_gen}
For any generator $a$ of a \vsemifield0 $F$,
$\overline{F(\Lambda)} = \langle a \rangle \prod_{i=1}^{n} \langle
\la_i \rangle$,  and $\overline{F(\Lambda)}$ is a principal
\vsemifield0\ with generator $\sum_{i=1}^n|\la_{i}|+ |a|$.
\end{cor}
\begin{proof}
By Corollary \ref{cor_principal_kernels_sublattice} it is enough
to prove that any monomial $f$ in $\overline{F[\Lambda]}$ belongs
to $\langle \sum_{i=1}^n|\la_{i}|+ |a| \rangle,$ which follows
from Theorem~\ref{thm_semi_with_a_gen}.
\end{proof}

 \subsection{Digression: Fractional $\nu$-kernels}$ $

Having a way of transferring information to and from \semifields0,
let us introduce fractional  $\nu$-kernels, to be able to pass to
$\nu$-\domains0. This should be a useful tool in the further study
of $\nu$-kernels in tropical mathematics.

%

\begin{defn}  A \textbf{fractional $\nu$-kernel} $K$ on a
$\nu$-\domain0 $R$  is a $\nu$-kernel on the \vsemifield0 $\Frac
_\nu R$.
\end{defn}

\begin{rem}
Extending Remark~\ref{bfacts1}, given a $\nu$-congruence $\Cong$
on a $\nu$-\domain0 $R$ with \vsemifield0 of fractions~$F$, we can
define the fractional $\nu$-kernel $K_\Cong = \left\{ ab^*\in F :
a,b \in R \text{ and } a \equiv_\nu b\right\}$. (Or, in other
words, extending $\Cong$ to $F$ as in Lemma~\ref{ext}, we require
that $ ab^* \equiv_\nu  \fone$.)
\end{rem}

A $\nu$-congruence $\Cong$ on $R$ is called
$\nu$-\textbf{cancellative}
 if $a_1 b  \equiv_\nu   a_2 b$ for some  $b\in R$ implies $a_1 \equiv_\nu  a_2$.
\begin{thm} There is a 1:1 correspondence between $\nu$-cancellative
congruences of a $\nu$-\domain0 $R$ and  fractional $\nu$-kernels,
given by $\Cong \to K_\Cong.$ Any  homomorphism of $\nu$-\domains0
$R \to R'$ (for $R'$ an arbitrary $\nu$-\domain0) gives rise to a
fractional $\nu$-kernel, and $\Cong$ is the congruence corresponding
to $R \mapsto R/K$.
\end{thm}
\begin{proof}  We extend $\Cong$ to $F$ as in Lemma~\ref{ext}. As noted in
\cite[Theorem~3.2]{Hom_Semifields}, $K_\Cong$ is clearly a
congruence, since $\a a + \beta b  \equiv  \a + \beta \nucong
\fone$ for all $a,b \in K_\Cong$.

In the other direction, we restrict this to $R$.

Given a homomorphism $R \to R',$ we compose it with the injection
of $R'$ into its \vsemifield0 of fractions $F',$ and then extend
this naturally to a homomorphism of \semifields0 $F \to F'$,
thereby obtaining a $\nu$-kernel. \end{proof}

%

 \section{The supertropical connection:
 Corner loci and \skeletonbs}\label{supertrop1}


In this section we apply our theory  to  tropical geometry. Let us
recall the basic notions concerning supertropical varieties.
Throughout,  $F$ denotes a \vsemifield0.

\subsection{Corner loci}$ $

In \cite[Section (5.2)]{SuperTropicalAlg} Izhakian and Rowen have
generalized the notion of (tangible) corner root to $F[\Lambda]$
(over a supertropical \vsemifield0 $F$) as follows:

\begin{defn}\label{def_st_corner_root}
Suppose $f \in F[\Lambda]$ is a supertropical polynomial, written
$f=\sum_{i=1}^{k}f_i$ where each $f_i$ is a monomial.  A point
$\bfa \in F^{(n)} $ is said to be a  \textbf{ghost root} of $f$ if
$f(\bfa) \in \mathcal{G}$, i.e., if $f$ has a ghost value
  at $\bfa$. \end{defn}

  This happens in one of the following cases:
\begin{enumerate}
  \item There are two distinct monomials $f_t$ and $f_s$ of $f$ such that $\small{f(a) = f_s(a) = f_t(a)}$.
  \item There exists a ghost monomial $f_t$ of $f$ such that $f(a) = f_t(a)$.
\end{enumerate}


 \begin{rem}\label{rem_cor_roots_of_prod}\begin{enumerate}\eroman
\item If $\bfa, \bfb \in F^{n}$ are ghost roots of $f,g \in
F[\Lambda]$ respectively, then both $\bfa$ and $\bfb$ are ghost
roots of the product $f g$.

\item In view of Proposition~\ref{Fr1}, any ghost root $\bfa$
of $f^k$ for $k\ge 1$ is a ghost root  of $f$.
%
\end{enumerate}
\end{rem}

\begin{defn}\label{def_cor_locus}
A set $A \subseteq \tT^{(n)} $ is said to be a \textbf{corner locus}
if $A$ is a set of the form
\begin{equation}
A  = \{ x \in F^{n} \ : \ \forall f \in S , \ x \ \text{is a ghost
root of} \ f \}
\end{equation}
for some $S \subset F[\Lambda]$. We write $\Corn(S)$ to denote the
corner locus defined by $S$.

In the case where $S = \{ f_1,...,f_m \}$ is finite,
 we say that $A$ is a \textbf{finitely generated} corner locus and
  write $A = \Corn(f_1,...,f_m)$.
\end{defn}

\begin{lem}\label{rootlift01} Suppose $f =  f'+f''  \in F[\Lambda]$.
   Then $\bfa \in \Corn(f)$ iff one of the following hold:
   \begin{enumerate}\eroman
\item $ \bfa \in  \Corn({f'})$ with $f'(\bfa) \ge f''(\bfa)$;
\item $ \bfa \in  \Corn({f''})$ with $f''(\bfa) \ge f'(\bfa)$;
\item $  f'(\bfa) \nucong f''(\bfa).$
   \end{enumerate}
\end{lem}
\begin{proof} These are the only ways to get a ghost value.
\end{proof}

 This concise
formulation
 enables us to  treat tropical varieties algebraically  as simultaneous roots of
 sets of
 polynomials. In particular, a \textbf{hypersurface} now is viewed as the corner locus of a single
polynomial. When the polynomials are tangible, the hypersurfaces
 are those from usual tropical geometry.

We write $\mathbf{2}^S$ for the power set of a set $S$.

In view of  Definition \ref{def_cor_locus}, we define an operator
$\small{\Corn : \mathbf{2}^{F[\Lambda]}} \rightarrow \mathbf{2}^{F^{(n)}}$ given
by
\begin{equation}
\Corn : S \subset F[\Lambda] \mapsto \Corn(S).
\end{equation}
%

We now proceed to study the behavior of the $\Corn$ operator.

\begin{rem}\label{rem_intersection_union_corner_loci}
If $S_A,S_B \subseteq F[\Lambda]$ then
\begin{equation}\label{cor_loci_prop1}
S_A \subseteq S_B \Rightarrow \Corn(S_B) \subseteq \Corn(S_A).
\end{equation}

Let $\{ S_i \}_{i \in I}$ be a family of subsets of $F[\Lambda]$ for
some index set $I$. Then $\bigcap_{i \in I}\Corn(S_i)$ is a corner
locus and
\begin{equation}\label{cor_loci_prop2}
\bigcap_{i \in I}\Corn(S_i) = \Corn\left(\bigcup_{i \in I}S_i\right)
\ ; \ \bigcup_{i \in I}\Corn(S_i) \subseteq \Corn\left(\bigcap_{i
\in I}S_i\right).
\end{equation}
In particular, $\Corn(S)= \bigcap_{f \in S}\Corn(f)$.
\end{rem}
\begin{proof}
First, equality \eqref{cor_loci_prop1} is a direct set-theoretical
consequence of the definition of corner loci. In turn this implies
that $\Corn(\bigcup_{i \in I}S_i) \subseteq \Corn(S_i)$ for each $i
\in I$ and thus $\Corn(\bigcup_{i \in I}S_i) \subseteq \bigcap_{i
\in I}\Corn(S_i)$. Conversely, if $x \in F^{(n)} $ is in $\bigcap_{i
\in I}\Corn(S_i)$ then $\bfa \in \Corn(S_i)$ for every $i \in I$,
which means that $\bfa$ is a common ghost root of $\{f \ : \ f \in
S_i \}$. Thus $\bfa$ is a common ghost root of $\{f \ : \ f \in
\bigcup_{i \in I} S_i \}$ which yields that  $\bfa \in
\Corn(\bigcup_{i \in I}S_i)$. For the second equation (inclusion) in
\eqref{cor_loci_prop2}, $\bigcap_{i \in I}S_i \subseteq S_j$  for
each $j \in I$. Thus, by \eqref{cor_loci_prop1}, $ \Corn(S_j)
\subseteq \Corn(\bigcap_{i \in I}S_i)$, and so, $\bigcup_{i \in
I}\Corn(S_i) \subseteq \Corn(\bigcap_{i \in I}S_i)$.
\end{proof}

The next lemma deals with the case where $A$ and $B$ are finitely
generated.

\begin{lem}\label{lem_intersection_union_finitely_generated_corner_loci} If
$A=\Corn(f_1,...,f_s)$ and $B = \Corn(g_1,...,g_t)$, then $A \cap B
= \Corn(f_1,...,f_s,g_1,...,g_t)$, and $A \cup B = \Corn(\{f_i
g_j\}_{i=1,j =1}^{s,t})$.
\end{lem}
\begin{proof}
In view of Definition \ref{def_cor_locus} and Remarks
\ref{rem_cor_roots_of_prod} and
\ref{rem_intersection_union_corner_loci}, we only need to prove that
$A \cup B = \Corn\left(\{f_i g_j\}_{i=1,j =1}^{s,t}\right)$. Indeed,
for each $1 \leq i \leq s$ and each $1 \leq j \leq t$, \ $A
\subseteq \Corn(f_i)$ and $B \subseteq \Corn(g_j)$. Thus $A \cup B
\subseteq \Corn(f_i) \cup \Corn(g_j) = \Corn\left(\{f_i
g_j\}\right)$ and so
$$A \cup B \subseteq \bigcap_{i=1,j=1}^{s,t}\Corn\left(\{f_i  g_j\}\right) = \Corn\left(\{f_i  g_j\}_{i=1,j =1}^{s,t}\right).$$ On the other hand, if $a \not \in A \cup B$ then there exist some $i_0$ and $j_0$ such that $a \not \in \Corn(f_{i_0})$ and $a \not \in \Corn(g_{j_0})$. Thus $a \not \in \Corn(f_{i_0}) \cup \Corn(g_{j_0}) = \Corn(f_{i_0}  g_{j_0})$. So   $a \not \in \bigcap_{i=1,j =1}^{s,t}\Corn\left(\{f_i  g_j\}\right) = \Corn\left(\{f_i  g_j\}_{i=1,j =1}^{s,t}\right)$, proving the opposite inclusion.
\end{proof}

 In \cite{Coordinate} the following density condition was introduced to
 pick out those varieties defined by several polynomials, which are more in line with the usual
 tropical viewpoint.

\begin{defn} A corner locus  $Z$ is  \textbf{admissible} if it
satisfies the following property:

If $f$ and $g$ agree on a $\nu$-dense subset of $Z$ then $f= g.$

A \textbf{variety} is an admissible corner locus that is not the
union of two admissible corner loci.
\end{defn}

%

\subsection{$\onenu$-sets}\label{Skel}$ $

We present a method for describing a `corner locus' by the
analogous concept for $\nu$-kernels, which we call a
\textbf{$\onenu$}-set.
 This in turn sets the stage for applying the theory of $\nu$-kernels to tropical
 geometry, focusing on the correspondence between $\nu$-kernels and $\onenu$-sets.
  We introduce
the geometric notion of `$\onenu$-set' and the algebraic notion of
`$\nu$-kernel of a $\onenu$-set', the respective analogs of affine
varieties and their ideals. To obtain a Zariski-type
correspondence, we define a pair of operators $\Skel$ and $\Ker$
where $\Skel$ maps a $\nu$-kernel to its $\onenu$-set and $\Ker$
maps a $\onenu$-set to its
corresponding $\nu$-kernel. 
%
%


 Throughout
this section, we take $F$ to be a \vsemifield0.
 $\overline{F(\Lambda)}$
was defined in Definition~\ref{ratfrac}.
\begin{defn}
A \textbf{kernel root} of $ f  \in \FunF$  is an element $\bfa \in
F^{(n)}$ such that $ f (\bfa) \nucong \fone.$

For $S \subseteq  \overline{F(\Lambda)}$, define the subset
$\Skel(S)$ of $F^{(n)}$  as
\begin{equation}
\Skel(S) = \{ \bfa \in F^{(n)}  \ : \  f(\bfa) \nucong 1, \ \forall
f \in S \}.
\end{equation}
\end{defn}

We write  $\Skel(f)$ for $\Skel(\{f\}).$

\begin{lem}\label{rem_max}
The following hold for $f,g \in \overline{F(\Lambda)}$:
\begin{enumerate}\eroman
  \item $\Skel(fg) = \Skel(f \dotplusss  g)  = \Skel(f) \cap \Skel(g)$ for all $f,g \ge_\nu 1$.
  \item $\Skel(f \wedge g)  = \Skel(f) \cup \Skel(g)$ for all $f,g \leq_\nu 1$.

\end{enumerate}
\end{lem}
\begin{proof}
(i)     $(f \dotplusss  g)(\bfa) = f(\bfa) \dotplusss  g(\bfa)
\nucong 1$ if and only if $f(\bfa)\nucong 1$ and $g(\bfa)\nucong 1$
(since  both $f,g \ge_\nu 1$). Analogously the same holds for
$f(\bfa)g(\bfa) \nucong 1$.

(ii)   $(f \wedge g)(\bfa) = f(\bfa) \wedge g(\bfa) \nucong 1$, if
and only if $f(\bfa) \nucong 1$ or $g(\bfa)\nucong 1$.
\end{proof}

 \begin{lem}\label{rem_max0} $\Skel(f) = \Skel(f^{-1}) = \Skel(|f|)  = \Skel(f \wedge f^{-1})$.
\end{lem}
\begin{proof} Clearly $f(\bfa) \nucong 1$ iff  $f^{-1}(\bfa) \nucong 1$, in
which case $|f|(\bfa) \nucong (f \wedge f^{-1})(\bfa) \nucong 1$.
Conversely, if $|f|(\bfa) \nucong 1$ then both $f(\bfa),f^{-1}(\bfa)
\le_\nu 1$, implying $f(\bfa),f^{-1}(\bfa) \nucong 1$. The last
assertion is by Lemma~\ref{rem_max}(ii).
\end{proof}

\begin{defn}
A subset $Z\subset F^{(n)}$ is said to be a \textbf{$\onenu$-set}
if there exists a subset $S \subset \overline{F(\Lambda)}$  such
that $Z = \Skel(S)$.
\end{defn}
%

\begin{prop}\label{prop_skel_property1}
For $S_i \subset \overline{F(\Lambda)}$ the following statements
hold:
\begin{enumerate}
  \item $S_1 \subseteq S_2 \Rightarrow \Skel(S_2) \subseteq \Skel(S_1)$.
  \item $\bigcap_{i \in I}\Skel(S_i) = \Skel(\bigcup_{i \in I}S_i)$ for any index set $I$ and in particular, $\Skel(S)= \bigcap_{f \in S}\Skel(f)$.
\end{enumerate}
\end{prop}
\begin{proof}
The  assertions are  formal.
\end{proof}

\begin{lem}\label{oneset}   $\Skel(S_1) = \Skel(\langle S_1 \rangle)$.
\end{lem}
\begin{proof} For any
$f,g$ such that $f(\bfa)\nucong g(\bfa)\nucong 1$,
 $$(fg)(\bfa) = f(\bfa)g(\bfa) \nucong 1 \cdot
1 = 1, \quad f^*(\bfa)\nucong f(\bfa)^* \nucong 1^{-1}=1, \quad
\text{ and } \quad (f+g)(\bfa) = f(\bfa) + g(\bfa) \nucong
1+1\nucong 1.$$ To show that convexity is preserved for
$f_1,...,f_t \in \overline{F(\Lambda)}$ with
$\sum_{i=1}^{t}f_i\nucong 1$ and for any $g_1,...,g_t$ in
$\Skel(S_1)$, we note that
$$\left(\sum_{i=1}^{t}f_ig_i\right)(\bfa) =  \sum_{i=1}^{t}f_i(\bfa)g_i(\bfa)\nucong
\sum_{i=1}^{t}(f_i(\bfa)\cdot 1)\nucong  \sum_{i=1}^{t}f_i(\bfa)
\nucong 1.$$
\end{proof}

\begin{prop}\label{prop_ker_skel_correspondence}\cite[Proposition (4.2.6)]{AlgAspTropMath}
Let $K_i $ be $\nu$-kernels of $\overline{F(\Lambda)}$, and let
$Z_i=\Skel(K_i)$ be their corresponding $\onenu$-sets. Then the
following statements hold:
\begin{equation}
\Skel(K_1  K_2) = Z_1 \cap Z_2;
\end{equation}
\begin{equation}
\Skel(\cap K_i) =   \cup Z_i.
\end{equation}
\end{prop}

\begin{note}\label{compl} The proposition fails miserably when we take $\wedge$ of an infinite
set. For example, we could take $f_m $ to be a series of constants
approaching $1.$ Then $\inf \{ f_m \} = 1,$ whose $\onenu$-set is
$F^{(n)},$ although $\Skel(f_m) = \emptyset$ for each $m$.
\end{note}

\begin{defn}
A $\onenu$-set $Z$  is said to be  \textbf{principal}  if there
exists $f \in \overline{F(\Lambda)}$ such that $Z = \Skel(f)$.
\end{defn}

\begin{defn}\label{defn_principal_kernels_in_semifield_of_fractions}
Denote the collection of $\onenu$-sets in $F^{(n)}$ by
$\Skl(F^{(n)})$ and the collection of principal $\onenu$-sets in
$F^{(n)}$ by $\hSkl$. ($(F^{(n)})$ will be understood in the
context.)
\end{defn}

Although principal $\onenu$-sets are analogous to hypersurfaces,
they are more pervasive because of Lemma~\ref{rem_max}.

\begin{prop}\label{prop_skel_ker_relations0}$ $
 $\hSkl$ is closed under finite unions and intersections.
\end{prop}
\begin{proof} By  the lemma.
\end{proof}

As a special case of Proposition \ref{prop_ker_skel_correspondence}
we have

\begin{cor}\label{cor_principal_skel_correspondence}
For $f,g \in \overline{F(\Lambda)}$,
\begin{equation}\label{eq_1_prop_principal_skel_correspondence}
\Skel(\langle f \rangle  \langle g \rangle ) = \Skel(f) \cap
\Skel(g);
\end{equation}
\begin{equation}\label{eq_2_prop_principal_skel_correspondence}
\Skel(\langle f \rangle \cap \langle g \rangle) = \Skel(f) \cup
\Skel(g).
\end{equation}
\end{cor}

\begin{prop}\label{gen21}
 If $\langle f \rangle$ is a principal $\nu$-kernel generated by $f \in \overline{F(\Lambda)}$, then $\Skel(f) = \Skel(\langle f \rangle)$.
Consequently, $\Skel(f) = \Skel(f')$ for any generator $f'$
of~$\langle f \rangle$.
\end{prop}
\begin{proof}
By Corollary \ref{cor_principal_ker_by_order}, $g \in
\overline{F(\Lambda)}$ is in $\langle f \rangle$ if and only if
there exists some $n \in \mathrm{N}$ such that   $|f|^{-n}
\leq_\nu g \leq_\nu |f|^{n}$. Hence, for every $g \in \langle f
\rangle$, $1 = 1^{-n} \le_\nu g(\bfa) \le_\nu 1^{n} = 1$ for each
$\bfa \in \Skel(f)$, implying $g(\bfa)\nucong 1$.

In particular, $\Skel(f) =  \Skel(\langle f \rangle)  =
\Skel(\langle f' \rangle) = \Skel(f')$.
\end{proof}

%
%
%
%

\subsection{Kernels of $\onenu$-sets}\label{Subsection:Kernels_of_1-Sets}$ $

In view of Lemma~\ref{oneset}, the operator $\Skel$ can be
restricted to $\nu$-kernels. In the other direction, we now
construct an operator that associates a $\nu$-kernel of the
\vsemifield0\ of fractions $\overline{F(\Lambda)}$ to any subset
of $ F^{(n)} $. Then we note that the operator $\Ker$ and the
operator $\Skel$ defined in the previous subsection are inverses
of each other.

\begin{defn}
Given a subset $Z$ of $F^{(n)}$, we define

\begin{equation}
\Ker(Z) = \{ f \in \overline{F(\Lambda)} \ : \ f(\bfa) \nucong 1,
\ \forall \bfa \in Z    \}.
\end{equation}
\end{defn}

 $\Ker(Z)$ is a $\nu$-kernel of $\overline{F(\Lambda)}$, by the same argument as given in
Lemma~\ref{oneset}.

\begin{rem}\label{rem_ker_properties_1}
The following statements hold for $Z,Z_i \subset F^{(n)}$:
\begin{enumerate}

  \item If $Z_1 \subseteq Z_2$, then $\Ker(Z_2) \subseteq \Ker(Z_1)$.
  \item $\Ker(\bigcup_{i \in I}Z_i) = \bigcap_{i \in I}\Ker(Z_i)$.
  \item $K \subseteq \Ker(\Skel(K))$ for any $\nu$-kernel $K$ of $\overline{F(\Lambda)}$.
  \item $Z \subseteq \Skel(\Ker(Z))$.
\end{enumerate}
\end{rem}

\begin{lem}\label{kersk} If $Z \subset F^{(n)}$, then  $\Ker (Z)$ is a
$\nu$-kernel.
\end{lem}
\begin{proof} It is closed under multiplication, and if $g_1 + g_2 \nucong 1$ then
$(g_1f + g_2 f )(\bfa) = g_1(\bfa) + g_2  (\bfa) \nucong \bfa.$
\end{proof}

\begin{defn}\label{defn_k_kernels}
A \textbf{$\mathcal{K}$-kernel} of $\overline{F(\Lambda)}$ is a
$\nu$-kernel of the form $\Ker(Z)$, where $Z$ is a $\onenu$-set.
\end{defn}

\begin{lem}\label{prop_skel_ker_relations}$ $
\begin{enumerate}
  \item $\Skel(\Ker(Z)) = Z$   for any  $\onenu$-set $Z$.
  \item $\Ker(\Skel(K)) = K$, for any  $\mathcal{K}$-kernel  $K$.
\end{enumerate}
\end{lem}

\begin{proof} This is a standard argument, applying Remark~\ref{rem_ker_properties_1} to the reverse
inclusion of   Proposition \ref{prop_skel_property1}.
\end{proof}

By Lemma~\ref{prop_skel_ker_relations}, we have

\begin{thm}\label{prop_correspondence}
There is a $1:1$, order-reversing correspondence
\begin{equation}
\{ \text{$\onenu$-sets of } \ F^{(n)}  \} \rightarrow \{
\mathcal{K}- \text{kernels  of } \ \overline{F(\Lambda)} \},
\end{equation}
given by $Z \mapsto \Ker(Z)$;  the reverse map is given by $K
\mapsto \Skel(K)$.
\end{thm}

One of the main goals in this paper is to find an intrinsic
characterization of $\mathcal{K}$-kernels, especially the principal
$\mathcal{K}$-kernels;  this is only done in
Corollary~\ref{corlatticegeneratedbyCIkernels}.
%
%

\begin{defn}   
The  \textbf{\skeletonb}\  of a  set $ S \subseteq \FunF$, denoted
$\skel(S)$, is $\bigcap \{\skel(f): f \in S\}.$ (Usually $S$ is
taken to be finite.)
\end{defn}

We need to cut down   the class of \skeletonbs\ for tropical
applications, in view of the following examples.

\begin{exampl} $ $\begin{enumerate}

\item Let $f = \la _1 + \la _2$ and $g=\fone.$

\item Let $f = \la _1 + \la _2 +\fone$ and $g =  \la _1^3 + \la
_2^2 +\fone $.

\item Let $f = \la _1 + \la _2 +\fone$ and $g =  \la _1^2 + \la
_2^2 +\fone $.

\medskip

In each of (1)--(3), $\skel(\frac f g)_\tT= \{ (\a,\fone): \a
\le_\nu \fone\} \cup \{ (\fone,\a): \a \le_\nu \fone\} $ is not a
 variety.

\medskip

 \item Let $f = \la _1 + \la _2 +\fone$ and $g =  \la _1 + \la _2
$. Then $\skel(\frac f g)_\tT = \{ (\a,\a): \a \ge_\nu \fone\} $.
\end{enumerate}
\end{exampl}

\subsection{The coordinate \vsemifield0 of a $\onenu$-set}$ $

\begin{defn}\label{def_coord_semifield_1} For $X \subset F^{(n)},$ The \textbf{coordinate
\vsemifield0} $F(X)$ of a $\onenu$-set $X$ is the set of
restrictions of the rational functions $\overline{F(\Lambda)}$ to
$X$.

$$\phi_{X}~:~  \overline{F(\Lambda)}~\rightarrow~F(X)$$
denotes the restriction map  $h \mapsto h|_{X}$.
\end{defn}

\begin{prop}\label{prop_coord_semifield_100}
$\phi_{X}$ is an onto \vsemifield0 homomorphism.
\end{prop}
\begin{proof} Straightforward verification.
\end{proof}

\begin{prop}\label{prop_coord_semifield_1}
$F(X)$ is a \vdomain0, isomorphic to
$\overline{F(\Lambda)}/\Ker(X).$
\end{prop}
\begin{proof}
The restriction map has $\nu$-kernel equal to those functions $f$
which restrict to 1, which is $\Ker(X).$
\end{proof}

When $X' \supseteq X$ is another $\onenu$-set, further restriction
gives us a \semiring0 homomorphism $F[X]\to F[X'],$ and chains of
these homomorphisms give us an algebraic view of dimension, which
is studied at the end of~\cite{AlgAspTropMath}.

 \section{The transition between tropical varieties and \skeletonbs}\label{supertrop16}

%
%

%


In view of Lemma~\ref{kersk}, we would like to pass back and forth
from tropical varieties to \skeletonbs. This is one of our main
themes.

\subsection{The hat construction}\label{hat1}$ $

We start by passing to the corner kernel locus from the corner
locus obtained from (super)tropical polynomials. Towards this end,
we formulate the following notion. We say that a function $f$
\textbf{dominates} $g$ at $\bfa$ if
 $f(\bfa)\ge_\nu g(\bfa);$ $f$ \textbf{dominates} $g$ if $f$  dominates  $g$ at each
 point $\bfa$.

\begin{defn}\label{datom} A \textbf{molecule} is a rational function
$  hg^*$ where $h$ is a  monomial and $g$ is a polynomial. The
molecule is \textbf{tangible} if $h$ and $g$ are tangible. (In this
case, $ hg^*  = \frac hg$.) Given a polynomial $f \in F[\Lambda] =
\sum_{i=1}^{k}
 f_i $, written as a sum of monomials, define the \textbf{molecules of} $f $
 to be the  molecules
\begin{equation}
\widehat{f} _i =   {f_i}\left(\sum_{j \neq i}f_j\right)^*\in
\overline{F(\Lambda)},
\end{equation}
and \begin{equation} \widehat{f}  =  \sum_{i=1}^{k} \widehat{f}
_i\end{equation}
\end{defn}

Usually $f$ is tangible, in which case we have
\begin{equation}
\widehat{f} _i =  \frac{f_i}{\sum_{j \neq i}f_j}\in
\overline{F(\Lambda)},
\end{equation}

\begin{defn}
Given a polynomial $f = \sum h_i$ written as a sum of monomials, and
selecting one of these monomials $h = h_i$, define $f_{(h)} $ to be
$ \sum _{h_j \ne h} h_j$. Then   $$\hat f = \sum _h hf_{(h)}^*.$$
\end{defn}

\begin{lem}\label{kersk3}
Given $f = \sum h$ written as a sum of monomials, then $\hat f \in
\tT^{+}(\la_1, \dots, \la _n)_\nu.$
\end{lem}
\begin{proof} Write $f = \sum h$ and $\hat f =
\sum _h hf_{(h)}^*.$ If some monomial $h$ dominates at $\bfa,$ then
$$\hat f (\bfa) = hf_{(h)}^*(\bfa) \ge _\nu \fone.$$ Thus, we may assume that two
monomials $g,h$ dominate at $\bfa$, and then $$\hat f (\bfa) =
gf_{(g)}^*(\bfa) + hf_{(h)}^*(\bfa) = \fone^\nu.$$
 \end{proof}

\begin{lem}\label{mol}  (i) $\widehat{f} _i (\bfa) \ge_\nu \fone$ iff $f_i(\bfa)$
dominates $f (\bfa)$.

(ii) $\widehat{f} _i (\bfa)  \nucong \fone$ iff, for some $j \ne i,$
$f_i(\bfa) = f_j(\bfa)$ which dominates each monomial. This means
$f_i(\bfa)$ dominates in $f (\bfa)$,  and $\bfa\in \Corn(f) $.
\end{lem}
\begin{proof} (i) Each side says that the numerator of $\widehat{f} _i$ dominates the
denominator.

(ii)  Clearly $\widehat{f} _i (\bfa) \nucong \fone$ iff, for some $j
\ne i,$ $f_i(\bfa) = f_j(\bfa)$ which dominates each other monomial.
This means $f_i(\bfa)+ f_j(\bfa) = f(\bfa)^\nu$, implying $f(\bfa)=
f(\bfa)^\nu \in \tG.$
\end{proof}

\begin{prop}\label{mol2} The following conditions are equivalent:
\begin{enumerate}
\eroman
\item $\widehat{f}  (\bfa) =  \fone^\nu;$
\item $\widehat{f}  (\bfa) \nucong \fone$;
\item$\bfa\in \Corn(f). $
\end{enumerate}
\end{prop}\begin{proof} (i) $\rightarrow$ (ii). Obvious.

(ii) $\rightarrow$ (iii). By Lemma~\ref{mol}, there are two
monomials $f_i$ and $f_j$ that dominate at $\bfa$, and so $$
\fone^\nu = \widehat{f} _i (\bfa) + \widehat{f} _j (\bfa),$$ which
then is $\widehat{f} (\bfa).$

(iii) $\rightarrow$ (i). $\widehat{f_i}  (\bfa) = \widehat{f_j}
(\bfa) \nucong   \fone^\nu,$ so $$\fone^\nu = \widehat{f_i} + (\bfa)
= \widehat{f_j} (\bfa)= \widehat{f} (\bfa).$$
\end{proof}

\begin{cor}\label{prop_ker_corner1} $\Skel( \widehat{f}) =  \Corn(f)$.
\end{cor}

A \textbf{principal corner locus} is a set of ghost roots of a
supertropical polynomial.

\begin{cor}\label{prop_ker_corner}
Any principal corner locus  is a principal $\onenu$-set. 
\end{cor}

\begin{rem}\label{rem_corner_to_skel}
For $S \subset \mathcal{F}(F[\Lambda])$, let $\widehat{S} = \{
\widehat{f}  \ : \ f \in S \} \subseteq \overline{F(\Lambda)}$.
Then by Remark~\ref{rem_intersection_union_corner_loci} and
Proposition \ref{prop_skel_property1},
$$\Corn(S) = \bigcap_{g \in \widehat{S}}\Skel(g) = \Skel(\widehat{S}).$$
Thus, the map $\Corn(f) \mapsto \Skel(\widehat{f})$  extends to a
map
$$\Phi:~\Corn(F^{(n)} )~\rightarrow~\Skl(F^{(n)} ),$$
where $\Phi : \Corn(f) \mapsto \Skl(\widehat{f})$. In particular,
taking only finitely generated corner loci,
 and recalling that finite intersections  and unions of principal
 $\onenu$-sets are  principal $\onenu$-sets, $\Phi$ sends every
 finitely generated corner locus to a principal $\onenu$-set.
 %
\end{rem}

\begin{lem}\label{lem_ker_prop}\cite[Lemma (10.2.5)]{AlgAspTropMath}
Let $f = \sum_{i=1}^{k}f_i \in F[\Lambda]$.  Then for $1 \leq i,j
\leq k$ such that $i \neq j$ and $\widehat f_i(\bfa) \nucong
\widehat f_j(\bfa)$,
\begin{equation}\label{eq_lem_ker_prop}
\text{Either}\qquad  \widehat f_i(\bfa) \nucong \widehat f_j(\bfa)
\nucong 1 \  \qquad \text{or} \qquad \ \widehat f_i(\bfa) \nucong
\widehat f_j(\bfa)\ \text{are dominated  by}
 \ \widehat{f} \text{ at }\bfa.
\end{equation}
\end{lem}
 \begin{proof} Take the two cases, where either $f_i$ and $f_j$ both dominate, or neither  dominates. \end{proof}

\begin{lem}\label{kersk1} For  $f,g \in F[\Lambda]$ with $g$ tangible, over a  \vsemifield0 $F$
(where $\Lambda = \{ \la _1, \dots, \la _n
\})$, the union of $\skel( \frac f g)$ with the corner loci of  $f$  and
$g$, is  the corner locus of $f+g$.\end{lem}
\begin{proof} $ \frac f g(\bfa) \nucong
\fone$   iff $f(\bfa) \nucong g(\bfa),$ which happens when
$f(\bfa) + g(\bfa) \in \tG.$
\end{proof}

\begin{prop}\label{kersk2}
An element $\bfa\in F^{(n)}$ is a $\nu$-kernel root of $\hat f$
iff $\bfa$ is a ghost root of $f$.
\end{prop}
\begin{proof} Apply Lemma~\ref{kersk1} taking $g = f_{(h)}.$
 \end{proof}

We can state this more explicitly.

\begin{prop}\label{hatcomp}  Suppose $f  = \sum _{i=1}^t h_i \in
\overline{F[\la_1, \dots, \la _n]}$, written as a sum of
monomials.
   Then
    $$\hat f = f^t \left(\prod _{i=1}^t f_{h_i}\right)^*.$$
 \end{prop}
\begin{proof} (i) Since $\hat f = \sum  _i h_i f_{h_i}^*  =  (\sum
_i h_i\prod_{j\ne i} f_{h_j})(\prod _j f_{h_j})^*,$   it suffices to
prove that $\sum _i h_i\prod_{j\ne i} f_{h_j} = f^t.$ But each
$h_i^t$ appears in both sides, and every other product of length $t$
of the $h_j$ appears as a ghost in both sides.
\end{proof}

We get Proposition~\ref{kersk2} as a consequence, since the kernel
roots must be precisely those $\bfa$ for which each $\prod
f_{h_i}(\bfa) \nucong f(\bfa),$ which are the ghost roots.

 \begin{rem}\label{hatcomp1} Applying Proposition~\ref{hatcomp} to $\hat f_{(h_i)}$ and observing
that $f$ is dominated at each ghost root  in two monomials, by
definition, we see that  the $\onenu$-set of $\hat f_{(h_i)}$ is
precisely the $\onenu$-set of $f_{(h_i)} ^{t-1}(\prod _{j \ne i}
f_{h_j})^*.$ From this point of view, we can cut down one summand
when passing to $\onenu$-sets.
 \end{rem}

\begin{rem}\label{bend} The hat construction
 in Perri's dissertation \cite{AlgAspTropMath} is related to the bend congruences
  of \cite{GG,Mac,MacRin} (up to $\nu$-equivalence), although with the proviso that
  these authors
  deal with polynomials per se, whereas at the outset we view polynomials as functions.
 (This gap  is seen to be  bridged by  region fractions,
 which produce \textbf{region
 $\nu$-kernels}, cf.~Definition~\ref{regker},   enabling us to identify
 those $\nu$-kernels which are trivial in a given neighborhood, thereby
 indicating which local information may be lost, as explained in the decomposition of  Theorem~\ref{thm_HP_expansion}.)

Namely, take a polynomial $f = \sum_i f_i = h + f_{(h)}$ where
$h=f_j$ is one of the $f_i$
 dominating at $\bfa$ and $f_{(h)}: = \sum_{i\ne j} f_i;$  the \textbf{bend congruence} is
 defined to be the congruence in
 $F[\Lambda]$
 generated by $(f,f_{(h)}).$ But when viewed as a function, this could be matched with the
 principal $\nu$-kernel defined by the
 rational function $$\frac f{f_{(h)}} = \sum_i \frac {f_i}{f_{(h)}} \nucong \frac
 {h}{f_{(h)}}\nucong \hat f _j,$$ and we can reverse directions.
 \end{rem}

\subsection{Examples of $\onenu$-sets}\label{onenusetex0}$ $

\begin{exmp}\label{onenusetex} In these examples, we write $\bfa = (a_1, \dots, a_n).$

 (i)  Take the tropical line
 $f = \la_1  \dotplusss   \la_2  \dotplusss   1$. Its corresponding
$\onenu$-set is defined by the rational function
\begin{equation}\label{standline1}\widehat{f} = \frac{\la_1}{\la_2  \dotplusss
1}  \dotplusss   \frac{\la_2}{\la_1 \dotplusss 1} \dotplusss
\frac{1}{\la_1  \dotplusss   \la_2} = \frac{(\la_1 \dotplusss \la_2
\dotplusss   1)^3}{(\la_1 \dotplusss 1) (\la_2 \dotplusss 1)(\la_1
\dotplusss \la_2)}.\end{equation}
 Moreover, in
Remark~\ref{hatcomp1} it is shown that any of the three terms above
can be omitted, and for example we could use instead
\begin{equation}\label{standline2}\frac{\la_1}{\la_2  \dotplusss   1}  \dotplusss   \frac{\la_2}{\la_1
\dotplusss 1} = \frac{\la_1  ^2 + \la _2 ^2 + \la _1+ \la _2}{(\la_1
\dotplusss   1) (\la_2 \dotplusss 1)}.\end{equation}

Although these two functions \eqref{standline1} and
\eqref{standline2} define the same $\onenu$-set, they define
different $\nu$-kernels, since $\la_1  \dotplusss   \la_2$ is not
in the $\nu$-kernel generated by $\frac{\la_1}{\la_2  \dotplusss
1} \dotplusss \frac{\la_2}{\la_1 \dotplusss 1}$. Indeed, if we
take the point $\bfa = (\alpha, \beta)$ for $\alpha < \beta <1,$
we get $\frac  1 \alpha
 $ in \eqref{standline1} but $\beta $ in \eqref{standline2}, and
 their ratio can be whatever we want, so the condition of
 Proposition~\ref{prop_principal_ker} fails. This ambiguity
 motivates much of the   theory espoused later in this paper.

(ii) The ghost roots of $a_2 \le_\nu a_1 \nucong 1$ and $a _1
\le_\nu a_2 \nucong 1$ provide two rays, but we are missing the ray
$1 \le_\nu a _1 \nucong  a_2.$

(iii) The $\onenu$-set of the rational function
  $\frac{1}{\la_1  \dotplusss   \la_2} +   \frac{\la_1}{\la_2} + \frac{\la_2}{\la_1}$
  is comprised only of the ray $1 \le_\nu a _1 \nucong a_2.$

 (iv) Next, we take the tropical line
 $g = \la_1  \dotplusss   \la_2  \dotplusss   2$. Its corresponding
$\onenu$-set is defined by the rational function $\widehat{g} =
\frac{\la_1}{\la_2  \dotplusss   2}  \dotplusss   \frac{\la_2}{\la_1
\dotplusss 2} \dotplusss   \frac{2}{\la_1  \dotplusss   \la_2}$.

(v) We take the   corner locus defined by the common ghost roots of
$f$ and $g$, which correspond to the $\onenu$-set defined by the
rational function $ \widehat{f} + \widehat{g}$. Taking the three
dominant monomials yields $\frac{\la_1}{\la_2 \dotplusss 1}
\dotplusss \frac{\la_2}{\la_1 \dotplusss 1} \dotplusss
\frac{2}{\la_1 \dotplusss \la_2}$, which again cannot come from a
single polynomial.
\end{exmp}

\subsection{From the $\nu$-kernel locus to  hypersurfaces}$ $

The other direction is more straightforward.

\begin{defn} Given   $f = \frac{h }{g}\in \overline{F(\Lambda)}$ for $g,h\in \overline{F[\la_1, \dots, \la _n]},$ define $\underline{f}
= g+h.$
\end{defn}

\begin{rem}\label{OK1}
Conceivably, one could have $f = \frac{h}{g}= \frac{h'}{g'}\in
\tT^{+}(\la_1, \dots, \la _n)_\nu$, so the definition technically
depends on how we define the representation. But in this case,
${h}g' = gh'.$ Thus, any  ghost root $\bfa$  of $h$   is a ghost
root of $g$ or $h',$ so we get the same ghost roots except when both
the numerator and denominator have a common ghost root.
\end{rem}

\begin{thm}\label{cor110}  Suppose $f  \in \overline{F[\la_1, \dots, \la _n]}$.
   Then $$\Corn(\underline{\hat f}) =  \Skel(\hat f) = \Corn(f).$$\end{thm}
\begin{proof} Using Proposition~\ref{hatcomp} and its proof, we can compute
$\underline{\widehat f}$ explicitly, as $f^t +\prod _{i=1}^t f_{h_i}
= f^t $ (since each $f_{h_i} \le_\nu f,$ and $\prod _{i=1}^t
f_{h_i}$ does not contain any $h_i^t$, which are the only tangible
monomials of $f^t$).
\end{proof}

In this way, we can pass from the corner locus of a single
polynomial to a 1-set, and back.
  This procedure breaks down for tropical varieties defined by more than one
polynomial, as seen in Example~\ref{onenusetex}(v).

Going the other way is trickier. Suppose $f = \frac hg.$ Then
  $\widehat { (\underline{f})}= \widehat { g+h} $.

  \begin{exampl}\label{ci10}
Take  $f = \frac{h }{g}\in \overline{F(\Lambda)}$, for $g=  \la _1
$ and $ h =
 2 \la _1 + 4$. Then $\underline{f} = h$ has the ghost root 2, which is not in
the 1-set of~$f$S.
\end{exampl}

\section{Distinguishing the canonical tropical varieties}\label{tropapp1}

By ``canonical'' tropical variety we mean the kind of variety that
arises in the usual study of tropical geometry,
cf.~\cite{IntroTropGeo}, although there really is no truly
``canonical'' definition. One possibility is the definition of
variety given above. Our first objective is to indicate how to
obtain these via $\nu$-kernels. In this section, following Perri's
dissertation, \cite{AlgAspTropMath}, we present a method for
describing a canonical tropical variety in terms of a single
rational function, thereby enabling us to pass from corner loci to
$\onenu$-sets  which often are principal. This sets the path for
applying the study conducted in the previous sections to tropical
geometry. We begin our discussion by introducing the notion of
``corner internal'' $\onenu$-sets, which characterize the
$\onenu$-sets and $\nu$-kernels arising from
 tropical varieties. Then we bring in ``regular'' rational functions, which
 avoid degeneracy in the variety.


\subsection{Corner internal functions}\label{admis1}$ $

First we want to know which $\nu$-kernels come from corner loci.
Example~\ref{ci10} shows that we can have $f = \frac hg,$ where
$\bfa$ is a ghost root of $h$ and thus of $\underline{f} = h+g,$
but having $h(\bfa)
>_\nu g(\bfa)$, and thus $\bfa$ is not in the 1-set of~$f$.  Our next definition is designed to exclude this
possibility.

\begin{defn}\label{defn_corner_admissible}
  A rational function $f \in \overline{F(\Lambda)}$ is \textbf{corner internal} if we can write  $f =   \frac hg $
  for polynomials $g,h$, with $g \in \overline{\tT[\Lambda]},$  such that every
ghost root of $g+h$ is a kernel root of $f$.
\end{defn}

In other words, if $g(\bold a) \le_\nu h(\bold a) \in \tG$ or
$h(\bold a) \le_\nu g(\bold a) \in \tG$, then
 $g(\bold a)
 \nucong h(\bold a)$.

\begin{rem}\label{corner_int_cond_1} Suppose $f =   hg^* $. Any ghost root of $g+h$ is
either a ghost root of $h$, or of $g$, or else satisfies $g(\bfa)
\nucong h(\bfa)$ in which case it is automatically in $\Skel(f)$.
Thus,
 corner internality is equivalent to showing that $\Corn(g)\cup
\Corn(h) \subseteq \Skel(f).$
\end{rem}


%
%
%

%

\begin{rem}
When considering whether $f \in \overline{F(\Lambda)}$ is  corner
internal, the choice of $g$ and $h$ is important, which we call
the \textbf{canonical} way of writing $f$.  For example, $\la +1
\in F(\la)$ is trivially corner internal, although $\frac{(\la+
\alpha)(\la +1)}{\la+\alpha}$ for $\alpha >_\nu 1$ is not - since
substituting $\alpha$ for $\la$ we get $\frac{\alpha^2 +
\alpha^2}{\alpha + \alpha} = \frac{(\alpha^\nu)2}{\alpha^\nu)} $.
Thus $\alpha$ is a ghost root of the numerator, which surpasses
the denominator since $(\alpha^2)^{\nu}
>_\nu \alpha^{\nu}$.

Also note that $g+h$ must be tangible, since otherwise some open set
around $\bfa$ is in $\Corn(g+h)$ but not in $\Skel(f).$ Thus, we
write  $f = \frac hg $ instead of $f =   hg^* $.  In particular, $f$
must be reduced as a fraction.

\end{rem}

Even though the canonical way of writing $f = \frac hg$ may not be
unique, the following observation is enough for our purposes.

\begin{lem}\label{welld} Suppose $f = \frac hg =  \frac {h'}{g'} $ are two
canonical ways of writing a corner internal rational function $f$.
Then $\Corn (h+g) = \Corn (h'+g') .$
\end{lem}
\begin{proof}  If $\bfa \in \Corn (h+g),$ then
either $g(\bfa) \nucong h(\bfa)$ and we are done, or say $\bfa \in
\Corn(g).$ But then $h (\bfa) \le_\nu g (\bfa),$ implying $h'(\bfa)
\le_\nu g'(\bfa),$ and also $\bfa \in \Corn(hg')$ since $hg'  =
h'g.$ We are done if $\bfa \in \Corn(g'),$ so we may assume that
$\bfa \in \Corn(h)$. Hence $h (\bfa) \le_\nu g(\bfa)\in \tG,$
implying $g (\bfa) \nucong h(\bfa),$ and again  we are done.
\end{proof}

Thus, $\Corn(\underline{f} )$ does not depend on the canonical way
we write the corner internal rational function $f$.

\begin{prop}\label{cor_absolute_corner_admissible2} A rational function $f$
is corner internal iff $\Corn(\underline{f} ) = \Skel(f),$ for any
canonical way of writing $f = \frac hg$.
\end{prop}
\begin{proof} $(\Leftarrow)$ is by definition.

 $(\Rightarrow)$ Any kernel root $\bfa$ must satisfy $g(\bfa) \nucong
 h(\bfa),$ implying $\bfa$ is a ghost root of $g+h.$
\end{proof}

\begin{lem}\label{ci1} For any polynomial $f\in F[\Lambda],$ the function $\hat f$  is
corner internal. \end{lem}
\begin{proof} $\Corn(\underline{\hat f} ) = \Skel(\hat
f)$ in view of Theorem~\ref{cor110}.\end{proof}

\begin{thm}\label{cor11} The correspondences $f \mapsto \hat f$ and $h \mapsto
\underline{h}$ induce a 1:1 correspondence between
hypersurfaces and $\onenu$-sets of corner internal rational
functions.
\end{thm}
\begin{proof} Combine
Proposition~\ref{cor_absolute_corner_admissible2} and
Lemma~\ref{ci1}.
\end{proof}

Having established the importance of being corner internal, let us
delve deeper into the elementary properties.

 By symmetry, a rational
function $f \in \overline{F(\Lambda)}$ is corner internal if and
only if $f^{-1} $ is corner internal.

\begin{rem}\label{writeabs} If $f = \frac hg, \ f' =  \frac {h'}{g'} \in \overline{F(\Lambda)} $, then
$$f+f' =  \frac {hg' + gh'}{gg'},$$ and
 $$|f| = f+f^{-1} = (h^2 + g^2)(gh)^*. $$
  By Proposition~\ref{Fr1}, any ghost root of ${h^2 + g^2}$ or of $h^2 + g^2 +gh \nucong (h+g)^2$  is a ghost root of
  $g+h$.
\end{rem}

\begin{prop}\label{prop_absolute_corner_integral}
A rational function $f \in \overline{F(\Lambda)}$ is corner
internal if and only if $|f| $ is corner internal. \end{prop}
\begin{proof} $|f|$ and $f$ have the same $\onenu$-sets.

$(\Rightarrow)$ Write $f = \frac hg $ canonically. Then $|f| =
h{g}^* +  g{h}^* =(h^2 + g^2)(gh)^*$. By Remark~\ref{writeabs}, any
ghost root $\bfa$ of $ g^2+ h^2 $ is a ghost root of $g+h$ and thus
a kernel root of $f$, by definition, or else $g(\bfa) = h(\bfa)$,
yielding a kernel root of $f$. But then $\bfa \in \Skel(|f|).$

$(\Leftarrow)$ If $\bfa$ is a ghost root of $g+h$, then, again by
Remark~\ref{writeabs}, $\bfa$ is a ghost root of $g^2 + h^2 + gh$
and thus by hypothesis is in $\Skel(|f|) = \Skel(f).$
\end{proof}

\begin{lem}\label{rem_corner_admissibility_of_an_expansion}
If $f \in \overline{F(\Lambda)}$ is corner internal, then
$f^{k}$, for any $k \in \mathbb{N}$ and $\sum_{i=1}^{m}f^{d(i)}$
with $d(1)< \cdots < d(m)$ in $\mathbb{N}$, also are corner
internal.
\end{lem}
\begin{proof}
Everything follows easily from Remark~\ref{writeabs}.
\end{proof}

We recall $f \wedge g$ from \eqref{dua01} after Example
\ref{vloc31}.

\begin{rem}
For intuition, in view of Remark~\ref{ord1},  $ (f\wedge g )(\bfa)
\nucong \min \{ f(\bfa)^\nu, g(\bfa)^\nu \}.$
\end{rem}

\begin{spacing}{1.1}
\begin{prop}\label{prop_wedge_corner_admissibility}
If $f,f'\in \overline{F(\Lambda)}$ are corner internal, then $ |f|
\wedge |f'|$ is corner internal.
\end{prop}
\begin{proof}
Since $f$ and $f'$ are corner internal, $|f|$ and $|f'|$ also are
corner internal. Write $|f| = \frac{h}{g}$ and $|f'| =
\frac{h'}{g'}$ in the canonical ways. Now,
\begin{equation}\label{f0} 1 \le_\nu |f| \wedge |f'| =  (gh^* +  g'(h')^*)^* =  {hh'}(gh'+g'h)^*,\end{equation} so
 \begin{equation}\label{f1}
h(\bfa)h' (\bfa) \ge _\nu (gh' + g'h)(\bfa) .\end{equation}  We need
to show that any ghost root $\bfa$ of $hh' +gh' + g'h$ satisfies
$$(hh' +gh' + g'h)(\bfa) \nucong  h(\bfa)h'(\bfa).$$ $\le_\nu$ follows from \eqref{f1},
so we need to show $\ge_\nu$. Clearly $\bfa$ is a ghost root of
$hh'+gh',$ $hh'+g'h,$ or $gh' + g'h$. In the latter case, $h(\bfa)h'
(\bfa) \le _\bfa (gh' + g'h)(\bfa)$, and we are done. Thus by
symmetry we may assume that  $\bfa \in \Corn(hh'+gh')
=\Corn((g+h)h').$ If $\bfa \in \Corn(g+h)$ then $\bfa \in
\Skel(|f|)$ and thus in $\Skel(|f| \wedge |f'|).$ So we are done
unless $\bfa \in \Corn(h')$.  If $h'(\bfa) \ge_\nu g'(\bfa),$ then
$\bfa \in \Corn(g'+h') = \Skel(f)$ and we are done. Thus we may
assume that $h'(\bfa) \le_\nu g'(\bfa),$ implying $g'(\bfa)h(\bfa)
\ge_\nu h'(\bfa)h(\bfa),$ and again we are done.
\end{proof}
\end{spacing}

\begin{cor}\label{repl1}
Let $f \in \overline{F(\Lambda)}$. Then $f$ is corner internal if and only if $|f| \wedge |\alpha|$ is corner internal for any $\alpha \neq 1$ in $F$. \\
\end{cor}
\begin{proof} Write $f = \frac hg.$
Since $|\alpha| >1$, the ghost roots of $ |\alpha|h+|\alpha|g + h$
and of $|\alpha|(g+h)$ are the same, and thus the kernel roots of
$|f|$ and
 of $|f| \wedge |\alpha| $ are the same.
\end{proof}

\begin{prop}\label{prop_regular0}
There is a $1:1$ correspondence between principal corner internal
$\onenu$-sets and principal corner-loci.
\end{prop}
\begin{proof} Take $f,f' \in \overline{F(\Lambda)}$   corner internal.

If $\Skel(f')= \Skel( f)$, then
$$\Skel(f') = \Skel(f) = \Corn(\underline{f}) = \Skel(\widehat{\underline{f}}).$$

 Conversely, if
$\Corn(f') = \Corn(f)$, then
$$\Corn(f') = \Corn(f) = \Skel(\hat{f}) = \Corn(\underline{\hat{f}}),$$
by Theorem~\ref{cor110}.
\end{proof}

\subsection{Corner internal $\nu$-kernels}\label{admis}$ $

\begin{defn}\label{defn_corner_admissibility}
A principal $\nu$-kernel $K$ of $\overline{F(\Lambda)}$  is said
to be \textbf{corner internal} if   it has a corner internal
generator. In this case, the $\onenu$-set $\Skel(K)$ corresponding
to~$K$ is said to be a \textbf{corner internal} $\onenu$-set.
\end{defn}

\begin{cor}\label{cor_intersection_of_CI}
Any finite intersection of principal corner internal $\nu$-kernels
is a principal corner internal $\nu$-kernel.
\end{cor}
\begin{proof}

By induction, it is enough to show that if $K $ and $K'$ are
principal corner internal   $\nu$-kernels, then so is $K  \cap
K'$. Write $K = \langle f \rangle$ and $K' = \langle f' \rangle$.
By Proposition~\ref{prop_absolute_corner_integral} we may assume
that $f,f' \geq 1$. By
Proposition~\ref{prop_wedge_corner_admissibility}, $f \wedge f'$
is corner internal, which generates $K  \cap K'$.
\end{proof}

%

\begin{rem}
If  $f,g \in \overline{F(\Lambda)}$ are corner internal then $|f|
+ |g|$ need not be corner internal. Thus the collection of corner
internal $\nu$-kernels is not a lattice. In our study we thus take
the lattice \textbf{generated} by principal corner internal
$\nu$-kernels. These elements will be shown to correspond to
finitely generated corner loci that are not necessarily principal.
\end{rem}

\subsection{The hat-construction for corner internal
$\nu$-kernels}$ $

Although the $f\mapsto \hat{{f}}$ correspondence given above
yields a fast and effective correspondence from corner loci of
polynomials to $\onenu$-sets, it does not work so well on
arbitrary rational functions $f \in \overline{F(\Lambda)}$, so we
turn to a subtler but more thorough correspondence, which
``explains'' what makes a $\nu$-kernel corner internal. We start
with the special case of a polynomial $f = \sum_i f_i \in
F[\Lambda],$ written as a sum of monomials $f_i$.

\begin{defn}\label{wedge}  Given $f = \sum_i f_i \in F[\Lambda],$ define the rational function  $\tilde{f}~=~ \bigwedge_{ i=1}^{\ k} |\widehat {f_i}|. $\end{defn}
 \begin{prop}    
 $\Corn(f) = \Skel(\tilde{f})$.
 \end{prop}
\begin{proof} We have seen in Theorem~\ref{cor110}
that $\bfa \in \Corn(f)$ iff $\widehat {f_i}(\bfa) = 1$ for some
$i$. But each  $|\widehat {f_i}|(\bfa) \ge 1,$ so we conclude with
Lemma~\ref{rem_max0}
\end{proof}
Thus, we have an alternative approach to that of
 Theorem~\ref{cor110},  motivating some of the intricate
computations we are about to make.

%
%

%
%
%

\begin{thm}\label{prop_CI_conditions} Suppose
$f = h g^* \in \tT(\Lambda)$ is a  rational function, where $h =
\sum_i h_i$ and $g = \sum _j g_j$ written as sums of monomials. Then
$f$ is corner internal if and only if the following conditions hold:
\begin{equation}\label{eq_prop_CI_conditions_1}
\left(\bigcup_{i=1}^{k}\Skel\left(\widehat{h_i}\right)\right) \cap
\Skel(f^* + 1) \subseteq \Skel(f)
\end{equation}
\begin{equation}\label{eq_prop_CI_conditions_2}
\left(\bigcup_{j=1}^{m}\Skel\left(\widehat{g_j}\right)\right) \cap
\Skel(f + 1) \subseteq  \Skel(f).
\end{equation}
\end{thm}

\begin{proof} Note that $f +1 =  (g+h )g^*,$ whereas $f^* +1 = (g+h)h^*.$  In view of Remark~\ref{corner_int_cond_1}, we need
to show that $\Corn(h) \subseteq \Skel(f)$ is equivalent to
\eqref{eq_prop_CI_conditions_1}, since the other condition is
symmetric (with respect to exchanging $g$ and $h$). By definition of
corner internality, any ghost root $\bfa$ of $h$ must satisfy
$h(\bfa) \le_\nu g(\bfa),$ implying $\bfa \in \Skel(f+1).$ Thus,
checking this at each ghost root, we have
$\Skel\left(\widehat{h_i}\right) \subseteq  \Skel(f + 1)$ for each
$i$, or equivalently,
\begin{equation}\label{eq_inclusion_1}
 \langle f + 1 \rangle \subseteq \left\langle \widehat{h_i}
\right\rangle.
\end{equation}

The steps are reversible.

%

Now, intersecting both sides  with $\Skel(f^* +1)$ yields
$$ \Skel\left(\widehat{h_i}\right) \cap \Skel(f^* + 1) \subseteq  \Skel(f + 1) \cap \Skel(f^* +
1).$$ Note that this step also is reversible since $\Skel(f^* +1)
\cup \Skel(f+1) = F^{(n)}$.

Passing again from \skeletonbs\  to $\nu$-kernels yields
\begin{equation*}
\langle f \rangle \subseteq \left\langle \widehat{h_i} \right\rangle
 \langle f^* + 1 \rangle.
\end{equation*}

Moreover, since the above inclusion holds for every $i \in \{
1,...,k \}$ we conclude that
\begin{equation*}
\left(\bigcup_{i=1}^{k}\Skel\left(\widehat{h_i}\right)\right) \cap
\Skel(f^* + 1) \subseteq \Skel(f),
\end{equation*}
as desired.

%
%
\end{proof}

\begin{rem}\label{rem_alternative_expression}
Note that
$$\bigcup_{i=1}^{k}\Skel\left(\widehat{h_i}\right) =
\Skel\left(\bigwedge_{i=1}^{k} \left|\widehat{h_i}\right|\right) =
\Skel(\tilde{h})$$ and similarly
$$\bigcup_{j=1}^{m}\Skel\left(\widehat{g_j}\right) =
\Skel(\tilde{g}).$$ Thus we can rewrite
\eqref{eq_prop_CI_conditions_1} and \eqref{eq_prop_CI_conditions_2}
as
$$\Skel\left( \tilde{h} \right) \cap  \Skel(f^* + 1) \subseteq \Skel(f) \qquad \text{and} \qquad  \Skel\left( \tilde{g} \right) \cap \Skel(f + 1) \subseteq  \Skel(f),$$
or as $$\Skel\left( \widehat{h} \right) \cap  \Skel(f^* + 1)
\subseteq \Skel(f) \qquad  \text{and} \qquad  \Skel\left(
\widehat{g} \right) \cap \Skel(f + 1) \subseteq  \Skel(f).$$ We
conclude that $\widehat{f}$ is corner internal for any $f \in
\overline{F(\Lambda)}$.
\end{rem}

\begin{rem}
In view of Theorem~\ref{prop_CI_conditions}, given $f =
\frac{h}{g} \in \overline{F(\Lambda)}$, in order to obtain a
corner internal rational function whose \skeletona\ contains
$\Skel(f)$ one must adjoin both
$$\left(\bigcup_{i=1}^{k}\Skel\left(\widehat{h_i}\right)\right)
\cap  \Skel(f^* + 1)$$ and
$$\left(\bigcup_{j=1}^{m}\Skel\left(\widehat{g_j}\right)\right) \cap \Skel(f + 1)$$ to the \skeletona\ of $f$.
\end{rem}

 writing

\begin{defn}\label{hat2}
Define the map $\Phi_{CI} : \overline{F(\Lambda)} \rightarrow
\overline{F(\Lambda)}$ by
\begin{equation} \Phi_{CI}(f) = |f| \wedge \left(|f^* +
1| + \tilde{h} \right) \wedge \left(|f + 1| + \tilde{g} \right),
\end{equation} where $f = \frac{h}{g}$.
\end{defn}

\begin{rem}\label{newr}
$ \Phi_{CI}(f) = \left(|f^* + 1| + \left(|f| \wedge \tilde{h}
\right)\right) \wedge \left(|f + 1| + \left(|f| \wedge
\tilde{g}\right)\right),$  since $|f^* + 1|, |f+1| \leq_\nu |f|$.
This is the fraction whose \skeletona\ is formed by adjoining all
the necessary points to $\Skel(f)$ to obtain corner internality.
Hence,\
$$\langle \Phi_{CI}(f) \rangle = \left(\langle f^* + 1 \rangle
\left(\langle f \rangle \cap \bigcap_{i=1}^{k} \left\langle
\widehat{h_i} \right\rangle\right)\right)  \cap \left(\left(\langle
f + 1 \rangle  \left(\langle f \rangle \cap \bigcap_{j=1}^{m}
\left\langle \widehat{g_j} \right\rangle\right)\right)\right).$$
\end{rem}

\begin{cor}\label{prop_CI_functor}
Let $f = \frac{h}{g} \in \overline{F(\Lambda)}$ be a rational
function, where $h = \sum_{i=1}^{k}h_i$ and $g =
\sum_{j=1}^{m}g_j$  are written as sums of
 monomials in $F[x_1,...,x_n]$.
  Then
\begin{equation}\label{eq_prop_CI_functor}
\Skel(\widehat{h+g})   =   \Skel(f) \nonumber
  \cup  \ \left(\left(\bigcup_{i=1}^{k}\Skel\left(\widehat{h_i}\right)\right)  \cap \Skel(f^* + 1)\right) \
  \cup  \ \left( \left(\bigcup_{j=1}^{m}\Skel\left(\widehat{g_j}\right)\right) \cap \Skel(f + 1)
  \right).
\end{equation}
Thus $\Skel(\widehat{h+g}) = \Skel(\Phi_{CI}(f))$.
\end{cor}
\begin{proof} Apply Lemma~\ref{corner_int_cond_1} and Remark \ref{rem_alternative_expression} to the theorem, to
get the three parts.
\end{proof}

\begin{thm}\label{cor_CI_map1}
If $f \in \overline{F(\Lambda)}$, then  $\langle \Phi_{CI}(f)
\rangle$ is corner internal, and  $\Skel(\Phi_{CI}(f)) \supseteq
\Skel(f)$. Furthermore, $\Skel(\Phi_{CI}(f)) = \Skel(f)$ if and
only if $f$ is corner internal.
\end{thm}
\begin{proof}
The first claim follows from Corollary~\ref{prop_CI_functor}, 
where $\widehat f$ is corner internal by Theorem~\ref{cor110}. The
second claim is straightforward from Remark~\ref{newr} since $$
\Skel\left(|f| \wedge \left(|f^* + 1| + \tilde{h} \right) \wedge
\left(|f + 1| + \tilde{g} \right)\right)  =
 \Skel(f) \cup  \Skel\left(|f^* + 1| + \tilde{h} \right)
\cup \Skel\left(|f + 1| + \tilde{g} \right).
$$
The last statement follows Theorem~\ref{prop_CI_conditions}.
\end{proof}
%

%
\begin{prop}\label{rem_corner_integrality_powers}
For any $f =\frac{h}{g} \in \overline{F(\Lambda)}$,
\begin{equation}
\langle \Phi_{CI}(\sum_{i=1}^{k}f^{d(i)}) \rangle = \langle
\Phi_{CI}(f) \rangle
\end{equation}
whenever  $0 < d(1) < \cdots < d(k),$ and
\begin{equation}
\langle \Phi_{CI}(f^k) \rangle = \langle \Phi_{CI}(f) \rangle
\end{equation}
for any $k \in \mathbb{Z} \setminus \{0\}$.
\end{prop}
\begin{proof}
By the Frobenius property,   $\sum_{i=1}^{k}f^{d(i)} =
\frac{\sum_{i=1}^{k}h^{s + d(i)}g^{t - d(i)}}{h^{s}g^{t}}=
\frac{h^{s+t} + g^{s+t}}{h^{s}g^{t}}$ where $t=|d(k)|$ and
$s=|d(1)|$. So
$$\Phi_{CI}\left(\sum_{i=1}^{k}f^{d(i)}\right) = \Phi_{CI}\left(\frac{h^{s+t} + g^{s+t}}{h^{s}g^{t}}\right) = \widehat{h^{s+t} + g^{s+t} + h^{s}g^{t}} = \widehat{h^{s+t} + g^{s+t}} = (\widehat{h + g})^{s+t}.$$
(The hats are over the enitre expressions.) But $\langle
(\widehat{h + g})^{s+t} \rangle = \langle \widehat{h + g} \rangle
= \langle \Phi_{CI}(f) \rangle$, since
$\Phi_{CI}(\frac{h^{k}}{g^{k}}) = \widehat{h^{k} + g^{k}}=
(\widehat{h + g})^{k}$. Thus, $\langle (\widehat{h + g})^{k}
\rangle = \langle \widehat{h + g} \rangle = \langle \Phi_{CI}(f)
\rangle$.
\end{proof}

\begin{cor}
Let $f \in \overline{F(\Lambda)}$ be such that  $f = u_1 \wedge
\dots \wedge u_k$ where $u_1,...,u_k \in \overline{F(\Lambda)}$
are corner internal. Then $f$ is corner internal and
\begin{equation}\label{eq_phi_wedge_similarity}
\Skel(\Phi_{CI}(f))= \Skel( \Phi_{CI}(u_1) \wedge \dots \wedge
\Phi_{CI}(u_k)).
\end{equation}
\end{cor}
\begin{proof}
$f$ is corner internal, by Corollary~\ref{cor_intersection_of_CI}.
 Replacing $f$ by $|f|$, we may we assume that $f \geq 1$ (and
 thus each $u_i \ge 1$. Since $u_i$ is
corner internal, $ \Skel(\Phi_{CI}(u_i))=  \Skel( u_i)$ for
$i=1,...,k$, and thus $$ \Skel(\Phi_{CI}(u_1) \wedge \dots \wedge
\Phi_{CI}(u_k)
)=  \Skel( u_1 \wedge \dots \wedge u_k) =  \Skel(f).$$ 
\eqref{eq_phi_wedge_similarity} holds  since $f$
is corner internal. 
\end{proof}

%
%

\subsection{Regularity}\label{sym}$ $

In the standard tropical theory, varieties have the property that
their complement is dense, whereas in the supertropical theory the
root set say of $\la^2 + 2^\nu \la + 3$ contains the closed
interval $[1,2]$. We would like to handle this issue through
$\nu$-kernels and their $\onenu$-sets.

The
relation $f(\bfa) = \frac{\sum h_i}{\sum g_j}(\bfa) \nucong 1$ is
studied locally in the following sense: For any $\bfa \in F^{(n)}, $
there is at least one monomial $h_{i0}$ of the numerator and at
least one monomial $g_{j0}$ of the denominator which are dominant at
$\bfa$. If more than one monomial at $\bfa$ is dominant, say
$\{h_{ik}\}_{k=1}^{s}$ and $\{g_{jm}\}_{m=1}^{t}$, then we have
additional relations of the form $h_{i0} = h_{ik}$ and $g_{j0} =
g_{jm}$.

Our motivating example: The relation $\fone +\la = \fone$ holds
for all $\bfa \le_\nu \fone.$ This will be  an example of an
\textbf{order relation}.

 In general, the dominant monomials of both
numerator and denominator at some point
define  relations $\{ h_{ik} = g_{jm} \ : \ 0 \leq k \leq s, \ 0
\leq m \leq t \}$  on  regions of $F^{(n)} $. Every such relation
can be converted by multiplying by inverses of monomials to obtain
a relation of the form $1 = \phi(\Lambda)$ with $\phi \in
\overline{F(\Lambda)} \setminus F$ a Laurent monomial, and thus
reduces the dimension. Note that in the case in which $h_{i0}$ and
$g_{j0}$ singly dominate and are the same monomial, no extra
relation is imposed on the region described above, so we are left
only with the order relations defining the region.

 In this way, two types of principal
$\onenu$-sets emerge from two distinct types of $\nu$-kernels,
characterized by their generators, distinguished via the following
definition.

\begin{defn}
 A rational function $f = \frac{h}{g}  $ is
\textbf{regular} at a  point $\bfa$ in $\Skel (f)$  if  each
$\nu$-neighborhood of $\bfa$ contains a point on which $h$ and $g$
do not agree. Otherwise $f$ is \textbf{irregular} at  $\bfa$. $f$
is \textbf{regular} at a set~$S$ if it is regular at every point
$\bfa$ in $S$. $f$ is \textbf{regular} if it is regular at $\Skel
(f)$. $\operatorname{Reg}(\overline{F(\Lambda)})$ denotes the set
of regular rational functions.
\end{defn}

Note that any rational function which is regular at a  point $\bfa$
is also regular at the region containing $\bfa$. Another way of
stating this condition, writing $f = \frac hg$ is to define a
\textbf{leading Laurent monomial} of $f$ to be of the form
$\frac{h_i}{gj}$ where $h_i$ is a dominant monomial of $h$ and $
g_i$ is a dominant monomial of $g$. Of course $f$ will have several
leading Laurent monomials at $\bfa$ if $\bfa$ is a ghost root of $h$
or $g$. The regularity condition is that $f$ possesses some leading
Laurent monomial $\ne \fone .$  For example, $f = \frac{\la_1 + \la
 _2}{\la _1 + \la _3}$ is regular at $(1,1,1)$ but not at $(2,1,1)$.

\begin{lem}\label{rem_every_generator_is_regular}
If $f$ is regular, then any other  generator $f'  $
 of $\langle f \rangle$ is regular.\end{lem} \begin{proof} If $f' $
 were not regular at $\bfa$, then its sole leading Laurent monomial at $\bfa$ would be
 $\fone ,$ implying that some $\nu$-neighborhood of $\bfa$ is in
 $\Skel(f')$ and thus of $\Skel(f),$ contrary to $f$ possessing some
leading Laurent monomial $\ne \fone .$
\end{proof}


\begin{lem}\label{rem_regular_closed_operations}
If $f,g \in \operatorname{Reg}(\overline{F(\Lambda)})$ such that
$f \neq \fone $ and $g \neq \fone $, then the following elements
are also in $\operatorname{Reg}(\overline{F(\Lambda)})$:
$$f^*, \ \ f^{k} \ \text{with} \  k \in \mathbb{Z}, \  f \dotplusss  g, \  |f|, \ f \wedge g.$$
\end{lem}
\begin{proof}
Follows at once from Lemma~\ref{rem_every_generator_is_regular} and
Proposition ~\ref{prop_absolute_generator1}. A more direct argument
for $|f|$ is that writing $f   =  \frac{h}{g} $,
 we have $$|f| = \frac{h^2
+ g^2} {gh},$$ but the leading Laurent monomial is either that of
$h^2(gh)^* = f$ or that of $ g^2(gh)^* = f^*$.
\end{proof}

\begin{defn}\label{regu} The relation $f(\bfa) \nucong \fone$ is
\textbf{regular} at $\bfa$ if $f$ is regular at $\bfa$. Otherwise
$f(\bfa) \nucong \fone$ is an \textbf{order relation}.

A $\onenu$-set is \textbf{regular} if it can be written as
$\Skel(S)$ where each $f\in S$ is regular.

  \end{defn}

\begin{lem}\label{rem_finite_unions_and_intersections_of_reg_loci}
If  $A=\Skel(f_1,...,f_s)$ and $B = \Skel(g_1,...,g_t)$ are regular
then $A \cap B$ and $A \cup B$ are regular.
\end{lem}
\begin{proof}
  Immediate from Lemma
\ref{lem_intersection_union_finitely_generated_corner_loci}.
\end{proof}

Considering only dominant monomials at the neighborhood of $\bfa$,
our local relations fall into two distinct cases:
\begin{itemize}
  \item An  order relation  of the form $\fone \dotplusss  g \nucong \fone$
   with $g \in \overline{F(\Lambda)}$, i.e.,
  $g(\bfa)
  \le_\nu \fone$.
 The resulting quotient \vsemifield0\
 $\overline{F(\Lambda)}/\langle 1 \dotplusss  g \rangle$ does not reduce the dimension at such a point $\bfa$, but only imposes new order relations on the variables.
  \item A regular  relation; this reduces the dimensionality of the image of $\overline{F(\Lambda)}$ in the quotient \vsemifield0.
\end{itemize}
We aim to characterize those elements of $\overline{F(\Lambda)}$
that do not translate (locally) to order relations but only to
regular relations (locally).  This  will allow us to characterize
those relations which correspond to corner loci (tropical
varieties in tropical geometry). The $\nu$-kernels corresponding
to these relations will be shown to form a sublattice of the
lattice of principal $\nu$-kernels (which is itself a sublattice
of the lattice of $\nu$-kernels).

\begin{exmp}
If $f \in \overline{F(\Lambda)}$ such that $f \neq f \dotplusss
1$ (i.e., $1$ is essential in $f \dotplusss  1$), then $f
\dotplusss 1 = \frac{f \dotplusss  1}{1}$ is not regular since $1$
being essential in the numerator coincides with the denominator
over some nonempty region.
\end{exmp}

\begin{exmp}
Consider the  map $\phi : F(\lambda) \rightarrow F(\lambda)/\langle
\lambda \dotplusss  1 \rangle$ given by  $\lambda \mapsto 1 \nucong
1$. This map imposes the relation $\lambda \dotplusss  1 \nucong 1$
on $F(\lambda)$, which is just the order relation $\lambda \leq _\nu
1$. Under the map $\phi$, $\lambda$ is sent to $\bar{\lambda} =
\lambda \langle \lambda \dotplusss  1 \rangle$, where now, in
$\Im(\phi) = F(\bar{\lambda})$, $\bar{\lambda}$ and $\bar{1}$ are
comparable, as opposed to the situation in $F(\la)$, where $\lambda$
and $1$ are not comparable.
   If instead of  $\lambda \dotplusss  1$,
   we consider $|\lambda| \dotplusss  1 \nucong \lambda \dotplusss  \lambda^{-1} \dotplusss  1$, then as $|\lambda| \ge_\nu 1$ the relation $|\lambda| \dotplusss  1 \nucong 1$ means $|\lambda| \nucong 1$, as yielded by the substitution map sending $\lambda$ to $1$. Note that $|\lambda|$ and $1$ are comparable in $F(\lambda)$, since, as mentioned above, $|\lambda| \geq 1$, which is equivalent to the relation imposed by the equality $|\lambda| \dotplusss  1 \nucong |\lambda|$ or equivalently by $|\lambda|^{-1} \dotplusss  1 \nucong 1$.
   The $\nu$-kernel $ \langle |\lambda|^{-1} \dotplusss  1 \rangle$, since $|\lambda|^{-1} \dotplusss  1 \nucong 1$,
   is just the trivial $\nu$-kernel $\langle 1 \rangle = c$.  \\
\end{exmp}
\begin{note} As seen in Proposition \ref{prop_principal_ker} and Corollary
\ref{cor_principal_ker_by_order}, order relations affect the
structures of $\nu$-kernels in a \vsemifield0. For instance,
consider the principal $\nu$-kernel in a \vsemifield0\ $F$
generated by an element $a \in F$. Then $b \not \in \langle a
\rangle$ for any element $b \in F$ such that $b$ not comparable to
$a$.
\end{note}

 \section{The role of the
$\nu$-kernel  $\langle F \rangle$ of $\overline{F(\Lambda)}$ in
the lattice of regular principal $\nu$-kernels
}\label{subsection:bounded_kernel0} $ $

Our overall objective in this section and the next is to describe
the theory in terms of principal $\nu$-kernels.

\begin{defn}\label{defn_principal_kernels_in_semifield_of_fractions}
Denote the sublattice of principal $\nu$-kernels of a \vsemifield0
$\mathcal{S}$ by $\PCon(\mathcal{S})$.
\end{defn}

\begin{defn}
 A principal $\nu$-kernel $K = \langle f \rangle$ of $\overline{F(\Lambda)}$ is
\textbf{regular} if $f \in \overline{F(\Lambda)}$ is regular. In
this case the 1-Set $\Skel(f)$ corresponding to $K$ also is called
\textbf{regular}.
\end{defn}

\begin{prop}\label{cor_regular_lattice}
The set of regular principal $\nu$-kernels forms a sublattice of
$\PCon(\overline{F(\Lambda)})$.
\end{prop}
\begin{proof}
This follows directly from Corollary
\ref{cor_max_semifield_principal_kernels_operations} and Lemma
\ref{rem_regular_closed_operations}.
\end{proof}
%

%

\subsubsection{Corner loci and principal $\onenu$-sets}$ $

  Suppose  $X \subseteq F^{(n)} $, and consider rational functions restricted to $X$.
By Remark \ref{rem_intersection_union_corner_loci} and Lemma
\ref{lem_intersection_union_finitely_generated_corner_loci}, the
collection of corner loci is closed under intersections, while the
collection of finitely generated principal corner loci is also
closed under finite unions. Also,  $\Corn(\emptyset) = X$ and $
\Corn(\alpha)= \emptyset $.

By Lemma~\ref{rem_finite_unions_and_intersections_of_reg_loci}, the
regular finitely generated principal corner loci comprise a
sublattice, which we want to investigate.

\begin{rem}\label{OK2}
 By Remark
\ref{rem_corner_to_skel} and Proposition~\ref{cor_regular_lattice}
 applied
to the correspondence of Theorem~\ref{cor11}, the lattice
generated by principal corner internal $\nu$-kernels with respect
to (finite) products and intersections   corresponds to the
lattice of finitely generated
 corner loci. Note that $\Skel(\one) = X$ and $
\Skel(\alpha)= \emptyset $ for $\alpha \not \nucong \one$.
\end{rem}

To delve deeper, we need to turn to the $\nu$-kernel $\langle F
\rangle $ of $\overline{F(\Lambda)}.$

\begin{thm}\label{thmcireglattice}\cite[Theorem (13.5.2)]{AlgAspTropMath}
The lattice $\PCon(\langle F \rangle )$  is generated by the
principal corner internal $\nu$-kernels, and the sublattice of
regular principal $\nu$-kernels is generated by the regular,
principal corner internal $\nu$-kernels.
\end{thm}
The proof is rather long, requiring the concept of bounded
$\nu$-kernels, so we defer it until \S\ref{thmpf}.

\begin{cor}\label{corlatticegeneratedbyCIkernels} The lattice of (tangible) finitely generated corner loci corresponds
to the lattice of principal (tangible)   $\nu$-kernels of $\langle
F \rangle$. Intersections of (supertropical) hypersurfaces
correspond to principal $\onenu$-sets and $\nu$-kernels, whereas
intersections of tangible hypersurfaces correspond to tangible
principal $\onenu$-sets and $\nu$-kernels.\end{cor}\begin{proof}
We use the correspondence of Remark~\ref{OK2}, between principal
(tangible) corner-loci and
principal (tangible) corner internal $\nu$-kernels of $\langle F \rangle$.%
\end{proof}

Thus, supertropical varieties correspond to principal
$\onenu$-sets and $\nu$-kernels, while tropical varieties
correspond to regular principal $\onenu$-sets and $\nu$-kernels.

In order to bypass the ambiguity between $\nu$-kernels and
$\onenu$-sets encountered in Example~\ref{onenusetex}(i), we
introduce  one particular sub-\vsemifield0\ of
$\overline{F(\Lambda)}$ of considerable interest.

 \subsection{The $\nu$-kernel $\langle F \rangle$}$ $

Assume that  $F: = (F, \tT, \nu, \tG)$ is an  archimedean
$\nu$-bipotent \vsemifield0.

\begin{rem}\label{princ}
$ \langle \alpha \rangle  = \langle \beta \rangle $ for any $\alpha,
\beta \neq 1$  in $F$, in view of
Proposition~\ref{prop_principal_ker}. \end{rem}

\begin{defn}  $\langle
F \rangle$ denotes the $\nu$-kernel  given in Remark~\ref{princ}.
\end{defn}

The $\nu$-kernel  $\langle F \rangle$ is preserved under  any
  homomorphism $\phi$ for which
$\phi(F) \ne 1$. In this subsection we
  show that $ \langle F \rangle$ retains all the
information in $\overline{F(\Lambda)}$ needed for the important
family of principal $\onenu$-sets and provides a 1:1
correspondence between $\nu$-kernels and $\onenu$-sets.

 The $\nu$-kernel $\langle F \rangle$ is
 much more sophisticated than what one might think at first
 blush, because of the convexity condition. Namely, any function
 lying between two constants is in the $\nu$-kernel. For example,
 $3 \wedge |\la|$ is constant except in the interval between $\frac 13$ and $3$,
 where it descends to $1$ and then increases back to $3.$

 We begin by introducing the motivating example
for this section.

\begin{exmp}\label{nonuniq}
 The  $\onenu$-set corresponding to the principal $\nu$-kernel
  $\langle \lambda \rangle$ of $F(\la)$ is the set of all
 $\bfa \nucong 1$. For any  $\alpha \neq 1$ in~$F$, we also have the principal  $\nu$-kernel
   $ \langle \lambda \rangle \cap \langle \alpha \rangle =
   \langle |\lambda| \wedge |\alpha| \rangle  =
    \left\langle  \frac{|\lambda| |\alpha|}{|\lambda| + |\alpha| } \right\rangle $. As $\lambda \not \in \langle \lambda \rangle \cap \langle \alpha \rangle$,
       we conclude  that $\langle \lambda \rangle \supset  \langle \lambda \rangle \cap \langle \alpha \rangle$.\\

    But  $ (|\lambda|\wedge|\alpha|)(\bfa) \nucong 1$ iff $ \bfa \nucong
    1$.
    It follows at once that
    $ \Skel(|\lambda| \wedge |\alpha|) =  \Skel(\lambda) $.
\end{exmp}

In other words, different $\nu$-kernels may have the same
$\Skel($.
 This ambiguity can be bypassed by intersecting all $\nu$-kernels with
$\langle F \rangle$. 

\begin{prop}\label{intersects}
If $K $ is a $\mathcal{K}$-kernel of $\overline{F(\Lambda)}$, then
$K' := K \cap \langle F\rangle$ is  a   $\mathcal{K}$-kernel  of
$\langle F \rangle$ satisfying $\Skel(K) = \Skel(K')$.
\end{prop}
\begin{proof}
 $K = \Ker(\Skel(K))$
since  $K$ is a $\mathcal{K}$-kernel. Fix $\alpha \in F \setminus
\{1\}$.  By Proposition
\ref{prop_algebra_of_generators_of_kernels}(ii) we have that  $K
\cap \langle F \rangle = K \cap \langle \alpha \rangle = \langle
\{|f| \wedge |\alpha| : f \in K \} \rangle$.  Now, for any $f \in
K$, $f(\bfa)\nucong 1$ for some $\bfa \in F^{(n)} $ if and only if
$f(\bfa) \wedge |\alpha| = 1$ (since $|\alpha| > 1$) so $\Skel(K') =
\Skel(K)$. Thus $$K' = K \cap \langle F \rangle = \Ker(\Skel(K'))
\cap \langle F \rangle = \Ker_{\langle F\rangle}(\Skel(K')),$$ and
so $K'$ is a $\mathcal{K}$-kernel of $\langle F \rangle$.
\end{proof}

We want $K'$ to be unique with this property, but for this we need
to assume that $F$ is complete and $\nu$-archimedean,
cf.~Proposition~\ref{prop_generator_of_skel2} below.


Thus, we can pass down to $ \langle F \rangle $, leading us to the
next definition.

\begin{defn}\label{defn_abstract_similarity_of_generators}
 Define
the equivalence relation

\begin{equation}
f \sim_{\FF} f'\Leftrightarrow \langle f \rangle \cap \langle F
\rangle  = \langle f'\rangle  \cap \langle F \rangle .
\end{equation}as $\nu$-kernels of $\overline{F(\Lambda)}$.   The equivalence classes   are
$$[f] = \{ f' \ : \ f' \ \text{is a generator of} \ \langle f
\rangle \cap  \langle F \rangle \}.$$ 
\end{defn}


\begin{prop}\label{cor_intersection_with_the kernel_of_H}\cite[Corollary (5.2.7)]{AlgAspTropMath}
For any $\nu$-kernel $K$ of $\overline{F(\Lambda)}$, if $K \cap
\langle F \rangle = \{ 1 \} ,$ then $K = \{ 1 \}.$
\end{prop}

 \begin{proof}  For any  $\bfa \in F^{(n)}$,
   \  $|f(\bfa)| \wedge |\alpha| \nucong 1$ if and only if $|f(\bfa)|\nucong 1$.
\end{proof}

\begin{lem}\label{rem_the_lattice_of_kernels_of_the_bounded_kernel}
$$\PCon(\langle F \rangle) = \{ \langle f \rangle \cap \langle F \rangle : \langle f \rangle \in \PCon(\overline{F(\Lambda)}) \}.$$
\end{lem}
\begin{proof}
$\langle f \rangle \cap \langle F \rangle = \langle |f| \wedge
|\alpha| \rangle \in \PCon(\langle F \rangle)$ for any
 $\langle f \rangle \in
\PCon(\overline{F(\Lambda)})$.
\end{proof}
%
%

 \subsubsection{Bounded rational functions}\label{subsection:bounded_kernel}\ $ $

$F$ is assumed to be an archimedean supertropical \vsemifield0.

\begin{defn} A rational function
$f \in \overline{F(\Lambda)}$ is said to be \textbf{bounded from
below} if there exists some $\alpha >_\nu 1$ in $F$ such that $|f|
\ge_\nu \alpha$. $f$ is said to be \textbf{bounded from above} (or
\textbf{bounded}, for short) if there exists some $\alpha \ge_\nu
1$ in~$F$ such that $|f| \leq_\nu \alpha$.
\end{defn}

%

\begin{rem}\label{rem_bounded_from_below_contain_H_kernel}\cite[Remarks (5.1.3-4), (5.1.10-12) ]{AlgAspTropMath}
Let $\langle f \rangle$ be a principal $\nu$-kernel of
$\overline{F(\Lambda)}$. Then
 \begin{enumerate}\eroman
   \item $f$ is bounded from below if and only if
   $\langle f \rangle \supseteq \langle \alpha \rangle$ for some $\alpha > 1$
    in $F$. Moreover, any generator $g \in \langle f \rangle$ is bounded from below.
   \item $f$ is bounded  if and only if $\langle f \rangle \subseteq \langle \alpha \rangle$ for some $\alpha \in F$. Moreover, any generator
    $g \in \langle f \rangle$ is bounded .  \end{enumerate}
 \end{rem}

\begin{defn}
A principal $\nu$-kernel $\langle f \rangle$  of
$\overline{F(\Lambda)}$ is said to be \textbf{bounded from below},
if it is generated by a rational function bounded from below.
$\langle f \rangle$ is said to be \textbf{bounded} if it is
generated by a bounded rational function.
\end{defn}

\begin{prop}\label{cor_empty_kernels_correspond_to_bfb_kernels}\cite[Corollary (5.1.8)]{AlgAspTropMath}
  $\Skel(\langle f \rangle) = \emptyset$
if and only if $f $ is bounded from below.
\end{prop}

\begin{lem}\label{rem_bounded_from_above_elements_kernel}
$\langle F \rangle = \{ f \in \overline{F(\Lambda)} \ : \ f \
\text{is bounded} \}.$
\end{lem}
\begin{proof}
 The assertion
follows from
Remark~\ref{rem_bounded_from_below_contain_H_kernel}(ii), since $f
\in \langle F \rangle $ if and only if $\langle f \rangle
\subseteq \langle FW \rangle$.
\end{proof}

\begin{lem}\label{rem_complete_semifield_of_functions}
If  $F$ is complete, then $Fun(X,F)$ is complete, for any subset
$X$ of~$F^{(n)}$.
\end{lem}
\begin{proof}
Suppose $\mathcal S \subset Fun(X,F)$ is bounded from below, say
by $h \in Fun(X,F)$. Then for any $\bfa \in X$ the set $W: =
\{f(\bfa) : f \in X \}$ is bounded from below by $h(\bfa)$ and
thus has an infimum $\bigwedge_{f \in W}f(\bfa)$. It is readily
seen that the function $g \in Fun(F^{(n)},F)$ defined by $g(\bfa)
= \bigwedge_{f \in W}f(\bfa)$ is an infimum for $X$, i.e., $g =
\bigwedge_{f \in W}f$. Analogously, if $W$ is bounded  then
$(\bigvee_{f \in W}f)(\bfa) = \bigvee_{f \in W}f(\bfa)$ is the
supremum of~$W$.
\end{proof}

   \subsubsection{The map $\omega$}
\begin{defn}
Fixing $\alpha >_\nu 1$, define the map $\omega :
\overline{F(\Lambda)}^{+} \rightarrow \langle F \rangle^{+}$ by
$$\omega(|f|) = |f| \wedge |\alpha|.$$\end{defn}
\begin{lem} $\omega$ is a lattice homomorphism.\end{lem}\begin{proof} Indeed, $\omega (|f| \dotplusss  |g|) = (|f| \dotplusss  |g|) \wedge |\alpha| = (|f| \wedge |\alpha|) \dotplusss  (|g| \wedge |\alpha|) = \omega(|f|) \dotplusss  \omega(|g|)$ and $\omega (|f| \wedge |g|) = (|f| \wedge |g|) \wedge |\alpha| = (|f| \wedge |\alpha|) \wedge (|g| \wedge |\alpha|) = \omega(|f|) \wedge \omega(|g|)$.\\
\end{proof}
$\omega$ induces a lattice map
$$\Omega : \PCon(\overline{F(\Lambda)}) \rightarrow \PCon(\overline{F(\Lambda)}) \cap \langle F \rangle$$
such that $\Omega(\langle f \rangle) = \langle \omega(|f|) \rangle =
\langle |f| \wedge |\alpha| \rangle = \langle f \rangle \cap \langle
F \rangle$.


\begin{lem}\label{rem_properties_of_H_kernels}
\begin{enumerate}\eroman
  \item   $\Omega$ preserves
  both  intersections and products of $\nu$-kernels. \item If $\langle f \rangle$ is a  $\nu$-kernel that is bounded from below,
   then $\langle f \rangle \cap \langle F \rangle = \langle F
   \rangle$. In fact, every principal $\nu$-kernel whose $\onenu$-set is the empty set is mapped to
   $\langle F \rangle$.
  \item $\Skel(\Omega(\langle f \rangle)) =
       \Skel(\langle f \rangle \cap \langle H \rangle) =
       \Skel(\langle f \rangle) \cup \emptyset =
       \Skel(\langle f \rangle).$
       \item Any
        $\nu$-subkernel $K \subset \langle \alpha \rangle$
        must satisfy  $K \cap F = \{ 1 \}$.
\end{enumerate}
\end{lem}
\begin{proof}
(i) By Proposition~\ref{prop_ker_skel_correspondence}.

(ii)   By Remark \ref{rem_bounded_from_below_contain_H_kernel} and
Lemma~\ref{rem_the_lattice_of_kernels_of_the_bounded_kernel}.

(iii)   $\Skel(\langle \alpha \rangle) = \emptyset$ and
$\Skel(\langle f \rangle \cap \langle g \rangle) = \Skel(\langle f
\rangle) \cup \Skel(\langle g \rangle)$, for any principal
$\nu$-kernel $\langle f \rangle$.

(iv) Any $\alpha \neq 1$ generates $\langle F \rangle$.
\end{proof}

We now show that restricting $\Skel$ and $\Ker$ to $\langle F
\rangle$ does not affect the $\onenu$-sets and that each
$\mathcal{K}$-kernel of $\overline{F(\Lambda)}$  restricts to a
$\mathcal{K}$-kernel in $\langle F \rangle$.

In
Proposition~\ref{prop_maximal_kernels_in_semifield_of_fractions_part1},
we have shown using the substitution homomorphism $\psi$ that any
point  $a = (\alpha_1,...,\alpha_n) \in F^{(n)} $ corresponds to
the maximal $\nu$-kernel
$$\left\langle \frac{\la_1}{\alpha_1} , ... , \frac{\la_n}{\alpha_n} \right\rangle = \left\langle \left|\frac{\la_1}{\alpha_1}\right| + ....+ \left|\frac{\la_n}{\alpha_n}\right| \right\rangle = \left\langle \left|\frac{\la_1}{\alpha_1}\right| \cdots  \left|\frac{\la_n}{\alpha_n}\right| \right\rangle = \left\langle \frac{\la_1}{\alpha_1} \right\rangle  \cdots \left\langle \frac{\la_n}{\alpha_n} \right\rangle.$$
For example, consider the   homomorphism $\psi : F(\la_1)
\rightarrow F$  given by
 $\la_1 \mapsto 1$. By Theorem \ref{thm_kernels_hom_relations}, the $\nu$-kernel of
  its
 restriction
  $\psi|_{\langle F \rangle} :\langle F
\rangle \rightarrow \psi(\langle  F \rangle) = F$ is
$\Ker(\psi|_{\langle F \rangle})= \Ker (\psi) \cap \langle F
\rangle = \langle \la_1 \rangle \cap \langle F \rangle$. Thus, the
proposition applies to $\langle F \rangle$ providing the maximal
$\nu$-kernel $\langle x \rangle \cap \langle F \rangle$. We will
now show that any maximal $\nu$-kernel of $\langle F \rangle $ has
that form.

\begin{prop}\label{prop_maximal_kernels_in_semifield_of_fractions_part2}
If $K$ is a maximal $\nu$-kernel of  $\langle F \rangle$, then
$$K = \Omega\left(\left\langle \frac{\la_1}{\alpha_1}, ... ,
\frac{\la_n}{\alpha_n} \right\rangle\right)$$ for suitable
$\alpha_1,...,\alpha_n \in F$.
\end{prop}
\begin{proof}
Denote $L_a = (|\frac{\la_1}{\alpha_1}| + ....+
|\frac{\la_n}{\alpha_n}|) \wedge |\alpha|$ with $\alpha \neq 1$,
for $a = (\alpha_1, ..., \alpha_n)$. By
Lemma~\ref{rem_properties_of_H_kernels} we may assume that
$\Skel(K) \neq \emptyset$, since the only $\nu$-kernel
corresponding to the empty set is $\langle F \rangle$ itself. If
$a \in \Skel(K)$, then as $\Skel(L_a) = \{a\} \subseteq \Skel(K)$,
we have that $\langle L_a \rangle \supseteq K$. Thus, the
maximality of $K$ implies that $K = \langle L_a \rangle$.
\end{proof}

%

\subsection{Completions of  idempotent $\nu$-\semifields0}\label{compl1}$ $

We recall Definition~\ref{defn_cond_complete} of ``complete,''
viewing a \vsemifield0 as a lattice. In essence it is enough to
consider $\mathscr{R}$ from Example~\ref{mathscrR}.

\subsubsection{Kernels of $\nu$-archimedean idempotent $\nu$-\semifields0} $ $

\begin{thm}\cite[Theorem~2.3.10 (H\"{o}lder)]{OrderedGroups3}\label{thm_holder_translated}
The following statements are equivalent for an idempotent
\vsemifield0\ $\mathcal{S}$.
\begin{enumerate}
  \item $\mathcal{S}$ is simple.
  \item $\mathcal{S}$ is totally ordered and archimedean.
  \item $\mathcal{S}$ can be embedded into the max-plus algebra
  of $\Rplus$.
\end{enumerate}

The  supertropical version: Any $\nu$-archimedean supertropical
\vsemifield0\ $\mathcal{S} = (\mathcal{S}, \tT, \nu, \tG)$   can
be embedded into
  $\mathscr{R}$.
\end{thm}

\begin{cor}\cite{OrderedGroups3}\label{cor_totaly_ordered_isomorphic_to_reals}
Any complete divisible totally ordered $\nu$-archimedean
supertropical \vsemifield0 is isomorphic to
  $\mathscr{R}$.
\end{cor}

To avoid duplication for $\nu$-kernels having the same
$\onenu$-set, our next step is to work with $\mathscr{R}$ and
restrict to $\nu$-subkernels of $\langle \mathscr{R}\rangle$ in
$\mathscr{R}(\Lambda)$.

\begin{defn}\label{defn_coinitial}
A subset $A$ of the poset $P$ is called  \textbf{co-initial} in $P$
if for every $x \in P$ there exists some $f \in A$ such that $f \leq
x$. $A$ is \textbf{co-final} in $P$ if for every $x \in P$ there
exists some $f \in A$ such that $x \leq f$.
\end{defn}

\begin{defn}\label{defn_lef_right_dense}
The subset $A$ of the idempotent \vsemifield0\ $\mathcal{S}$ is
called \textbf{left dense} in~$\mathcal{S}$ if $A^{+} $
(cf.~Definition~\ref{pos})
 is co-initial in
$\mathcal{S}^{+} $, and $A$ is called \textbf{right dense} in H if
$A^{+}$ is co-final in $\mathcal{S}^{+}$.
\end{defn}

\begin{defn}
A \textbf{completion} of the idempotent \vsemifield0\
$\mathcal{S}$ is a pair $(H,\theta)$ where $H$ is a complete
idempotent \vsemifield0\ and $\theta : \mathcal{S} \rightarrow H$
is a monomorphism whose image is dense (left and right) in $H$.
\end{defn}

\begin{rem}\label{thm_completion_of_l_group} Each $\nu$-archimedean idempotent \vsemifield0\ has a unique
completion up to isomorphism. The proof follows the standard lines,
with details given in \cite[Theorem
(2.3.4)]{OrderedGroups3}.\end{rem}


\begin{rem}\label{rem_carachterization_of_completions}
When the supertropical \vsemifield0 $F$ is $\nu$-archimedean,  any
$\nu$-kernel $K$ of $\overline{F(\Lambda)}$  also has a
completion, which we denote as $\bar{K} \subseteq
 \overline{F(\Lambda)}$, which
by Corollary \ref{cor_completion_of_kernels} is a $\nu$-kernel of $ \overline{F(\Lambda)}$.\\ 
\end{rem}

\begin{thm}\cite[Theorem~2.3.6]{OrderedGroups3}\label{thm_completion_of_a_semifield}
Let $\mathcal{S}$ be an $\nu$-archimedean sub-\vsemifield0\ of a
complete idempotent \vsemifield0~$\hat{\mathcal{S}}$. Then the
following are equivalent:
\begin{enumerate}
  \item $\hat{\mathcal{S}}$ is the completion of $\mathcal{S}$.
  \item $\mathcal{S}$ is left dense in $\hat{\mathcal{S}}$, and if $A$ is an idempotent subsemifield of $\hat{\mathcal{S}}$ that is complete and
contains $\mathcal{S}$, then $A = \hat{\mathcal{S}}$.
\end{enumerate}
\end{thm}

\begin{cor}\label{cor_completion_of_kernels}\cite{OrderedGroups3}
Suppose  $\mathcal{S}$ is a left dense $\nu$-archimedean
idempotent sub-\vsemifield0 of the complete idempotent
\vsemifield0\ $\hat{\mathcal{S}}$. Then  the $\nu$-kernel of
$\hat{\mathcal{S}}$ generated  by a $\nu$-kernel $K$   of \
$\mathcal{S}$, is the completion of $K$ in  $\hat{\mathcal{S}}$.
\end{cor}

By Remark~\ref{thm_completion_of_l_group}, $\mathscr{R}(\Lambda)$ has a
 unique completion to a complete $\nu$-archimedean idempotent
 \vsemifield0\ $\overline{\mathscr{R}(\Lambda)}$ in  $Fun(\mathscr{R}^{(n)} ,\mathscr{R})$. By Theorem \ref{thm_completion_of_a_semifield}, $\mathscr{R}(\Lambda)$ is dense in $\overline{\mathscr{R}(\Lambda)}$. \\

\subsection{The  $\nu$-kernel $\langle
\mathscr{R} \rangle$ of $\mathscr{R}(\Lambda)$ as bounded
functions}\label{subsection:bounded_kernel2}$ $


\begin{prop}\label{prop_unbounded_generator}\cite[Proposition (5.2.1)]{AlgAspTropMath}
For any principal  bounded $\nu$-kernel $\langle f \rangle \in
\PCon(\langle \mathscr{R} \rangle)$, there exists an unbounded
 $\nu$-kernel $\langle f' \rangle \in
\PCon(\mathscr{R}(\Lambda))$,   such that
$$ \langle f \rangle = \langle f' \rangle \cap \langle \mathscr{R}
\rangle.$$ In particular, $\langle f' \rangle \supset \langle f
\rangle$  and $\Skel(f') = \Skel(f)$.
\end{prop}%

\begin{proof}
Suppose $f\in \langle \mathscr{R} \rangle$ is bounded. Then there
exists some $\beta_1 \in \mathscr{R}$ such that $|f(\bfa)| \le_\nu
\alpha_1 \in \mathscr{R}$. Similarly for each $2 \leq i \leq n$
there exists some $\beta_i \in \mathscr{R}$ such that $|f(\bfa)|
\le_\nu \alpha_i $ whenever $\a_i \ge_\nu \beta_i$. As~$|f|$ is
continuous we may assume that $\alpha_i = \alpha$ are all the same.
Now define the function
$$f' = | (\beta^{-1}|\la_1| \wedge .... \wedge \beta^{-1}|\la_n|+1)| + |f(\Lambda)|$$
where $\beta = \sum_{i=1}^{n}|\beta_i|$. Write $g(\la_1,...\la_n) = \beta^{-1}|\la_1| + .... + \beta^{-1}|\la_n|+1$. Let
$$S = \{\bfa =(a_1, \dots, a_n) \in F^{(n)}  :  |a_i| > \beta \ \forall i \}.$$ Then for every $a \in S$, $f'(a)=g(a)+|\alpha|$. Moreover, for every $b=(b_1,...,b_n) \not \in S$ there exists some $j$ such that $|b_j| \leq \beta$   thus we have that $(\beta^{-1}|b_1| + .... + \beta^{-1}|b_n|) \leq 1$ and so $g(b) = 1$. By construction $\Skel(f) \subseteq \Skel(g)$, so  $\Skel(f') = \Skel(|g| + |f|) = \Skel(g) \cap \Skel(f) = \Skel(f)$. Finally $f'$ is not bounded
since $|g|$ is not bounded and $|f'| = |g| + |f| \ge_\nu |g|$. Now,
as $|f'| = |g| + |f|$ we have that $|f| \leq |f'|$, so $f \in
\langle f' \rangle$. On the other hand, since $f'$ is not bounded,
clearly $f' \not \in \langle f \rangle$. Finally,  \ $g(a) \ge_\nu
1$ for any $a \in S$. Thus, $f'(a) \wedge |\alpha| \nucong (g(a) +
|\alpha|) \wedge |\alpha| \nucong |\alpha|$, while for $a \not \in
S$ \ $f'(a) \wedge |\alpha| \nucong (g(a) + |f(a)|) \wedge |\alpha|
\nucong (1 + |f(a)|) \wedge |\alpha| \nucong  |f(a)| \wedge |\alpha|
\nucong |f(a)|$, since $|f| \leq_\nu |\alpha|$. So we get that $|f'|
\wedge |\alpha| \nucong |f|$ which means that $\langle f \rangle =
\langle f' \rangle \cap \langle \mathscr{R} \rangle$.  (Note that
$f' = |f'|$ by definition, since $f' \ge_\nu 1$.)
\end{proof}

\subsection{Principal $\onenu$-sets and bounded
$\nu$-kernels}\label{4.1}$ $

 It turns out that $\langle \mathscr{R}
\rangle$ possesses enough distinct bounded copies of the principal
$\nu$-kernels of $\mathscr{R}(\Lambda)$ to  represent faithfully
the principal $\onenu$-sets.


\begin{rem}
The restriction of the image of the operator  $\Ker : \mathbf{2}^{\mathscr{R}^{(n)}} \rightarrow \mathbf{2}^{\mathscr{R}(\Lambda)}$
to $\mathbf{2}^{\langle \mathscr{R} \rangle}$ is
\begin{equation}
\small{\Ker_{\langle \mathscr{R} \rangle}(Z) = \{ f \in \langle
\mathscr{R} \rangle \ : \ f(a_1,...,a_n) = 1,  \ \forall
(a_1,...,a_n) \in Z \} = \Ker(Z) \cap \langle \mathscr{R} \rangle.}
\end{equation}

Furthermore, the assertions of this subsection  apply to
$\Ker_{\langle \mathscr{R} \rangle}$ and the restriction
 $\Skel|_{\langle \mathscr{R} \rangle} : \mathbf{2}^{\langle
\mathscr{R} \rangle} \rightarrow \mathbf{2}^{\mathscr{R}^{(n)}}$  of
 $\Skel $  to
$\mathbf{2}^{\langle \mathscr{R} \rangle}$.
\end{rem}

When there is no ambiguity, we denote
 $\Ker_{\langle \mathscr{R} \rangle}$  and
$\Skel|_{\mathbf{2}^{\langle \mathscr{R} \rangle}}$ respectively as
$\Ker$ and $\Skel$.

Applying Lemma~\ref{rem_properties_of_H_kernels} to the
\vsemifield0\ $\mathscr{R}$, we see that
$$\Omega : \PCon(\mathscr{R}(\Lambda)) \rightarrow \PCon(\langle \mathscr{R} \rangle)$$
is a lattice homomorphism of the lattice $(\PCon(\mathscr{R}(\Lambda)), \cdot, \cap)$ onto $(\PCon(\langle \mathscr{R} \rangle), \cdot, \cap)$, such that $\Skel(\langle f \rangle) = \Skel(\Omega(\langle f \rangle))$.\\

Let $f \in \PCon(\langle \mathscr{R} \rangle)$ and let $A = \{ g \in \mathscr{R}(\Lambda) : \Omega(\langle g \rangle) = f \}$. Define the $\nu$-kernel $K = \langle A \rangle $ of $\mathscr{R}(\Lambda)$. Then by Remark \ref{rem_properties_of_H_kernels}, if $g \in K$ then   $\Skel(g) = \Skel(f)$.\\

We aim for the 1:1 correspondence
$$ \langle f \rangle \in \PCon(\langle \mathscr{R} \rangle) \leftrightarrow \Skel(f)$$
between the principal $\onenu$-sets in $\mathscr{R}^{(n)} $ and the $\nu$-kernels in  $\PCon(\langle \mathscr{R} \rangle)$.\\
If $\Skel(g) = \Skel(f)$, then   $\Skel(\Omega(g)) = \Skel(f)$ since
$\Skel(g) = \Skel(\Omega(g))$, and $\Omega(g),f \in \PCon(\langle
\mathscr{R} \rangle)$. Thus in view of the above $\Omega(\langle g
\rangle) = \langle f \rangle$. Consequently
  $\Skel(K) = \Skel(f)$, and $K$ is the maximal $\nu$-kernel of
$\mathscr{R}(\Lambda)$ having this property. Our next result
justifies the use of $\langle \mathscr{R} \rangle$.

\begin{prop}\label{prop_generator_of_skel2}
If $\langle f \rangle \subseteq \langle \mathscr{R} \rangle$ and $h
\in \langle f \rangle$ is such that $\Skel(h) = \Skel(f)$, then $h$
is a generator of $\langle f \rangle$.
\end{prop}
\begin{proof}
The assertion is obvious in the special case for which  $\langle f
\rangle \nucong \{1\}$.  So, as $\Skel(h) = \Skel(f)$ we may assume
that $f$ and $h$ are not $\nu$-equivalent to $1$. If $\langle f
\rangle = \langle \alpha \rangle$, then $\Skel(h) = \Skel(f) =
\emptyset$ implies by Lemma~\ref{rem_properties_of_H_kernels} (3)
that $\langle h \rangle = \langle \mathscr{R} \rangle = \langle f
\rangle$.

%

Suppose there is a rational function $h \in \langle f \rangle$ which
is not a generator of $\langle f \rangle$   but satisfying
$\Skel(h) =\Skel(f)$. By Corollary~\ref{cor_principal_ker_by_order},
for each $k \in \mathbb{N}$ there exists some $\bfa \in
\mathscr{R}^{n}$ for
which $|f(\bfa)| >_\nu |h(\bfa)|^{k}$. 
For any $k \in \mathbb{N}$, define the set $U_k = \{\bfa \ : \
|f(\bfa)|
>_\nu |h(\bfa)|^k \}$.  As $\mathscr{R}$ is  dense, for any $\bfa \in U_k$
there exists a  neighborhood $B_\bfa \subset U_k$ containing $\bfa$
such that for all $\bfa' \in B_\bfa$,  $|f(\bfa')| > |h(\bfa')|^k $.
Now, since $h$ and $f$ are bounded, the $U_k$ are bounded regions
inside $\mathscr{R}^{(n)} $. Taking the closures, we may assume that
the $U_k$ are closed (nonempty). Since $\Skel(h) =\Skel(f)$,
$|f(\bfa)|
>_\nu |h(\bfa)|^k $ implies that $|h(\bfa)|,|f(\bfa)| >_\nu 1$ , so, by the
definition of $U_k$ we get the sequence of strict inclusions $U_1
\supset U_2 \supset \dots \supset U_k \supset \dots$. Thus, since
$\mathscr{R}$ is complete, there exists an element $\bfb \in
\mathrm{B} = \bigcap_{\mathbb{N}} B_k$. Now, for $\bfa \not \in
\Skel(h)$, $|h(\bfa)|
> 1$, and thus there exists some $r = r(\bfa) \in \mathbb{N}$ such that
$|h(\bfa)|^r >_\nu |f(\bfa)|$ and thus $\bfa \not \in \mathrm{B}$
thus $\bfb \not \in \mathscr{R}^{(n)}  \setminus \Skel(h)$. On the
other hand, if $\bfb \in \Skel(h)$ then $\bfb \in \Skel(f)$ so $
1\nucong |f(\bfb)| \leq_\nu |h(\bfb)|\nucong 1$. Thus
$\bigcap_{\mathbb{N}} B_k = \emptyset$, contradiction.
\end{proof}
%

Using  Theorem~\ref{prop_correspondence} and
Proposition~\ref{intersects}, we conclude:

\begin{thm}\label{cor0}
There is a $1:1$, order-reversing correspondence
\begin{equation}
\{ \text{principal $\onenu$-sets of } \ \mathscr{R}^{n}  \}
\leftrightarrow \{ \mathcal{K}-  \text{principal $\nu$-kernels of
} \ \langle \mathscr{R} \rangle \},
\end{equation}
given by $Z \mapsto \Ker_{\langle \mathscr{R} \rangle}(Z)$;  the
reverse map is given by $K \mapsto \Skel(K)$.
\end{thm}
\begin{proof}
Every principal $\nu$-kernel gives rise to a principal
$\onenu$-set by the definition of $\Skel$. The reverse direction
follows from Proposition \ref{prop_generator_of_skel2}, as every
principal $\nu$-kernel which produces a principal $\onenu$-set via
$\Skel$ is in fact a $\mathcal{K}$-kernel.
\end{proof}

\begin{prop}\label{prop_kernel_is_k_kernel}
Let   $\langle f \rangle$ be a principal $\nu$-kernel in
$\PCon(\langle \mathscr{R} \rangle)$. Then $\langle f \rangle$ is
a $\mathcal{K}$-kernel.
\end{prop}
\begin{proof}
We need to show that $\Ker(\Skel(f)) \subseteq \langle f \rangle$.\\
Let $h \in \langle \mathscr{R} \rangle $ such that $h \in
\Ker(\Skel(f))$. Then $h(x) = 1$  for every $x \in \Skel(f)$ and so
$\Skel(f) \subseteq \Skel(h)$. If $|h| \leq |f|^{k}$ for some $k \in
\mathbb{N}$ then $h \in \langle f \rangle$. Thus in particular we
may assume that $h \neq 1$. Now, by
Corollary~\ref{cor_principal_skel_correspondence},  $\Skel(\langle f
\rangle \cap \langle h \rangle) = \Skel(f) \cup \Skel(h) =
\Skel(h)$. Since $h \neq 1$, $\Skel(h) \neq \mathscr{R}^{(n)} $ and
thus $\langle f \rangle \cap \langle h \rangle \neq \{1\}$. Again,
$\Skel(\langle f \rangle  \langle h \rangle) = \Skel(f) \cap
\Skel(h) = \Skel(f)$ by Corollary
\ref{cor_principal_skel_correspondence}.

 Thus $\langle f,h \rangle = \langle f \rangle
 \langle h \rangle \neq \langle \mathscr{R} \rangle$ for
otherwise $\Skel(f) = \emptyset$ . Consequently for  $g=|f| \wedge
|h|$,  the $\nu$-kernel $K = \langle g \rangle = \langle f \rangle
\cap \langle h \rangle$ satisfies $\{ 1 \} \neq K \subseteq
\langle f \rangle$, implying $g \in \langle f \rangle$ and
$\Skel(g)= \Skel(h)$. Thus by
Proposition~\ref{prop_generator_of_skel2}, $g$ is a generator of $
\langle h \rangle$, so   $\langle h \rangle = K \subseteq \langle
f \rangle$ \ as desired.
\end{proof}

\subsection{The wedge decomposition in $\langle F \rangle $}\label{wedge1}$ $

\begin{lem}
Suppose $f = \sum_{i=1}^{k}f_i \in F[x_1,...,x_n]$ is a
supertropical polynomial written as the sum of its   monomials, then
$\tilde{f} := \bigwedge_{i=1}^{k} \widehat{f_i}$ is corner
internal.\end{lem} \begin{proof} $\tilde{f} \geq 1$. For any given
$i \in \{1,...,k\}$,
 then  Remark \ref{rem_corner_integrality_powers} yields
$\Skel(\Phi_{CI}(|\widehat{f_i}|) )= \Skel(\Phi_{CI}(\widehat{f_i}))
= \Skel(\widehat{f_i})$. Then
$$\Skel(\Phi_{CI}(\tilde{f}))= \Skel(\bigwedge_{i=1}^{k}\Phi_{CI}(|\widehat{f_i}|)) =
\Skel(\bigwedge_{i=1}^{k}\widehat{f_i})  $$ so $\Phi_{CI}(\tilde{f})
=
 \tilde{f} =\bigwedge_{i=1}^{k}\widehat{f_i}$, which by Theorem~\ref{cor_CI_map1}
 is corner internal.
\end{proof}

\begin{defn}\label{defn_decomposition}
 A \textbf{wedge decomposition} of a rational function   $f \in F(x_1,...,x_n)$
  is an expression
\begin{equation}\label{eq_decomposition}
|f| =  \bigwedge _i |u_i|
\end{equation}
for $u_i\in F(x_1,...,x_n)$. This wedge decomposition is
\textbf{associated to} a $\nu$-kernel intersection $\langle f
\rangle = \bigcap K_i$ if each $K_i =  \langle u_i \rangle.$
\end{defn}

Note by Lemmas ~\ref{rem_max}(ii) and~\ref{rem_max0}  that $\Skel
(f)
= \bigcup_i \Skel (u_i).$ 

\begin{thm}\label{thm_every_gen_is_reducible}
Any  intersection  $\langle f \rangle  = \bigcap K_i$  of
(principal)
 $\nu$-subkernels of $\langle \mathscr{R} \rangle$  has an
 associated
  wedge decomposition.\end{thm}

\begin{proof} It is enough to prove the assertion when $\langle f \rangle =  \langle u \rangle \cap  \langle v \rangle$
for $u,v \in \langle \mathscr{R} \rangle $, and then to conclude by
induction. Let $\bar f = |u|\wedge |v|.$
     $\langle \bar f \rangle = \langle f \rangle$,
    so there exist some $q_1,...,q_k \in \mathscr{R}(\Lambda)$
    such that $\sum_{i=1}^{k}q_i = 1$ and $|f| = \sum_{i=1}^{k}q_i \bar f^{d(i)}$
    with $d(i) \in \mathbb{Z}_{\geq 0}$ ($d(i) \geq 0$ since $|f| \geq 1$).
Thus $$f = \sum_{i=1}^{k}q_i (|u| \wedge |v|)^{d(i)}   =
\left(\sum_{i=1}^{k}q_i|u|^{d(i)}\right)\bigwedge
\left(\sum_{i=1}^{k}q_i|v|^{d(i)}\right)$$
$$=|g| \wedge |h|$$ where $g = |g| = \sum_{i=1}^{k}q_i|u|^{d(i)}$ and $ h
= |h| = \sum_{i=1}^{k}q_i|v|^{d(i)}$. Thus $$\Skel(f) \supseteq
\Skel(g) \supseteq \Skel(u); \qquad \Skel(f) \supseteq \Skel(h)
\supseteq \Skel(v).$$

We claim that $|g|$ and $|h|$ generate $\langle |u| \rangle $ and
$\langle |v| \rangle$, respectively. Since $\langle \bar f \rangle =
\langle f \rangle$, we see that $\bar f(\bfa)=1 \Leftrightarrow
|f|(\bfa)=1$ for any $\bfa \in \mathscr{R}^{(n)}$. Let $q_j \bar
f^{d(j)}$ be a dominant term of $|f|$ at $\bfa$, i.e.,
$$|f(\bfa)| \nucong \sum_{i=1}^{k}q_i(\bfa) (\bar f(\bfa))^{d(i)} \nucong q_j(\bfa) (\bar f(\bfa))^{d(j)}.$$
Then   $\bar f(\bfa)  \nucong 1 \Leftrightarrow q_j(\bfa) (\bar
f(\bfa))^{d(j)} \nucong 1$. Hence, $\bar f(\bfa)  \nucong 1
\Leftrightarrow q_j(\bfa) = 1$.

Now, consider $\bfa \in \Skel(g)$. Then  $g(\bfa)  =
\sum_{i=1}^{k}q_i|u|^{d(i)}  \nucong 1$. Let $q_t|u|^{d(t)}$ be a
dominant term of $g$ at $\bfa$. If $q_t(\bfa)  \nucong 1$ then
$|u|^{d(t)(\bfa)}  \nucong 1$ and thus  $\bfa \in \Skel(u)$.

So we may assume that $q_t(\bfa) < _{\nu} 1$ (since
$\sum_{i=1}^{k}q_i \nucong 1$) and so, as above, $q_t \bar f^{d(t)}$
is not a dominant term of $|f|$ at $\bfa$. Hence,
$$|u(\bfa)|^{d(j)} = q_j(\bfa)|u(\bfa)|^{d(j)} <_{\nu} q_t(\bfa)|u(\bfa)|^{d(t)} \nucong g(\bfa) \nucong 1,$$
for any dominant term of $|f|$ at $\bfa$. Thus
\begin{equation}\label{eq_1_thm_every_gen_is_reducible}
q_j(\bfa)(\bar f(\bfa))^{d(j)}= q_j(\bfa) (|u|(\bfa) \wedge
|v|(\bfa))^{d(j)} \leq q_j(\bfa)|u(\bfa)|^{d(j)} <_{\nu} 1.
\end{equation}
On the other hand, $\bar f(\bfa)\nucong 1$ since $\Skel(f) \supseteq
\Skel(g)$,
  and thus
$q_j(\bfa)(\bar f(\bfa))^{d(j)} \nucong 1$, contradicting
\eqref{eq_1_thm_every_gen_is_reducible}. Hence $\Skel(g) \subseteq
\Skel(u)$, so, by the above, $\Skel(g) = \Skel(u)$, which in turn
yields that $g$ is a generator of $\langle |u| \rangle = \langle u
\rangle$, by Proposition \ref{prop_generator_of_skel2}. The proof
for $h$ and $|v|$ is analogous. Consequently, we have that $\Skel(g)
= \Skel(|g|) = \Skel(|u|)$ and $\Skel(h)= \Skel( |v|)$, as desired.
\end{proof}

\subsection{Example: The tropical line revisited}$ $

For ease of notation, we write $x$ for $\la_1$ and $y$ for $\la _2$.

\begin{note}
In the following example we consider the rational function
$$\widehat{f} = \left|\frac{x}{y  \dotplusss   1}  \dotplusss
\frac{y}{x  \dotplusss   1}  \dotplusss   \frac{1}{x  \dotplusss
y}\right| \wedge |\alpha| \in \langle \mathscr{R} \rangle$$ for any
$\alpha \in \mathscr{R} \setminus \{1 \}$.
\end{note}

\begin{figure*}
\centering

\begin{tikzpicture}[scale=3,
    axis/.style={thin, ->},
    important line/.style={very thick, ->},
    gap line/.style={thin},
    dline/.style={thick,dashed}]


    \draw[axis] (0,0)  -- (1,0) node[right, color=black] {$x$} ;
    \draw[axis] (0,0) -- (0,1) node[above, color=black] {$y$};
    \draw[important line] (0,0)  -- (-1,0) ;
    \draw[important line] (0,0) -- (0,-1) ;
    \draw[important line] (0,0) -- (0.7,0.7) ;

    \draw[dline] (-0.1,-1.1)  -- (-0.1,-0.1) -- (-1.1,-0.1) ;
    \draw[dline] (0.1,-1.1)  -- (0.1,-0.07) -- (1,0.8) ;
    \draw[dline] (-1.1,0.1)  -- (-0.1, 0.1) -- (0.8,1)   ;

    \node[left] at (0.5,-0.25) {$|\frac{y}{x \dotplusss  1}|$};
    \node[right] at (-0.5,0.25) {$|\frac{x}{y \dotplusss  1}|$};
    \node[left] at (-0.25, -0.25) {$|\frac{1}{x \dotplusss  y}|$};

 \end{tikzpicture}
\caption{$\tilde{f} = |\frac{x}{y  \dotplusss   1}| \wedge
|\frac{y}{x  \dotplusss   1}| \wedge |\frac{1}{x  \dotplusss   y}|$}
\label{fig:Ex1}
\end{figure*}

\begin{exmp}\label{tropline1}
Let $f = x  \dotplusss   y  \dotplusss   1$ be the tropical line,
considered already in Example~\ref{onenusetex}(i). Its
corresponding $\onenu$-set is defined by the rational function
$\widehat{f} = \frac{x}{y \dotplusss   1} \dotplusss \frac{y}{x
\dotplusss   1} \dotplusss \frac{1}{x \dotplusss y}$, and so, its
corresponding $\nu$-kernel in $\mathscr{R}(x,y)$ is $\langle
\widehat{f} \rangle$ and the bounded copy in $\langle \mathscr{R}
\rangle$ is $\langle |\widehat{f}| \wedge |\alpha| \rangle =
\langle \widehat{f} \rangle \cap \langle \mathscr{R} \rangle$ . As
mentioned above $\widehat{f} \sim_{\langle \mathscr{R} \rangle}
\left|\frac{x}{y  \dotplusss 1}\right| \wedge \left|\frac{y}{x
\dotplusss   1}\right| \wedge \left|\frac{1}{x \dotplusss
y}\right|$, and any of the three above terms  can be omitted. Thus
we have that
$$\langle \widehat{f} \rangle  \cap \langle \mathscr{R} \rangle = \left\langle \frac{x}{y  \dotplusss   1} \right\rangle \cap \left\langle \frac{y}{x  \dotplusss   1} \right\rangle \cap \left\langle x  \dotplusss   y \right\rangle  \cap \langle \mathscr{R} \rangle,$$
where each of the $\nu$-kernels comprising the intersection
(excluding $\langle \mathscr{R} \rangle$) is contained in each of
the other $\nu$-kernels. (In the third $\nu$-kernel from the left
we choose to take $x \dotplusss y$ as a generator instead of its
inverse.) Now, taking logarithms, it can be seen that
$\Skel(\langle x \dotplusss y \rangle)$ is precisely the union of
the bounding rays of the third quadrant.
 Furthermore,  $$\langle x \dotplusss  y \rangle = \langle |x \dotplusss  y| \dotplusss  |x| \rangle \cap \langle  |x \dotplusss  y|  \dotplusss   |y|
 \rangle,
$$ since
$$x \dotplusss  y \sim_{\langle \mathscr{R} \rangle} |x \dotplusss  y| = |x \dotplusss  y| \dotplusss   (|x| \wedge |y|) = \left(|x \dotplusss  y| \dotplusss  |x|\right) \wedge \left(|x \dotplusss  y|  \dotplusss   |y|\right).$$

This wedge decomposition of $|x  \dotplusss y|$  is   quite natural.
The geometric locus of the equation $|x| \wedge |y|
 = 1$ in logarithmic scaling is precisely the union of the
 $x$-axis corresponding to $|x|\nucong 1$ and the $y$-axis corresponding to $|y|\nucong 1$.
 \end{exmp}

 Using the wedge decomposition, we can define any segment and ray in
$\mathscr{R}^2$  by means of principal $\nu$-kernels, so the only
 irreducible $\onenu$-sets turn out to be the points in the plane.
 This does not hamper us in developing geometry,
 since we may also restrict our attention to
  sublattices of the lattice of principal $\nu$-kernels.

\section{Hyperspace-$\nu$-kernels and region $\nu$-kernels}\label{thmpf}

In order to prove Theorem~\ref{thmcireglattice}, we separate
principal $\nu$-kernels into two classes of $\nu$-kernels,
HO-kernels and bounded from below $\nu$-kernels.The latter
$\nu$-kernels are eliminated by passing to the $\nu$-kernel
$\langle F \rangle$. Using these classes of $\nu$-kernels we
introduce the hyperspace-region decomposition given in
Theorem~\ref{thm_HP_expansion}.
\subsection{Hyperspace-kernels.}$ $

\begin{rem}
Though we consider the \vsemifield0 of rational functions
$F(\la_1,...,\la_n)$, most of the results introduced in this
section are applicable  to any finitely generated sub-\vsemifield0
of $\overline{F(\Lambda)}$ over $F$. In particular, $\langle F
\rangle \subset F(\la_1,...,\la_n)$ is just another case of a
finitely generated \vsemifield0 over $F$, taking  the generators $
\la_i \wedge |\alpha|$ for $1 \leq i \leq n$, and $\alpha \in F
\setminus \{1\}$. In this case, the generators  are bounded, and
we specifically designate the results that are true only for
unbounded generators.
\end{rem}

\begin{defn}
An $\scrL$-\textbf{monomial} is a non-constant Laurent monomial $f
\in F(\la_1,...,\la_n)$; i.e.,  $f = \frac{h}{g}$ with $h,g \in
F[\la_1,...,\la_n]$ non-proportional monomials.
\end{defn}
%
%

\begin{rem}\label{rem_HP_fraction_not_bounded} Whenever $F$
is $\nu$-archimedean, $\scrL$-monomials in $F(\la_1,...,\la_n)$ are not
bounded; i.e., for any $\scrL$-monomial~$f$ there does not exist
$\alpha \in F$   for which  $|f| \leq |\alpha|$.
\end{rem}
%

\begin{defn}
A rational function  $f \in F(\la_1,...,\la_n)$ is  a
\textbf{hyperspace-fraction}, or HS-fraction, if $f \sim_{\FF}
\sum_{i=1}^{t} |f_i|$ where  the  $f_i $ are non-proportional
$\scrL$-monomials.
\end{defn}

%

\begin{rem}\label{rem_HS_fraction_not_bounded}
HS-fractions in $F(\la_1,...,\la_n)$ are not bounded, since $|f_j|
\leq \sum_{i=1}^{t} |f_i| = f$.
\end{rem}

\begin{defn}
A \textbf{hyperplane $\nu$-kernel}, or  \textbf{HP-kernel}, for
short, is is a principal $\nu$-kernel of $F(\la_1,...,\la_n)$
generated by an $\scrL$-monomial.

A \textbf{hyperspace-fraction $\nu$-kernel}, or
\textbf{HS-kernel}, for short, is is a principal $\nu$-kernel of
$F(\la_1,...,\la_n)$ generated by a hyperspace fraction.
\end{defn}

 By definition,
any HP-kernel is regular. Also, a fortiori, every HP-kernel is an
HS-kernel.

\begin{defn}\label{Omega}
 $\Omega(\overline{F(\Lambda)})$ is the lattice of
$\nu$-kernels finitely generated by  HP-kernels of
$\overline{F(\Lambda)}$, i.e., every element $\langle f \rangle
\in \Omega(\overline{F(\Lambda)})$ is obtained via finite
intersections and products of HP-kernels.
\end{defn}

\begin{prop}\label{rem_HS_prod_HP}
Any principal HS-kernel is a product of distinct HP-kernels, and
thus is in $\Omega(\overline{F(\Lambda)})$.
\end{prop}
\begin{proof}
If $\langle f \rangle$ is an HS-kernel, then $\langle f \rangle =
\left\langle \sum_{i=1}^{t} |f_i| \right\rangle = \prod_{i=1}^{t}
\langle f_i \rangle$ with  $f_i\in F[\la_1,...,\la_n]$
non-proportional $\scrL$-monomials.
\end{proof}

\begin{cor}
Any HS-Kernel is regular.
\end{cor}
\begin{proof}
Apply Proposition~\ref{cor_regular_lattice}.
\end{proof}

\begin{defn}
A $\onenu$-set  in $F^n$ is  a \textbf{hyperplane $\onenu$-set }
(HP-$\onenu$-set for short) if it is defined by an $\scrL$-monomial.
A $\onenu$-set  in $F^n$ is   a \textbf{hyperspace-fraction
$\onenu$-set } (HS-$\onenu$-set for short) if it is defined by an
HS-fraction.
\end{defn}

\begin{cor}
A $\onenu$-set  is an HS-$\onenu$-set  if and only if it is an
intersection of HP-$\onenu$-sets.
\end{cor}
\begin{proof}
As $\Skel(\langle f \rangle    \langle g \rangle ) = \Skel(\langle f
\rangle)  \cap  \Skel(\langle g \rangle )$, the assertion follows
directly from Proposition \ref{rem_HS_prod_HP}.
\end{proof}

\subsection{Region $\nu$-kernels.}$ $

Our last kind of $\nu$-kernel marks out a tropical region, and
thus is called a ``region $\nu$-kernel.''  It also relates to
information lost by passing from $F[\Lambda]$  to
$\overline{F[\Lambda]}$.

\begin{lem}\label{prop_HP_element_is_a_generator}
Let $\langle f \rangle$ be an HP-kernel, with $F$ divisible. If $w
\in \langle f \rangle$ is an $\scrL$-monomial, then $w^s = f^k$ for
some $s,k \in \mathbb{Z} \setminus\{ 0 \}$.
\end{lem}
\begin{proof}
  By assumption $\langle w \rangle \subseteq \langle
f \rangle$, implying $\Skel(w) \supseteq \Skel(f)$. Assume that $w^s
\neq f^k$ for any $s,k \in \mathbb{Z} \setminus\{ 0 \}$. We will
show that there exists some $\bfb \in F^n$ such that $\bfb  \in
\Skel(f) \setminus \Skel(g)$ for some $g \in \langle f \rangle$,
which clearly is impossible.

Let $(p_1,...,p_n), (q_1,...,q_n) \in \mathbb{Z}^n$ be the vectors
of degrees of $\la_1,...,\la_n$ in the Laurent monomials $f$ and
$w$. Since $w$ and $f$ are nonconstant, $(p_1,...,p_n),
(q_1,...,q_n)~\neq~(0)$. Since $F$ is divisible, we may take
appropriate roots and assume that $gcd(p_1,...,p_n) =
gcd(q_1,...,q_n) = 1$.   Since $w \in \langle f \rangle$ we have
$|w| \leq |f|^m$ for some $m \in \mathbb{N}$, so if $\la_i$ occurs
in $w$ it must also occur in $f$. Finally, since $w^s \neq f^k$ for
any $k \in \mathbb{Z} \setminus\{ 0 \}$ we may also assume that
$(p_1,...,p_n) \neq (q_1,...,q_n)$, for otherwise $w = \alpha f$ for
some $\alpha \neq 1$ and thus
$$\Skel(f) \subseteq \Skel(f) \cap \Skel(w) = \Skel(f) \cap \Skel(\alpha f) =\emptyset,$$
a contradiction. Suppose that ${\la_j}^l$ occurs in $w$ for some $l
\in \mathbb{Z} \setminus \{ 0 \}$ such that $\la_j$ is not
identically $1$ on $\Skel(w)$ (and thus also on $\Skel(f)$). Thus
there exists some $k \in \mathbb{Z} \setminus \{ 0 \}$ such that
${\la_j}^{k}$ occurs in $f$. Then $\la_j$ does not occur in the
Laurent monomial $g : = w^{-k}f^\ell \in \langle f \rangle$. Without
loss of generality, assume that $j = 1$. If
$$\bfa = (\alpha_1,...,\alpha_n) \in \Skel(f),$$ then $g(\bfa)=
w(\bfa)^{-1}f(\bfa) = 1$. By assumption that $w^s \neq f^k,$ there
exists $\la_t$ occurring in $f$  and not in $w$. Take $\bfb =
(1,\alpha_2,...,\beta,..., \alpha_n)$ with $\beta \in F$ occurring
in the $t$-th component, such that
$\beta^{p_t}=\frac{\alpha_{t}^{p_t}}{\alpha_{1}^{p_1}}$. Then as
$\la_j$ is not identically $1$ over $\Skel(f)$, we can choose $\bfb
\in \Skel(w)$ such that  $f(\bfb) = 1$ but $g(\bfb) \neq 1$.
\end{proof}

\begin{prop}\label{cor_HP-element_is_a_generator}
Let $ f  $ be an $\scrL$-monomial. Then $w \in \langle f \rangle$ is
an $\scrL$-monomial if and only if $w$ is a generator of $\langle f
\rangle$.
\end{prop}
\begin{proof}
The claim follows from Lemma~\ref{prop_HP_element_is_a_generator}
and the property that $\langle g^k \rangle = \langle g \rangle$.
\end{proof}

\begin{defn}
The  \textbf{$\scrL$-binomial} $o$ defined by an $\scrL$-monomial
$f$ is the rational function   $1 +   f$.

The \textbf{complementary $\scrL$-binomial} $o^{c}$ of $o$ is $ 1 +
f^{-1}$. By definition $(\mathcal O^c)^c = \mathcal O$.

The \textbf{order $\nu$-kernel} of the \vsemifield0
$F(\la_1,...,\la_n)$ defined by $f$ is the principal $\nu$-kernel
$\mathcal O = \langle o \rangle$ for the $\scrL$-binomial $o = 1 +
f$.

 The \textbf{complementary order $\nu$-kernel} $\mathcal O^c$  of $\mathcal O$ is $ \langle o^c \rangle$.
\end{defn}

Since  $f$ is an $\scrL$-monomial, so is $f^{-1}$ and thus
$\mathcal O^{c}$ is an order $\nu$-kernel.

\begin{lem}\label{rem_complementary_order}
Let $\mathcal O = \langle 1+f \rangle$ be an order $\nu$-kernel of
$F(\la_1,...,\la_n)$. Then
$$\mathcal O \cap \mathcal O^{c} = \langle 1 \rangle \ \text{and} \  \mathcal O  \mathcal O^{c} = \langle f \rangle.$$
\end{lem}
\begin{proof}
 $\mathcal O \cap \mathcal O^{c} = \langle |1  + f| \wedge |1 +
f^{-1}| \rangle =   \langle 1 \rangle$ and $\mathcal O  \mathcal
O^{c}  = \langle |1 +   f| +  |1  +  f^{-1}| \rangle = \langle 1 + f
+  f^{-1} \rangle = \langle 1 +  |f| \rangle = \langle |f| \rangle =
\langle f \rangle$ (noting that $(1 + f),(1 + f^{-1}) \geq_\nu 1$
implies $|1  +  f| \nucong 1 +  f$ and $|1 + f^{-1}| \nucong 1 +
f^{-1}$).
\end{proof}

\begin{defn}\label{regionfrac}
A rational function  $f \in F(\la_1,...,\la_n)$ is said to be a
\textbf{region fraction}   if   $\Skel (f)$ contains some nonempty
open interval.
\end{defn}

\begin{lem}
 $f \sim{\FF} \sum_{i=1}^{t} |o_i|$ is a region fraction iff,
 writing $o_i = 1+f_i$ for $\scrL$-monomials $f_i$, we have $f_i \not\nucong f_j^{\pm1}$ for every $i\ne
 j.$\end{lem} \begin{proof} Since $1 +   f_i \geq_\nu 1$ for every $i$, we
have that
$$\sum_{i=1}^{t} |o_i| = \sum_{i=1}^{t} |1  +   f_i| = \sum_{i=1}^{t} (1  +   f_i) = 1  +   \sum_{i=1}^{t}f_i.$$
 Thus a region fraction $r$ can be defined as $r \sim_{\FF} 1  +   \sum_{i=1}^{t}f_i$, so the last condition of the definition can be stated as $f_j \neq f_i^{-1}$ for any $1 \leq i,j \leq t$.
 If there exist $k$ and $\ell$ for which $f_\ell \nucong  f_k^{-1}$,
 then
 $$\sum_{i=1}^{t} |o_i| \nucong |f_k|  +   (1   +   \sum_{i \neq k,\ell}f_i) \nucong |f_k|  +   \Big|1   +   \sum_{i \neq k,\ell}f_i \Big|,$$
and thus $\Skel(r) = \Skel(\sum_{i=1}^{t} |o_i|) = \Skel(f_k) \cap \Skel(1   +   \sum_{i \neq k,\ell}f_i) \subseteq \Skel(f_k)$.\\
\end{proof}

\begin{defn}\label{regker}
A \textbf{region $\nu$-kernel} is a principal $\nu$-kernel
generated by a region fraction.
\end{defn}

\begin{lem}\label{rem_Region_prod_Order} A principal $\nu$-kernel
$K$ is a region $\nu$-kernel if and only if it has the form
$$K = \prod_{i=1}^{v}\mathcal O_i$$
for order $\nu$-kernels $\mathcal O_1,...,\mathcal O_v \in
\PCon(F(\la_1,...,\la_n))$.
\end{lem}
\begin{proof}
If $f = \sum |o_i|$ is a generating region fraction of $K$, then
$$K = \langle f \rangle = \langle \sum_{i=1}^{t} |o_i| \rangle~=~\prod_{j=1}^{v}\langle o_i \rangle.$$

Conversely, if $K = \prod_{i=1}^{v}\mathcal O_i$ then  writing
$\mathcal O_i = \langle o_i \rangle$, we see that $  \sum_{i=1}^{t}
|o_i|$ is a region fraction generating $K$, since $\left\langle
\sum_{i=1}^{t} |o_i| \right\rangle  = \prod_{i=1}^{v}\mathcal O_i =
K$.
\end{proof}

%

\begin{lem}\label{lem_HP_and_Oreder_kernels_are_CI}
Any order $\nu$-kernel $f = 1+ \frac hg$ (for monomials $g$ and
$h$)  is corner-internal.
\end{lem}
\begin{proof}  $f = \frac{g + h}{g}$. Any ghost root $\bfa$  of $\bar f
= g^\nu + h$ is either a ghost root of $g$ dominating $h$, or a
ghost root of $g + h$; in either case it is a  $\nu$-kernel root
of $f$.
\end{proof}

\begin{lem}\label{rem_regularity_of_HP_times_order_kernels}
If $\langle f \rangle \ne \langle 1 \rangle$ is a regular
$\nu$-kernel and $\langle o \rangle$ is an order $\nu$-kernel,
then  $ \langle f \rangle
 \langle o \rangle$ is regular.
\end{lem}
\begin{proof} Write $f = \frac{h}{g}$ and $o = 1  +
\frac{h'}{g'}$ with $h'$ and $g'$ monomials in
$F[\la_1,...,\la_n]$. Since regularity does not depend on the
choice of  generator of the $\nu$-kernel and since $$ \langle f
\rangle  \langle o \rangle = \langle |f| +  |o| \rangle = \langle
|f| + o\rangle,$$
  we check the condition on $|f|  + o$.  Then $|f| + o =
|\frac{h}{g}| +   1 + \frac{h'}{g'}$. Since  $|\frac{h}{g}| \geq_\nu
1$ we have
$$|f|  +   o = \left(\left|\frac{h}{g}\right|  +   1\right)  +   \frac{h'}{g'}
= \left|\frac{h}{g}\right|  +   \frac{h'}{g'} \nucong \frac{(h^2  +
g^2)g' + ghh'}{ghg'}.$$ Since $g\ne h$, we are done unless $g'>_\nu
h',$
and Frobenius   enables us to reduce to $\frac{h^2  + g^2}{gh} =
|f|$. But then we are done since $|f|$ is regular, by
Lemma~\ref{rem_regular_closed_operations}.
\end{proof}

\begin{prop}\label{rem_regularity_of_HP_times_order_kernels_generalization}
$\langle f \rangle  \langle o_1 \rangle \cdots \langle o_k
\rangle$ is regular, for any HP-kernel $\langle f \rangle \neq 1$
and order $\nu$-kernels $\langle o_1 \rangle, ... , \langle o_k
\rangle$.
\end{prop}
\begin{proof} Iterate Lemma~\ref{rem_regularity_of_HP_times_order_kernels},
noting that every  HP-kernel is regular.
\end{proof}

\subsection{Geometric interpretation of HS-kernels and region $\nu$-kernels}\label{subsection:geometric_interpretation}
 $ $


\begin{rem}
In view of Theorem~\ref{thm_holder_translated}, for the case that
$F$ is a divisible, $\nu$-archimedean supertropical \vsemifield0,
we may take $F= \mathscr{R}$.
\end{rem}

By Remark~\ref{log0}, any $\scrL$-monomial $f$ may be considered as
a linear functional over $\mathbb{Q}$ and thus the HP-kernel given
by the equation $f = 1$ over $\mathscr{R}$ translates to $\frak f =
0$ over $(\mathscr{R}, +)^{(n)}$ (in logarithmic notation), where
$\frak f$ is the linear functional obtained from $f$ by applying
\eqref{eq_vector_space_translation}.

\begin{lem}
If $f$ is an $\scrL$-monomial in $F(\la_1,...,\la_n)$, then $f$ is
completely determined by the set of vectors $\{ w_0, ...,w_n \}$ for
any $w_i =
(\alpha_{i,1},...,\alpha_{i,n},f(\alpha_{i,1},...,\alpha_{i,n})) \in
F^{(n+1)}$ where   $\{ a_i = (\alpha_{i,1},...,\alpha_{i,n}) \}
\subset F^{(n)}$ such that $w_0,....,w_n$ are in general position
(are not contained in an $(n-1)$-dimensional affine subspace of
$F^{(n+1)}$).
\end{lem}
\begin{proof}
Writing $f(\la_1,...,\la_n) = \alpha \prod_{i=1}^{n}\la_i^{k_i}$
with $k_i \in \mathbb{Z}$, we have $\alpha  =
f(a_{0})\prod_{i=1}^{n}a_{0,i}^{-k_i}$. After $\alpha$ is
determined, since $w_0,....,w_n$ are in general position the set
$$\left\{a_1,...,a_n, b = \left(f(a_1),...,f(a_n)\right)\right\}
\subset F^{(n)}$$ define a linearly independent set of $n$ linear
equations in the variables $k_i$, and thus determine them uniquely.
\end{proof}

Consider the HS-kernel of $F(\la_1,...,\la_n)$ defined by the
HS-fraction $f = \sum_{i=1}^{t}|f_i|$ where $f_1,...,f_t$ are
$\scrL$-monomials. Then $f(\bfa) \nucong 1$ if and only if
$f_i(\bfa) \nucong  1$ for each $i = 1,...,t$, which translates to
a homogenous system of  linear equations (over $ \mathbb Q$) of
the form $\frak f_i = 0$ where $\frak f_i$ is the logarithmic form
of $f_i$. This way $\Skel(f) \subset F^{(n)} \cong (F^{+})^{n}$ is
identified with an affine subspace of $F^{(n)}$ which is just the
intersection of the $t$ affine hyperplanes defined by $\frak f_i =
0, \ 1 \leq i \leq t$. Analogously, the $\scrL$-binomial $o = 1  +
g$ has $\nu$-value $1$ if and only if $g \leq_\nu 1$,  giving rise
to the half space of $\mathbb{R}$ defined by the weak inequality
$\frak g \leq 0$. Thus, the region $\nu$-kernel defined by $r =
\sum_{i=1}^{t}|o_i| = 1 + \sum_{i=1}^{t}g_i$ where $o_i = 1 + g_i$
are $\scrL$-binomials, yields the nondegenerate polyhedron formed
as the intersection of the affine half spaces each  defined using
the
  $\frak g_i$ corresponding to  $o_i.$

\begin{figure}
\centering

\begin{tikzpicture}[scale=2]

  \tikzstyle{axis}=[thin, ->]
  \tikzstyle{dline}=[very thick, ->]
  \tikzstyle{important line}=[very thick, dashed, ->]

    \colorlet{anglecolor}{green!50!black}

    \filldraw[fill=gray!40, draw = gray!40 ] (0,-1) -- (0,1) -- (-1,1) -- (-1,-1);
    \draw[dline] (0,0)  -- (0.5,0) node[above] {$\{\la_2=1\}$} -- (1,0)  node[right] {$\la_1$} ;
    \draw[axis] (0,-1) -- (0,1)   node[above] {$\la_2$} ;
    \draw[important line]  (0,0)  -- (-0.5,0) node[above] {$X$} -- (-1,0);

    \node[left] at (-0.2,0.5) {$\{ x \leq 1 \}$};

 \end{tikzpicture}
\caption{Order relations} \label{fig:Ex2}
\end{figure}
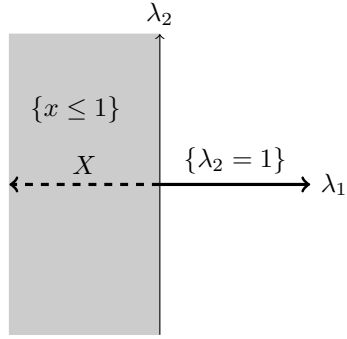

\subsection{The HO-decomposition}\label{HOdecomp}$ $


Since our next major goal, Theorem~\ref{thm_HP_expansion}, is
somewhat technical, we start this section with a motivating
example. Any point $\bfa = ( \alpha_1,...,\alpha_n)$ in $F^{(n)}$
is just $ \Skel(f_\bfa)$, where
$$f_\bfa(\la_1,...,\la_n) = \left|\frac{\la_1}{\alpha_1}\right| +
\dots  + \left|\frac{\la_n}{\alpha_n}\right| \in
F(\la_1,...,\la_n).$$ We would like $\langle f_\bfa \rangle$ to
encode the reduction of dimension from $F^{(n)}$ to the point
$\{\bfa \}$.

For each $1 \leq k \leq n$ define $f_{k,\bfa} =
\left|\frac{\la_1}{\alpha_1}\right|  +   \dots  +
\left|\frac{\la_k}{\alpha_k}\right|$ and $f_0 = 1$, and consider the
 chain of principal HS-kernels
\small{\begin{equation}\label{eq_chain_of_origin} \langle f_{\bfa}
\rangle = \langle f_{n,\bfa} \rangle   \supset \langle f_{n-1,\bfa}
\rangle \supset \dots \supset \langle f_{1,\bfa} \rangle \supset
\{1\} .
\end{equation}}
The factors  $\langle f_{k,\bfa} \rangle / \langle f_{k-1,\bfa}
\rangle$ are the quotient \semifields0
$$ \prod_{j=1}^{k}\left\langle \frac{\la_j}{\alpha_j} \right\rangle \bigg/ \prod_{j=1}^{k-1}\left\langle \frac{\la_j}{\alpha_j} \right\rangle \cong \left\langle \frac{\la_k}{\alpha_k} \right\rangle  \bigg/ \left(\left(\prod_{j=1}^{k-1}\left\langle \frac{\la_j}{\alpha_j} \right\rangle \right) \cap \left\langle \frac{\la_k}{\alpha_k} \right\rangle\right),$$
 a nontrivial homomorphic  image of an
HP-kernel of the quotient \vsemifield0 $F(\la_1,...,\la_n) \bigg/
\left(\left(\prod_{j=1}^{k-1}\langle \frac{\la_j}{\alpha_j}
\rangle \right) \cap \langle \frac{\la_k}{\alpha_i}
\rangle\right).$

We claim that this chain of HS-kernels can be refined to a longer
descending chain of principal   $\nu$-kernels descending from
$\langle f_a \rangle$. Indeed, the $\nu$-kernels $\langle |\la_1|
+ |\la_2| \rangle$ and
 $\langle |\la_1| \rangle = \langle \la_1 \rangle$
 both are \semifields0,  and
 $\langle |\la_1| \rangle$ is a $\nu$-subkernel of $\langle |\la_1|  +   |\la_2| \rangle$. Consider the substitution map $\phi$ sending $\la_1$ to $1$. Then $\Im(\phi) = \langle 1  +   |\la_2| \rangle = \langle |\la_2| \rangle_{F(\la_2)}$. The $\nu$-kernel $\langle |\la_2| \rangle$ is not
 simple as a principal $\nu$-kernel of the \vsemifield0  $\langle |\la_2|
 \rangle_{F(\la_2)}$, for
  the chain  $\langle |\la_2| \rangle \supset \langle |1  +   \la_2| \rangle \supset \langle 1 \rangle$ is the image of the refinement  $$\langle |\la_1|  +   |\la_2| \rangle \supset \langle |\la_1  +   \la_2|  +   |\la_1| \rangle \supset \langle |\la_1| \rangle$$ (since  $\phi(|\la_1  +   \la_2|  +   |\la_1|) = |\phi(\la_1)  +   \phi(\la_2) |  +   | \phi(\la_1)| = |1  +   \la_2|  +   |1|= |1 +  \la_2|$).

  On  the other hand, $\langle 1  +   \la_2 \rangle$ is an order $\nu$-kernel which induces the order
  relation $\la_2 \leq 1$ on the \vsemifield0 $\langle |\la_2|
  \rangle$.

In view of these considerations, we would like to exclude order
$\nu$-kernels and ask:
\begin{itemize}
\item Can \eqref{eq_chain_of_origin}   be refined to a longer descending chain of HS-kernels descending from $\langle f_a \rangle$?  \item Are the lengths of descending chains of HS-kernels beginning at $\langle f_{a} \rangle$ bounded? \item Can any chain of  HS-kernels be refined to such a chain of maximal length? \\
\end{itemize}
We provide answers to these three questions, for which the chain
\eqref{eq_chain_of_origin} is of maximal unique length common to
all chains of HS-kernels descending from $\langle f_a \rangle$.
Our method is to provide an explicit decomposition of a principal
$\nu$-kernel $\langle f \rangle$ as an intersection of
$\nu$-kernels of two types: The first, called an
\textbf{HO-kernel}, is a product of an HS-kernel and a region
$\nu$-kernel. The second is a  product of a region $\nu$-kernel
and a bounded from below $\nu$-kernel.  Whereas the first type
defines the $\onenu$-set  of $\langle f \rangle$, the second type
corresponds to the empty set and thus has no effect on the
geometry. This latter type is the source of ambiguity in relating
a $\onenu$-set to a $\nu$-kernel, preventing the $\nu$-kernel of
the $\onenu$-set from being principal.

When intersected with $\langle F \rangle$, the factors in the
decomposition coming from $\nu$-kernels of the second type are
degenerate. Restriction to $\langle F \rangle$ thus removes our
ambiguity,
 and each HO-kernel (intersected with $\langle F \rangle$) is in
  $1:1$ correspondence with its $\onenu$-set
    (the segment in the $\onenu$-set  defined by $\langle f \rangle$).
    Then the `HO-part' is unique and independent of the choice of the $\nu$-kernel generating the $\onenu$-set.

Geometrically, the decomposition to be described below is just the decomposition
 of a principal $\onenu$-set  defined by~$\langle f \rangle$ to its ``linear" components.
 Each   component is obtained
  by bounding an affine subspace of $F^{(n)}$ defined by an appropriate HS-fraction
  (which in turn generates an HS-kernel)  using a region fraction
  (generating a region $\nu$-kernel).
  Although the HS-fraction and region fraction defining each segment may vary when we pass from one generator of a principal $\nu$-kernel to the other,
  the HS-kernels and region $\nu$-kernels are independent of the choice of generator, for they correspond to the   components of the $\onenu$-set  of $\langle f \rangle$.

\begin{constr}\label{HO_construction}
Take a rational function  $f \in F(\la_1,...,\la_n)$ for which
$\Skel(f)\ne \emptyset$. Replacing $f$ by $|f|$, we may assume that
$f \geq_\nu 1$. Write $f = \frac{h}{g} = \frac{\sum_{i=1}^k
h_i}{\sum_{j=1}^m g_j}$ where $h_i$ and $g_j$ are monomials in
$F[\la_1,...,\la_n]$. For each $\bfa\in \Skel(f)$, let
$$H_a \subseteq H = \{ h_i \ : \ 1 \leq i \leq k \}; \qquad G_a \subseteq G = \{ g_j \ : \ 1 \leq j \leq m \}$$
be the sets of dominant monomials   at $\bfa$; thus,  $ h_i(\bfa) =
g_j(\bfa) $ for any $h_i \in H_\bfa$ and $g_j \in G_\bfa$. Let
$H_{\bfa}^{c} = H \setminus H_\bfa$ and $G_{\bfa}^{c} = G \setminus
G_\bfa$. Then, for any $h' \in H_\bfa$ and $h'' \in H_{\bfa}^{c}$, $
h'(\bfa)  + h''(\bfa) = h'(\bfa),$ or, equivalently,    $1 +
\frac{h''(\bfa)}{h'(\bfa)} = 1$. Similarly, for any $g' \in G_\bfa$
and $g'' \in G_{\bfa}^{c}$, $g'(\bfa)  + g''(\bfa) = g'(\bfa)$  or,
equivalently,  $1 + \frac{g''(\bfa)}{g'(\bfa)} = 1$.

Thus for any such $\bfa$ we obtain the  relations
\begin{equation}\label{eq_relations1}
\frac{h'}{g'} = 1, \qquad \forall h' \in  H_{\bfa},\ g' \in
G_{\bfa},
\end{equation}
\begin{equation}\label{eq_relations2}
1  +   \frac{h''}{h'} = 1  ; \ 1  +   \frac{g''}{g'} = 1,  \qquad
\forall h' \in H_{\bfa},\ h'' \in H_{\bfa}^{c},\ g' \in G_{\bfa},\
g'' \in G_{\bfa}^{c}.
\end{equation}
As $\bfa $ runs over $\Skel(f)$,  there are only finitely many
possibilities for $H_a$ and $G_a$ and thus for the relations in
\eqref{eq_relations1} and~\eqref{eq_relations2}; we denote these as
$(\theta_1(i), \theta_2(i)),$ $i = 1,...,q$.

In other words, for any $1 \le i \le q,$ the pair $(\theta_1(i),
\theta_2(i))$ corresponds to a $\nu$-kernel $K_i$ generated by the
corresponding elements
$$\frac{h'}{g'}  ,\ \left(1  +   \frac{h''}{h'}\right),
\ \text{and} \  \left(1  +   \frac{g''}{g'}\right),$$ where $\{
\frac{h'}{g'} = 1 \}  \in \theta_1$ and $\{ 1  +   \frac{g''}{g'} =
1\},\ \{ 1  +   \frac{h''}{h'} = 1\} \in \theta_2$.

Reversing the argument, every point satisfying one of these $q$ sets
of relations  is in $\Skel(f)$. Hence,
\begin{equation}\label{localexpansion} \Skel\left(\langle f \rangle
\cap \langle F \rangle \right) =\Skel(f) =
\bigcup_{i=1}^{q}\Skel(K_i) = \bigcup_{i=1}^{q}\Skel\left(K_i \cap
\langle F \rangle\right)\end{equation} $$=
\Skel\left(\bigcap_{i=1}^{q}(K_i \cap \langle F \rangle)\right),$$
Hence  $\langle f \rangle \cap \langle F \rangle =
\bigcap_{i=1}^{q}K_i \cap \langle F \rangle$, since $\langle f
\rangle \cap \langle F \rangle, \ \bigcap_{i=1}^{q}K_i \cap \langle
F \rangle \in \PCon(\langle F \rangle)$.
 $\bigcap_{i=1}^{q}K_i$ provides a local description of $f$ in a neighborhood of its $\onenu$-set .

Let us view this construction globally. We used the $\onenu$-set
of $\langle f \rangle$ to construct $\bigcap_{i=1}^{q}K_i$.
Adjoining various points $\bfa$ in $F^{(n)}$ might add some
regions, complementary to the regions defined by
\eqref{eq_relations2} in $\theta_2(i)$ for $i=1,...,q$, over which
$\frac{h'}{g'} \neq 1, \forall h' \in H_{\bfa}, \forall g' \in
G_{\bfa}$ for each $\bfa$, i.e., regions over which the dominating
monomials never agree. Continuing the construction above using
$\bfa \in F^{(n)} \setminus \Skel{(f)}$ similarly produces a
finite collection of, say $t \in \mathbb{Z}_{\geq 0}$,
$\nu$-kernels generated by elements from \eqref{eq_relations2} and
their complementary order fractions and by elements of the form
\eqref{eq_relations1} (where now $\frac{h'}{g'} \neq 1$ over the
region considered). Any
 principal $\nu$-kernel $N_j = \langle q_j \rangle$,
   $1 \leq j \leq t$, of this complementary set of $\nu$-kernels has the
   property that $\Skel(N_j) = \emptyset$, and thus by
   Corollary \ref{cor_empty_kernels_correspond_to_bfb_kernels}, $N_j$ is bounded from below.
    As there are finitely many such $\nu$-kernels there exists small enough $\gamma >_\nu 1$ in $\tT$
for which $|q_j| \wedge \gamma = \gamma$ for $j = 1,...,t$. Thus
$\bigcap_{j=1}^{t}N_j$ is bounded from below and thus
$\bigcap_{j=1}^{t}N_j \supseteq \langle F \rangle$ by Remark
   \ref{rem_bounded_from_below_contain_H_kernel}.

Piecing this together with \eqref{localexpansion} yields $f$ over
all of $F^{(n)}$, so we have
\begin{equation}\label{full_expansion}
\langle f \rangle = \bigcap_{i=1}^{q}K_i \cap \bigcap_{j=1}^{t}N_j.
\end{equation}

So, $\langle f \rangle \cap \langle F \rangle = \bigcap_{i=1}^{q}K_i \cap \bigcap_{j=1}^{t}N_j \cap \langle F \rangle = \bigcap_{i=1}^{q}K_i \cap \langle F \rangle$.\\
\end{constr}

In  this way, we see that intersecting a principal
 $\nu$-kernel $\langle f \rangle$  with $\langle F \rangle$ `chops off' all of the  bounded from below $\nu$-kernels in \eqref{full_expansion} (the $N_j$'s given above). This eliminates ambiguity in the $\nu$-kernel corresponding to $\Skel(f)$.\\
Finally we note that if $\Skel(f) = \emptyset$, then $\langle f \rangle = \bigcap_{j=1}^{t}N_j$ for appropriate $\nu$-kernels $N_j$ and   $\langle f \rangle \cap \langle F \rangle = \langle F \rangle$.\\

\begin{rem}\label{RNL} $ $
\begin{enumerate}\eroman
  \item If $K_1$ and $K_2$ are such that $K_1  K_2 \cap F = \{1\}$ (i.e., $\Skel(K_1) \cap \Skel(K_2) \neq \emptyset$), then the sets of $\scrL$-monomials~$\theta_1$ of $K_1$ and of $K_2$ are not the same
   (although one may contain the other), for otherwise together they would yield a single $\nu$-kernel via Construction~\ref{HO_construction}.

  \item The $\nu$-kernels $K_i$, being finitely generated, are in fact principal,
  so we can write $K_i = \langle k_i \rangle$ for rational functions $k_1,
\dots, k_q$. Let $\langle f \rangle \cap \langle F \rangle =
\bigcap_{i=1}^{q}(K_i \cap \langle F \rangle)  =
\bigcap_{i=1}^{q}\langle |k_i| \wedge |\alpha| \rangle =
\bigwedge_{i=1}^{q} \langle |k_i| \wedge |\alpha| \rangle$ with
$\alpha \in F \setminus \{ 1 \}$.
By~Theorem~\ref{thm_every_gen_is_reducible}, for any generator
$f'$ of $\langle f \rangle \cap \langle F \rangle$ we have that
$|f'| = \bigwedge_{i=1}^{q} |k_i'|$ with $k_i' \sim_{\FF} |k_i|
\wedge |\alpha|$ for every $i = 1,...,q$. In particular,
$\Skel(k_i') = \Skel(|k_i| \wedge |\alpha|) = \Skel(k_i)$. Thus
the $\nu$-kernels $K_i$ are independent of the choice of generator
$f$, being defined by the components $\Skel(k_i)$ of   $\Skel(f)$.
\end{enumerate}
\end{rem}

We now provide two instances of
Construction~\ref{HO_construction}. In what follows, we always
make use of the above notation  for the different types of
$\nu$-kernels involved in the construction. When denoting
$\nu$-kernels, $R$ stands for "region", $N$ for "null" (which are
bounded from below), and $L$ for "linear" (representing
HS-kernels, which are unbounded).

\begin{exmp}\label{exmp1_for_construction} $F = (F, \tT, \nu, \tG).$
Let $f = |\la_1| \wedge \alpha = \frac{\alpha|\la_1|}{\alpha  +
|\la_1|} \in F(\la_1,\la_2),$ where $\alpha
>_\nu 1$ in $\tT$. The order relation $ \alpha \leq_\nu |\la_1|$ translates to
the relation $\alpha +   |\la_1|\nucong |\la_1|$  or equivalently to
$\alpha|\la_1|^{-1} +   1 \nucong 1$. Over the region defined by
this relation we have $f =  \frac{\alpha |\la_1|}{|\la_1|} = \alpha
$. Similarly, its complementary order relation $ \alpha \geq_\nu
|\la_1|$ translates to $\alpha^{-1}|\la_1|  +  1 = 1$ (via $|\la_1|
+ \alpha \nucong \alpha$) over whose region $f = \frac{\alpha
|\la_1|}{\alpha}= |\la_1|$ . So
$$\langle f \rangle = K_1 \cap K_2 = (R_{1,1}  L_{1,1}) \cap (R_{2,1}  N_{2,1})$$
where $R_{1,1} =  \langle \alpha^{-1}|\la_1| +   1 \rangle$,
$L_{1,1} = \langle |\la_1| \rangle$, $R_{2,1} = \langle
\alpha|\la_1|^{-1}  +   1 \rangle$, and $N_{2,1} = \langle \alpha
\rangle$. Geometrically $R_{1,1}$ is a strip containing the axis
$\la_1 = 1$, and $R_{2,1}$ is the complementary region. The
restriction of $f$ to $R_{1,1}$ is $|\la_1|$ while $f$ restricted to
$R_{2,1}$ is $\alpha$. Deleting $N_{2,1}$ we still have $\Skel(f) =
\Skel(R_{1,1}  L_{1,1})$, although $ R_{1,1}  L_{1,1}$ properly
contains $\langle f \rangle$.
$$\langle f \rangle = \langle f \rangle \cap \langle F \rangle = (R_{1,1}
 L_{1,1} \cap R_{2,1}  N_{2,1}) \cap \langle F \rangle =
(R_{1,1}  L_{1,1}) \cap \langle F \rangle.$$
\end{exmp}

\begin{exmp}\label{exmp2_for_construction}
Let $f = |\la_1  +   1| \wedge \alpha \in F(\la_1,\la_2)$ for some
$\alpha
>_\nu 1$ in $\tT$. First note that  $|\la_1 +  1| = \la_1  +   1$ since $\la_1 +  1 \geq 1$, allowing us to rewrite $f$ as $(\la_1 +  1) \wedge \alpha$. Then $
f = \frac{\alpha (\la_1 +   1) }{\alpha  +   (\la_1 +  1)} =
\frac{\alpha \la_1 + \alpha}{\alpha +  \la_1}$. The order relation $
\alpha \leq_\nu \la_1$ translates to the relation $\alpha + \la_1
\nucong \la_1$ or equivalently to $\alpha \la_1^{-1} +   1 \nucong
1$, over whose region    $f = \frac{\alpha \la_1 + \alpha}{\alpha +
\la_1} = \frac{\alpha \la_1 + \alpha}{\la_1} = \alpha +
\frac{\alpha}{\la_1} = \alpha$. Similarly, the complementary order
relation $ \alpha \geq \la_1$ translates to $\alpha^{-1}\la_1 +  1
\nucong 1$  over whose region  $f = \frac{\alpha \la_1 +
\alpha}{\alpha + \la_1} = \frac{\alpha \la_1 +  \alpha}{\alpha} =
\la_1 +   1$. So
$$\langle f \rangle = K_1 \cap K_2 = (R_{1,1}  \mathcal O _{1,2}) \cap (R_{2,1}  N_{2,1}) = \mathcal O _{1,2} \cap (R_{2,1}  N_{2,1})$$
where $R_{1,1} = \langle \alpha^{-1} \la_1  +   1 \rangle$,
$\mathcal O _{1,2} = \langle \la_1  +   1 \rangle$, $R_{2,1} =
\langle \alpha \la_1^{-1}  +   1 \rangle$ and $N_{2,1} = \langle
\alpha \rangle$. But $\Skel(R_{2,1}  N_{2,1}) \subseteq
\Skel(N_{2,1})= \emptyset$. So
$$\Skel(f) = \Skel(\mathcal O _{1,2}) \cup \Skel(R_{2,1}  N_{2,1}) = \Skel(\la_1 +  1) \cup \emptyset = \Skel(\la_1 +  1).$$
\end{exmp}

Now suppose $L_i$ are HP-kernels and $\mathcal O_j$ are order
$\nu$-kernels, and let $L =  \prod L_i$ and  $R =\prod {\mathcal
O}_j$. As can be  seen easily from examples
\ref{exmp1_for_construction} and \ref{exmp2_for_construction}, by
substituting any $\scrL$-monomial for $\la_1$ and any order
fraction for $\la_1 + 1$,
\begin{align*}
 ( L  R)\cap \langle F \rangle  =& \prod (L_i \cap \langle F \rangle)   \prod (\mathcal O_j \cap \langle F \rangle) \\
= & \left( \prod \left((L_i  R_i) \cap (N_i  R_i^{c})\right)  \prod
\left((\mathcal O_j  {R'}_j) \cap (M_j   {{R'}_j}^{
c})\right)\right)\cap \langle F \rangle \\ & =  \left( \prod L_i
\prod {\mathcal O'}_j \right)  \cap N \cap  \langle F \rangle
\end{align*} where the $R_i$ and $R'_j$ are region $\nu$-kernels and $N = \left(\prod_i N_i\right)\left(\prod_j M_j\right)$.
  The $\nu$-kernels $N_i,
M_j,$ and thus $N$, are bounded from below;  $ \langle F \rangle
\subseteq N$, and the right side is  $( L R' ) \cap \langle F
\rangle $ where $R'$ is a region $\nu$-kernel.

 Note that the $\mathcal
O'_j$s involve the $R_i$s, the $R'_j$s and the $\mathcal O_j$s. Also
note that intersecting with $\langle F \rangle$ keeps the HS-kernel
unchanged in the new decomposition.

As  the $N_j$s in \eqref{full_expansion}, being  bounded from
below, do not affect $\Skel(f)$, we put them aside for the time
being and proceed to study the structure of the $\nu$-kernels
$K_i$ and their corresponding $\onenu$-sets.

Take one of these $\nu$-kernels $K_i$. Recall that  $K_i$ is
generated by a set comprised of   $\scrL$-monomials  and order
elements. Let $L_{i,j}$,  $1\le j \le u$, and  $\mathcal O_{i,k}$,
$1 \le k \le v$, be the HP-kernels and the order $\nu$-kernels
generated respectively by these   $\scrL$-monomials  and order
elements. Then we can write
\begin{equation}\label{eqLO}
K_i = L_i  R_i = \left(\prod_{j=1}^{u}L_{i,j}
\right)\left(\prod_{k=1}^{v}\mathcal O_{i,k}\right),
\end{equation}
where $L_i=\prod_{j=1}^{u}L_{i,j}$ is an HS-kernel and $R_i =
\prod_{k=1}^{v}\mathcal O_{i,k}$ is a region $\nu$-kernel. By
assumption, $\Skel(K_i) \neq \emptyset$ since
 at least one point of the $\onenu$-set  was used in its construction.
  Moreover, one cannot write $L = M_1 \cap M_2$ for distinct HS-kernels $M_1$ and
  $M_2$,
  for otherwise the construction
  would have produced two distinct $\nu$-kernels, one with $M_1$ as its HS-kernel
  and the other with $M_2$ as its HS-kernel, rather than $K_i$.

  Let us formalize this situation.

\begin{defn}\label{defn_HO_fraction}
A rational function  $f \in F(\la_1,...,\la_n)$ is  an
\textbf{HO-fraction} if it is the sum of   an HS-fraction $f'$ and a
region fraction $o_f$.
\end{defn}

\begin{defn}
A principal $\nu$-kernel $K \in \PCon(F(\la_1,...,\la_n))$ is said
to be an \textbf{HO-kernel} if it is generated by an HO-fraction.
\end{defn}

 Note that any HS-kernel or any
region $\nu$-kernel is an HO-kernel.

\begin{lem}
 A principal $\nu$-kernel $K$ is an HO-kernel if and only if
   $K = L  R$ where $L$ is an HS-kernel and $R$ is a region $\nu$-kernel.
\end{lem}
\begin{proof} $(\Rightarrow)$
Write $K = \langle f \rangle$, where $f = f' +   o_f$ is an
HO-fraction. Thus $K = \langle f'  +  o_f \rangle = \langle f'
\rangle  \langle o_f \rangle$ where $\langle f' \rangle$ is an
HS-kernel and $\langle o_f \rangle$ is a region $\nu$-kernel.

$(\Leftarrow)$ Write the HO-fraction $f = f'  +   r$ where $f'$ is
an HS-fraction generating $L$ and $r$ is a region fraction
generating~$R$; then $\langle f \rangle = \langle f'  +   r \rangle
= \langle f' \rangle  \langle r \rangle =  L  R = K$.
\end{proof}

\begin{thm}\label{thm_HP_expansion}
Every principal $\nu$-kernel $\langle f \rangle$ of
$F(\la_1,...,\la_n)$ can be written as the intersection of
finitely many principal $\nu$-kernels
$$\{K_i : i=1,...,q\} \ \text{and} \ \{N_j : j = 1,...,m \},$$
whereas each $K_i$ is the product of an HS-kernel and a region
$\nu$-kernel
\begin{equation}\label{consab}
K_i = L_i  R_i =  \prod_{j=1}^{t_i}L_{i,j} \prod_{k=1}^{k_i}\mathcal
O_{i,k}
\end{equation}
while each $N_j$ is a product of bounded from below $\nu$-kernels
and (complementary) region $\nu$-kernels. For $\langle f \rangle
\in \PCon(\langle F \rangle)$, the~$N_j$ can be replaced by
$\langle F \rangle$ without affecting the resulting
$\nu$-kernel.\end{thm}
\begin{proof}

Let $K = L  R$ be an HO-kernel with  $R$ a region $\nu$-kernel and
$L$ an HS-kernel. By Proposition~\ref{rem_HS_prod_HP} and
Lemma~\ref{rem_Region_prod_Order}, we have that $L =
\prod_{i=1}^{u}L_i$ for some HP-kernels $L_1,...,L_u$ and $R =
\prod_{j=1}^{v}\mathcal O_j$ for some order $\nu$-kernels
$\mathcal O_1,...,\mathcal O_v$. Thus
\begin{equation}
K = L  R = \left(\prod_{i=1}^{u}L_i \right)\left(
\prod_{j=1}^{v}\mathcal O_j\right),
\end{equation}
where $u \in \mathbb{Z}_{\geq 0}$, \ $v \in \mathbb{N}$,
$L_1,...,L_u$ are HP-kernels and $\mathcal O_1,...,\mathcal O_v$
are order $\nu$-kernels.

Let $K_1$ and $K_2$ be region $\nu$-kernels (respectively
HS-kernels) such that  $(K_1K_2)~\cap~F~=~\{ 1 \}$. Then $K_1 K_2$
is a region $\nu$-kernel (respectively HS-kernel). Consequently,
if $K_1$ and $K_2$ are HO-kernels such that $(K_1  K_2) \cap F =
\{ 1 \}$ , then $K_1 K_2$ is an HO-kernel. Indeed, the assertions
follow from the decomposition   $K_{i} = L_{i} \mathcal O_{i} =
(\prod_{j=1}^{u_{i}}L_{i,j} )( \prod_{k=1}^{v_{i}}\mathcal
O_{i,k})$ for $i =1,2$, so that
$$K_1  K_2 =
(L_1 L_2)  (\mathcal O_1 \mathcal O_2) =
\left(\prod_{i=1}^{u_1}L_{1,i} \prod_{i=1}^{u_2}L_{2,i}\right)
\left(\prod_{j=1}^{v_1}\mathcal O_{1,j} \prod_{j=1}^{v_2}\mathcal
O_{2,j}\right) = L  \mathcal O,$$ with the appropriate $u_{i},v_i$
taken for $i =1,2$.
\end{proof}
\begin{prop}$ $
\begin{itemize}\eroman
  \item If $\langle f \rangle$ is an HS-kernel, then the decomposition degenerates to $\langle f \rangle = K_1$ with $K_1 = L_1 = \langle f \rangle$.
  \item If $\langle f \rangle$ is a region $\nu$-kernel, then $\langle f \rangle = K_1$ with $K_1 = R_1= \langle f \rangle$.
  \item $\langle f \rangle$ is an irregular $\nu$-kernel if and only if there exists some $i_0 \in \{ 1,...,q \}$ such that $K_{i_0} =  R_{i_0} = \prod_{k=1}^{k}\mathcal O_{k,i_0}$.
  \item $\langle f \rangle$ is a regular $\nu$-kernel if and only if $K_i$ is comprised of at least one HP-kernel, for every $i = 1,...,q$.
\end{itemize}
\end{prop}

\begin{proof} (i)
By Proposition \ref{rem_HS_prod_HP}, $\langle f \rangle =
\prod_{j=1}^{t}L_j$ where $L_j$ is an   HP-kernel for each $j$.
Write $f = \frac{\sum h_{i}}{\sum g_{j}}$ for monomials $h_i,g_j \in
F[\Lambda]$. If $\langle f \rangle$ is irregular, then  at some
neighborhood of a point $\bfa \in F^{(n)}$  we have some $i_0$ and
$j_0$ for which $h_{i_0}\nucong g_{j_0}$ where
$(g_{j_0}(\bfa)\nucong ) h_{i_0}(\bfa)
>_\nu h_i(\bfa) , g_j(\bfa)$ for every $i \neq i_0$ and $j \neq j_0$. The
$\nu$-kernel corresponding to (the closure) of this region has its
relations \eqref{eq_relations1} degenerating to $1 = 1$ as
$\frac{h_{i_0}}{g_{j_0}} \nucong 1$ over the region, thus is given
only by its order relations of \eqref{eq_relations2}.

The last three assertions are direct consequences of \eqref{consab};
 namely, if $\langle f \rangle$ is either an HS-kernel or a region $\nu$-kernel, $\langle f \rangle$   already takes on the form of its decomposition. The fourth is equivalent to the third.\\

\end{proof}

\subsection{Conclusion of the proof of
Theorem~\ref{thmcireglattice}}$ $

Having the palate of $\nu$-kernels at our disposal, we can now
better understand Theorem~\ref{thmcireglattice}.

\begin{defn}
The \textbf{HO-decomposition} of  a principal $\nu$-kernel
$\langle f \rangle $ is   its decomposition given in Theorem
\ref{thm_HP_expansion}.
In the special case where  $\langle f \rangle \in \PCon(\langle F
\rangle)$, all bounded from below terms of the intersection are
equal to $\langle F \rangle$.
\end{defn}


\begin{defn}
For a subset  $S \subseteq F(\la_1,...,\la_n)$, denote by
$\text{HO}(S)$ the family of   HO-fractions in $S$, by
$\text{HS}(S)$ the family of HS-fractions in $S$, and by
$\text{HP}(S)$ the family of $\scrL$-monomials in $S$.
\end{defn}

\begin{rem} $\text{HP}(S) \subset \text{HS}(S)
\subset \text{HO}(S)$ for any $S \subseteq F(\la_1,...,\la_n)$,
since every $\scrL$-monomial is an HS-fraction and every HS-fraction
is an HO-fraction.
\end{rem}

\begin{exmp}\label{exmp_HO_decomposition}
Consider the $\nu$-kernel $\langle f \rangle$ where $f =
\frac{\la_1}{\la _2 +1} \in F(\la_1,...,\la_n)$. The points on the
$\onenu$-set  of $f$ define three distinct HS-kernels: $\langle
\frac{\la_1}{\la_2} \rangle $ (corresponding to   $\la_1=\la_2$)
over the region $\{\la_2 \geq 1\}$  defined by the region
$\nu$-kernel $\langle 1 + \la_2^{-1} \rangle$, $\langle \la_1
\rangle = \langle \frac{\la_1}{1} \rangle$ (corresponding to
$\la_1=1$) over the region $\{\la_2 \leq 1\}$  defined by the
region $\nu$-kernel $\langle 1 + \la_2 \rangle$, and $\langle
|\la_1| + |\la_2| \rangle$ (corresponding to the point defined by
$\la_1=\la_2=1$). Thus by Construction \ref{HO_construction},
$$\langle f \rangle = \left(\left\langle \frac{\la_1}{\la_2} \right\rangle
 \langle 1 + \la_2^{-1} \rangle\right) \cap \langle \la_1
\rangle  \langle 1 + \la_2 \rangle \cap \langle |\la_1| + |\la_2|
\rangle  \langle 1 \rangle$$ and
$$\Skel(f) = \left(\Skel\left(\frac{\la_1}{\la_2}\right) \cap \Skel(1 +
\la_2^{-1})\right) \cup \left(\Skel(\la_1) \cap \Skel(1 +
\la_2)\right) \cup \left(\Skel(|\la_1| + |\la_2|) \cap F^{(2)}
\right).$$ The third component of the decomposition (i.e., the
HS-kernel $\langle |\la_1| + |\la_2| \rangle$) could be omitted
without effecting $\Skel(f)$.

 The decomposition is shown (in
logarithmic scale) in Figure \ref{fig:Ex6}, where the first two
components are the rays emanating from the origin and the third
component is the origin itself.

\end{exmp}

\begin{figure}
\centering

\begin{tikzpicture}[scale=3,
    axis/.style={thin, color=gray, ->},
    important line/.style={very thick, ->}]

    \draw[axis] (-1,0)  -- (1,0) node[right, color=black] {$\lambda_1$} ;
    \draw[axis] (0,-1) -- (0,1) node[above, color=black] {$\la_2$} ;
    \draw[important line] (0,0)  -- (0,-1) node[above= 1.5cm,left=0.4cm] {$\Skel(\langle \lambda_1 \rangle \cap \langle 1 + \la_2 \rangle)$};
    \draw [black, fill=black] (0,0) circle (1pt) node[below=0.5cm,right=0.3cm] {$\Skel(\langle |\lambda_1| + |\la_2| \rangle)$};
    \draw[important line] (0,0) -- (0.7,0.7) node[below=0.2cm,right=0.2cm] {$\Skel(\left\langle \frac{\lambda_1}{\la_2} \right\rangle  \cap \langle 1 + \la_2^{-1} \rangle)$};

 \end{tikzpicture}

\caption{\small{$\Skel(\frac{\la_1}{\la_2+1}) = \Skel(\left\langle
\frac{\la_1}{\la_2} \right\rangle  \cap \langle 1 + \la_2^{-1}
\rangle) \cup \Skel(\langle \la_1 \rangle \cap \langle 1 + \la_2
\rangle) \cup \Skel(|\la_1| + |\la_2|)$}} \label{fig:Ex6}

\end{figure}

\subsubsection{The lattice generated by regular corner-internal
principal $\nu$-kernels}$ $

Recall from
Proposition~\ref{rem_regularity_of_HP_times_order_kernels_generalization}
  that the principal $\nu$-kernel $\langle f \rangle
\langle o_1 \rangle \cdots  \langle o_k \rangle$ is regular, for
any HP-kernel $\langle f \rangle \neq 1$ and order $\nu$-kernels
$\langle o_1 \rangle, ... , \langle o_k \rangle$.

\begin{lem}\label{cor_regularity_of_intersection_terms_of_HS O_expansion}
Let $K \in \PCon(\langle F \rangle)$ and let $$K = (L  R) \cap
\langle F \rangle = \left(\prod_{j=1}^{u}L_j\right)
\left(\prod_{k=1}^{v}\mathcal O_k\right) \ \cap \langle F
\rangle$$ be the decomposition of $K$, as given in \eqref{eqLO},
where $L_1,....,L_u$ are HP-kernels and $\mathcal O_1,...,\mathcal
O_v$ are order $\nu$-kernels. If $u \neq 0$, i.e., $L \neq \langle
1 \rangle$, then $K$ is regular.
\end{lem}
\begin{proof}
Indeed, $K = (\prod_{j=2}^{u}L_j)  (L_1  \prod_{k=1}^{v}\mathcal
O_k) \cap \langle F \rangle$. By
Proposition~\ref{rem_regularity_of_HP_times_order_kernels_generalization},
$(L_1  \prod_{j=1}^{v}\mathcal O_j)$ is regular since $L_1$ is
 an HP-kernel. Thus $K$ is  a regular $\nu$-kernel,
 since a product of regular $\nu$-kernels is
regular and since intersection with $\langle F \rangle$ does not
affect regularity.
\end{proof}

We are finally ready to prove Theorem~\ref{thmcireglattice}.
\begin{proof}
Let $\langle f \rangle$ be a principal $\nu$-kernel and let
$$\langle f \rangle = \left(\bigcap_{i=1}^{q}K_i\right)  \cap \langle F \rangle, \ \ K_i = \prod_{j=1}^{t}L_{i,j} \prod_{k=1}^{k}\mathcal O_{i,k}$$
be its HO-decomposition. By Lemma
\ref{lem_HP_and_Oreder_kernels_are_CI}, each HP-kernel $L_{i,j}$
and each order $\nu$-kernel $\mathcal O_{i,k}$ are corner
internal. Thus $\langle f \rangle$ as a finite product of
principal corner internal $\nu$-kernels is in the lattice
generated by principal corner-internal $\nu$-kernels.
%
%

For the second assertion,  if $\langle f \rangle$ is regular, then
by Theorem \ref{thm_HP_expansion}, for every $1 \leq i \leq q$, we
have that $L_{i,1} \neq 1$. Thus by
Lemma~\ref{cor_regularity_of_intersection_terms_of_HS O_expansion},
  each $K_i$ is a product of principal regular
corner-internal $\nu$-kernels. Thus $\langle f \rangle$ is in the
lattice generated by principal regular corner-internal
$\nu$-kernels. We conclude with
Proposition~\ref{cor_regular_lattice}.
\end{proof}

%

\section{Polars: an intrinsic
description of $\mathcal{K}$-kernels}\label{Polar}

 To characterize
 $\mathcal{K}$-kernels intrinsically, we need a kind of orthogonality relationship,
 and introduce a new kind of $\nu$-kernel, called
\textbf{polar}, borrowed from the theory of lattice ordered groups
(\cite[section (2.2)]{OrderedGroups3}). We fix $X \subseteq
F^{(n)} .$  $\mathcal{S}$ is always assumed to be an idempotent
1-\vsemifield0, often $F(X)$ or even  $\FunF.$ In this section, we
lay out the general basics of the theory, although its full
strength is only obtained in the following sections when the
underlying \vsemifield0 $F$ is taken to be $\nu$-divisible,
$\nu$-archimedean, and complete, e.g., $F = \mathscr{R}$.

\subsection{Basic properties of polars}$ $

\begin{defn}\label{def_polar}
We write $f\perp g $ for $f,g \in F(X)$ if   $|f(\bfa)|\wedge
|g(\bfa)| \nucong 1$ for all $\bfa \in X$, i.e., $\Skel(f) \cup
\Skel(g) = X.$

 For  subsets $K,L$ of $F(X)$ we write  $K\perp L $ if $f\perp g $ for all $f \in
K$ and $g \in L$. (This is not necessarily implied by $ (\wedge |
K|) \perp (\wedge  |L|) ,$ as evidenced by Note~\ref{compl}.) For
$V\subseteq \mathcal{S}$, we define
\begin{equation}
V^{\bot} = \{ g  \in \mathcal{S} : g\perp V.  \}
\end{equation}
Such a set $V^{\bot}$ is called a \textbf{polar} in the literature.

For $f \in \mathcal{S}$  we write $f^{\bot}$ for $\{ f \}^{\bot}$.
Thus, $ f^{\bot}= |f|^{\bot}$. The set of all polars in
$\mathcal{S}$ is denoted as $\Plr(\mathcal{S})$.
\end{defn}

\begin{rem}  $f\perp g $ iff  $|f|\perp g $, iff  $\langle f \rangle \perp \langle g
\rangle.$\end{rem}

\begin{rem}[\cite{OrderedGroups3}]\label{rem_polar_of_kernel}
If $K \subset \mathcal{S}$, then   $K^{\bot} = \langle K
\rangle^{\bot}.$ Consequently,
            $$\langle K \rangle^{\bot \bot} = K^{\bot \bot}.$$
Thus, $L \perp K$ iff $K \subseteq L^{\bot},$  for  any kernels
$K,L$. \end{rem}

Although this kind of property cannot arise in classical algebraic
geometry since a variety cannot be the proper union of two algebraic
sets, it is quite common in the tropical setting. For example $(1+f)
\perp (1+f^{-1}).$ The usual properties of orthogonality go through
here.

\begin{lem}\label{rem_double polar of a polar}
The following statements are immediate consequences of Definition
\ref{def_polar}. For any  $K,L \subset \mathcal{S}$,
\begin{enumerate}\eroman
  \item $K \subseteq L \Rightarrow K^{\bot} \supseteq L^{\bot}$.
  \item $K \subseteq K^{\bot \bot}$.
  \item $K^{\bot} = K^{\bot \bot \bot}$.
   \item $K$ is a polar iff  $K^{\bot \bot} = K$.
\end{enumerate}
\end{lem}
\begin{proof} The first two assertions are obvious,
and the third follows from using  $K^{\bot}$ in (ii), applying (i).
  Finally, if $K$ is a
polar, then $K = V^{\bot}$ for some $V \subseteq \mathcal{S}$, and
$K^{\bot \bot} = (V^{\bot})^{\bot \bot} =V^{\bot} = K$.
\end{proof}

The following facts, taken from \cite{OrderedGroups3}, can be
checked pointwise.

\begin{thm}\label{thm_properties of polars}
  For any subset $V$ of  $\mathcal{S}$,  $V^{\bot}$ is a $\mathcal{K}$-kernel of \ $\mathcal{S}$.
 \end{thm}
\begin{proof}
 As noted in \cite[Theorem 2.2.4(e)]{OrderedGroups3}, $K := V^{\bot}$ is a convex
(abelian) subgroup, and thus a $\nu$-kernel. It remains to show
that $K = \Ker(\Skel(K)).$ Clearly $\subseteq$ holds, so we need
to show that any $g \in \Ker(\Skel(K))$ belongs to $K$. We are
given $\Skel(g) \supseteq \Skel(K)$ and $K \perp v$ for each $v$
in $V$, so $\Skel(g) \supseteq \Skel(K)$ implies that
$\Skel(g)\cup \Skel(v) \supseteq \Skel(K) \cup \Skel(v) = X $,
i.e., $g \perp V,$ as desired.
\end{proof}

\begin{prop}\label{thm_properties of polars1}
 $(\Plr(\mathcal{S}), \cdot, \cap, \bot, \{1\},\mathcal{S})$ is a complete Boolean algebra.
 \end{prop}
\begin{proof}
 \cite[Theorem 2.2.5]{OrderedGroups3}; If  $\{K_i ; i \in I\}$
are subsets of $\mathcal{S}$,  then  $$\left(\bigcup_{i}
K_i\right)^{\bot} = \bigcap_{i\in I}K_i^{\bot}.$$
 \end{proof}

  Closure under complements is a consequence of (ii).

%
%
%

\begin{prop}\label{prop_minimal_polar_containing_a_set}
For any subset $V \subseteq \mathcal{S}$, $V^{\bot \bot}$ is the
minimal polar containing~$V$.
\end{prop}
\begin{proof}
By  definition,  $V^{\bot \bot}$ is a polar containing $S $. Let $P
 \supseteq S$ be a polar. Then  $S^{\bot \bot} = (S^{\bot})^{\bot} \subseteq
(P^{\bot})^{\bot}=P$.
\end{proof}

\begin{defn}
Let $S \subseteq \mathcal{S}$. We say that a polar $P$ is
\textbf{generated} by $S$ if $P~=~S^{\bot \bot}$. If $S = \{f \}$
then we also write $f^{\bot \bot}$ for the polar generated by
$\{f\}$.
\end{defn}

\begin{defn}
A polar $P$ of $\mathcal{S}$ is \textbf{principal} if there exists
some $f \in \mathcal{S}$ such that $P  = f^{\bot \bot} ( = \langle f
\rangle^{\bot \bot})$.
\end{defn}

\begin{lem}\label{rem_lattice_of_principal_polars}
The collection of principal polars  is a sublattice of
$(\Plr(\mathcal{S}), \cdot, \cap)$.
\end{lem}
\begin{proof}
For any  $f,g \in \mathcal{S}$,  $$(\langle f \rangle \cap \langle g
\rangle)^{\bot \bot}  = (\langle |f| \wedge |g| \rangle)^{\bot \bot}
=  (|f| \wedge |g|)^{\bot \bot}  ,$$ and $$ (\langle f \rangle
\langle g \rangle)^{\bot \bot} = (\langle |f| \dotplusss |g|
\rangle)^{\bot \bot} =  (|f| \dotplusss |g|)^{\bot \bot},$$ by
Corollary~\ref{eq_max_principal_operations}.
\end{proof}

%
%

\begin{rem}
Any $\nu$-kernel  $K$ of an idempotent \vsemifield0\ $\mathcal{S}$
cannot be orthogonal to a nontrivial $\nu$-subkernel of $K^{\bot
\bot}$.
\end{rem}
\begin{proof}
Any $\nu$-kernel $L$ of $K^{\bot \bot}$  is a $\nu$-kernel of
$\mathcal{S}$. If $K \perp L $, then $L \subseteq K^{\bot}\cap
K^{\bot \bot} = \{1\}$, yielding $L = \{ 1 \}$.
\end{proof}

\subsection{Polars over complete Archimedean \semifields0}\label{polar}$ $

In Theorem~\ref{prop_correspondence} we showed that
$\mathcal{K}$-kernels of $\overline{F(\Lambda)}$ correspond to
$\onenu$-sets in $F^{(n)} $. Our aim here is to characterize these
special kind of $\nu$-kernels as polars, as a converse to
Theorem~\ref{thm_properties of polars}. In analogy to algebraic
geometry, polars play the role of `radical ideals,' corresponding
to root sets by an analog to the celebrated Nullstellensatz
theorem, but for this we need to make extra assumptions on   the
underlying \vsemifield0.

\subsection{The polar-$\onenu$-set correspondence}$ $

Recall that $\mathscr{R}$ is presumed to be a   $\nu$-bipotent,
divisible, archimedean and complete \vsemifield0.
In this subsection we concentrate on $\overline{\mathscr{R}(\Lambda)}$. Doing so, we consider the natural extensions to $\overline{\mathscr{R}(\Lambda)}$ of the operators $\Skel$ and $\Ker$  defined in \S\ref{Subsection:Kernels_of_1-Sets} with respect to $\mathscr{R}(\Lambda)$. We write $\Ker_{\mathscr{R}(\Lambda)}$ to denote the  restriction of $\Ker$ to $\mathscr{R}(\Lambda)$.\\

%
%
%
Recall ``completely closed'' from
Definition~\ref{defn_cond_complete}.

\begin{lem} All $\nu$-kernels
of $\overline{\mathscr{R}(\Lambda)}$ that are completely closed
are completions of $\nu$-kernels of
$\mathscr{R}(\Lambda)$.\end{lem}
\begin{proof} By Theorem \ref{thm_nother_1_and_3}(1), the $\nu$-kernels of
$\mathscr{R}(\Lambda)$ are precisely $\{ K \cap
\mathscr{R}(\Lambda) : K$ is a $\nu$-kernel of
$\overline{\mathscr{R}(\Lambda)}\}$. Since $\mathscr{R}(\Lambda)$
is dense in $\overline{\mathscr{R}(\Lambda)}$ for every
$\nu$-kernel $K$ of $\overline{\mathscr{R}(\Lambda)}$, the
$\nu$-kernel $L = K \cap \mathscr{R}(\Lambda)$ of
$\mathscr{R}(\Lambda)$ is dense in $K$. Since $L \subseteq K$, we
conclude that $\bar{L} = \bar{K}$, i.e., $\bar{L} = K$ if and only
if $K$ is completely closed.
\end{proof}

\begin{prop}\label{cor_correspondence_of_kernels_with_the_completions}
Each completely closed $\nu$-kernel $K$ of
$\overline{\mathscr{R}(\Lambda)}$ defines a unique $\nu$-kernel of
$\mathscr{R}(\Lambda)$ given by $L = K \cap \mathscr{R}(\Lambda)$
for which $\overline{L} = K$.
\end{prop}

\begin{exmp}\label{exmp_completion_of_bounded_kernel}
Consider the $\nu$-kernel $K = \langle |\lambda_1| \wedge |\alpha|
\rangle \in \PCon(\mathscr{R}(\lambda_1))$ and its subset $X =
\{|\lambda_1| \wedge |\alpha|^n : n \in \mathbb{N}  \}$. Then
$|\lambda_1| = \bigvee_{f \in X}f \in \overline{K}$. Hence
$\langle \lambda_1 \rangle \subset \overline{K}$ yielding
$\overline{\langle |\lambda_1| \wedge |\alpha| \rangle} =
\overline{\langle |\lambda_1| \rangle}$.
\end{exmp}

\begin{rem}
$\Skel(\bar{K}) = \Skel(K)$ for every $\nu$-kernel $K $ of
$\mathscr{R}(\Lambda)$.
\end{rem}
\begin{proof}
First note that  $\Skel(\bar{K}) \subseteq \Skel(K)$ since $K
\subset \bar{K}$ in $\overline{\mathscr{R}(\Lambda)}$. Now, let $X =
\Skel(K)$. For any nonempty subset $A$ of $K$. If $\bigvee_{f \in
A}f \in \overline{\mathscr{R}(\Lambda)}$ then for any $a \in X$,
$$\left(\bigvee_{f \in A}f\right) (\bfa) = \bigvee_{f \in A}f(\bfa) = \bigvee_{f \in
A}1 \nucong 1,$$ yielding  $\Skel(\bigvee_{f \in A}) \supseteq X$
(cf.~Lemma~\ref{rem_complete_semifield_of_functions}). Similarly, if
$\bigwedge_{f \in A}f \in \overline{\mathscr{R}(\Lambda)}$ then for
any $\bfa \in X$ we have   $(\bigwedge_{f \in A}f) (\bfa) =
\bigwedge_{f \in A}f(\bfa) = \bigwedge_{f \in A}1 \nucong 1$,
yielding that $\Skel(\bigwedge_{f \in A}f) \supseteq X$. We conclude
that $\Skel(\bar{K}) = \Skel(K)$.
\end{proof}



\begin{prop}\label{prop_mathcal_K_kernel_is_completely_closed}
Every $\mathcal{K}$-kernel of $\overline{\mathscr{R}(\Lambda)}$ is
completely closed.
\end{prop}
\begin{proof}
By Lemma~\ref{prop_skel_ker_relations}, if  $Z =\Skel(K)$ then
$\Skel(\Ker(Z)) = Z$. Let $Z = \Skel(K) \subseteq
\mathscr{R}^{(n)} $ where  $K = \Ker(Z)$, a $\nu$-kernel of
$\overline{\mathscr{R}(\Lambda)}$.     $\Skel(f) \supseteq Z$ for
any $f \in K$. Now, if $\bigvee_{f \in A}f \in
\overline{\mathscr{R}(\Lambda)}$ and $\bigwedge_{f \in A}f \in
\overline{\mathscr{R}(\Lambda)}$, then by
Lemma~\ref{rem_complete_semifield_of_functions}, for any $a \in Z$
$$\left(\bigvee_{f \in A}f\right)(a) = \bigvee_{f \in A}f(a) = \bigvee_{f \in A}1 \nucong 1 \ \ \text{and} \ \
(\bigwedge_{f
 \in A}f) (a) = \bigwedge_{f \in A}f(a) = \bigwedge_{f
\in A}1 \nucong 1.$$ Thus
$$\Skel\left(\bigvee_{f \in A}f\right), \Skel\left(\bigwedge_{f \in A}f\right) \supseteq Z.$$
So $\bigvee_{f \in A}f, \bigwedge_{f \in A}f \in K$, and $K$ is
completely closed in $\overline{\mathscr{R}(\Lambda)}$.
\end{proof}

\begin{rem}
Proposition \ref{prop_mathcal_K_kernel_is_completely_closed} is not
true when taking $\mathcal{K}$-kernels of $\mathscr{R}(\Lambda)$
instead of $\overline{\mathscr{R}(\Lambda)}$. Let $\alpha \in
\mathscr{R}$ such that $\alpha > 1$. Consider the subset
$$X = \{  |\lambda_1^n| \wedge \alpha : n \in \mathbb{N}\}.$$
Then $X \subset \langle \lambda_1 \rangle$, \ $\bigvee_{f \in X}f =
\alpha$ (the constant function) and $\Skel(f) = \{ 1 \} \subset
\mathscr{R}$ for every $f \in X$. Thus $\alpha = (\alpha)(1) =
(\bigvee_{f \in X}f) (1) \neq \bigvee_{f \in X}f(1) = 1$ and
$\alpha$ is not in the preimage of $\Skel(\lambda_1)$. So we deduce
that $\mathcal{K}$-kernels of $\mathscr{R}(\Lambda)$ need not  be
completely closed in $\mathscr{R}(\Lambda)$, and  thus not polars
since every polar is completely closed by
Proposition~\ref{prop_mathcal_K_kernel_is_completely_closed}. Also
note that $\alpha \in \lambda_1^{\bot \bot}$,  yielding
$\lambda_1^{\bot} = \lambda_1^{\bot \bot \bot} = (\lambda_1^{\bot
\bot})^{\bot} =\{1\}$.
\end{rem}

%


\begin{thm}\label{prop_polar_k_kernel_completion_equivalence}
 The  following properties of  a $\nu$-kernel  $K$ of $\overline{\mathscr{R}(\Lambda)}$ are equivalent:
 \begin{enumerate}\eroman
   \item $K$ is a $\mathcal{K}$-kernel.
   \item $K$ is a polar.
   \item $K$ is completely closed.
 \end{enumerate}
\end{thm}
\begin{proof}$ $
$(i) \Rightarrow (iii)$ By Proposition
\ref{prop_mathcal_K_kernel_is_completely_closed}.
$\mathcal{K}$-kernel of $\overline{\mathscr{R}(\Lambda)}$ is
completely closed.

$(iii) \Rightarrow (ii)$ See \cite[Theorem 2.3.7]{OrderedGroups3}.
One checks the condition of Lemma~\ref{rem_double polar of a polar}.

$(ii) \Rightarrow (i)$ By Theorem \ref{thm_properties of polars}.
\end{proof}

\begin{prop}\label{prop_principal}
For any $f \in \overline{\mathscr{R}(\Lambda)}$,
$$\Skel(f) = \Skel(f^{\bot \bot}) \  \text{and} \ f^{\bot \bot} = \Ker(\Skel(f)).$$
\end{prop}
\begin{proof}
   $\Skel(f^{\bot \bot}) \subseteq \Skel(f)$ since $f^{\bot \bot} \supseteq \langle f \rangle$.
   Now, let
$K$ be the $\mathcal{K}$-kernel containing $f$ such that $\Skel(K) =
\Skel(f)$.  $K$ is a polar by  Lemma~\ref{rem_double polar of a
polar}(iv). By Proposition
\ref{prop_minimal_polar_containing_a_set}, $K \supseteq f^{\bot
\bot}$, and so $$\Skel(f) = \Skel(K) \subseteq \Skel(f^{\bot \bot})
\subseteq \Skel(f).$$  But $f^{\bot \bot}$ is a polar and thus a
$\mathcal{K}$-kernel, implying $f^{\bot \bot}= \Ker(\Skel(f^{\bot
\bot})) = \Ker(\Skel(f))$.
\end{proof}

\begin{thm}\label{thm_polar}
There is a $1:1$ correspondence
\begin{equation}\label{eq_polar1}
\Plr(\overline{\mathscr{R}(\Lambda)}) \leftrightarrow
\Skl(\mathscr{R}^{(n)} )
\end{equation}
between the polars of $\overline{\mathscr{R}(\Lambda)}$ and the
$\onenu$-sets in $\mathscr{R}^{n}$, given by $B \mapsto \Skel(B)$
and $Z \mapsto \Ker(Z)$.

This correspondence restricts to a correspondence
between the principal polars of $\overline{\mathscr{R}(\Lambda)}$
and the principal $\onenu$-sets in~$\mathscr{R}^{n}$.
\end{thm}
\begin{proof}
By  Lemma~\ref{rem_double polar of a polar}(iv), $B$ is a polar iff
$B$ is  a $\mathcal{K}$-kernel; thus, $\Ker(\Skel(B)) = B$. If $B =
f^{\bot \bot}$ for some $f \in \overline{\mathscr{R}(\Lambda)}$,
then  $\Skel(B) = \Skel(f)$ by Proposition \ref{prop_principal},
yielding $\Ker(\Skel(f^{\bot \bot})) = \Ker(\Skel(f)) = f^{\bot
\bot}$. The restriction to principal polars follows from
Lemma~\ref{rem_lattice_of_principal_polars}.
\end{proof}

\begin{cor}\label{cor_polar}
Let $$\mathcal{B} = \{ B \cap \mathscr{R}(\Lambda) : B \in
\Plr(\overline{\mathscr{R}(\Lambda)}) \}.$$  There is a 1:1
correspondence
\begin{equation}\label{eq_cor_polar_-set_corrsepondence}
 \mathcal{B} \leftrightarrow \Skl(\mathscr{R}^{(n)} )
\end{equation}
given by $K \mapsto \Skel(K)$ and $Z \mapsto \Ker(Z) \cap
\mathscr{R}(\Lambda) = \Ker_{\mathscr{R}(\Lambda)}(Z)$, which
restricts to a correspondence
\begin{equation}\label{eq_polar_-set_corrsepondence}
\text{Principal $\nu$-kernels of }\mathcal{B} \leftrightarrow
\hSkl.\end{equation} Furthermore,
$\Ker_{\mathscr{R}(\Lambda)}(\Skel(K)) = \overline{K} \cap
\mathscr{R}(\Lambda)$ for any $\nu$-kernel $K$ of
$\mathscr{R}(\Lambda)$.

\end{cor}
\begin{proof}
 Since $(\Plr(\mathscr{R}(\Lambda)), \cdot, \cap)$ is a lattice,
the first assertion follows by applying
Proposition~\ref{cor_correspondence_of_kernels_with_the_completions}
to the correspondence in Theorem \ref{eq_polar_-set_corrsepondence}.
The second assertion is by Theorem~\ref{thm_polar} and
\eqref{eq_cor_polar_-set_corrsepondence}, and the last by
Proposition \ref{prop_polar_k_kernel_completion_equivalence} and~
Remark~\ref{rem_carachterization_of_completions}.\end{proof}
%
%

Here is an analog to Proposition~\ref{prop_coord_semifield_1}. For
$X \subset \mathscr{R}^{n}$ we define the \textbf{restriction map}
$\phi_{X}: \mathscr{R}(\Lambda) \to \mathscr{R}(X)$ by $f \mapsto
f|_X.$

\begin{prop}\label{prop_coord_semifield_2}
For any  $\onenu$-set $X = \Skel(\mathcal{S}) \subseteq
\mathscr{R}^{n}$ with $\mathcal{S} \subseteq \mathscr{R}(\Lambda)$,
 $\phi_{X}$ is a homomorphism and
\begin{equation}
\mathscr{R}(\Lambda)  / K_{\mathcal{S}} \cong \mathscr{R}[X].
\end{equation}
where $K_{\mathcal{S}} = \mathcal{S}^{\bot \bot} \cap
\mathscr{R}(\Lambda)$ with $\mathcal{S}^{\bot \bot}$ is taken in the
completion  $\overline{\mathscr{R}(\Lambda)}$ of
$\mathscr{R}(\Lambda)$ in $Fun(\mathscr{R}^{(n)} , \mathscr{R})$.
\end{prop}
\begin{proof}
  Corollary \ref{cor_polar} implies that  $$\Ker(\phi_X) = \{ g \in
\mathscr{R}(\Lambda) : g|_X = 1 \} = \{  g \in \mathscr{R}(\Lambda)
: g \in \mathcal{S}^{\bot \bot} \} = K_\mathcal{S}.$$   We conclude
with the isomorphism theorems \ref{thm_nother_1_and_3}.
\end{proof}

%


\subsubsection{The restriction of the coordinate \vsemifield0}$ $

\begin{cor}\label{prop_coord_semifield_10}
For any (principal) $\onenu$-set $X = \Skel(\langle f \rangle)
\subseteq \mathscr{R}^{n}$, the restriction map  of $\phi_{X}$ to
  $\mathscr{R} $ satisfies
\begin{equation}
\langle \mathscr{R} \rangle /(\langle f \rangle \cap \langle
\mathscr{R} \rangle) \cong \mathscr{R}[X].
\end{equation}
\end{cor}
\begin{proof}  By Proposition
\ref{prop_kernel_is_k_kernel},  $\Ker(\phi_X) = \{g \in \langle
\mathscr{R} \rangle : g|_X = 1 \} = \{  g \in \langle \mathscr{R}
\rangle : g \in \langle f \rangle \} = \langle f \rangle \cap
\langle \mathscr{R} \rangle$, so we conclude with the isomorphism
theorems \ref{thm_nother_1_and_3}.
\end{proof}

\section{Dimension}\label{section:Convexity degree and Hyperdimension}

Throughout, $F$ denotes a divisible, $\nu$-archimedean \vsemifield0.
   Having the HS $\nu$-kernels at our disposal, our next
major goal is to use them to define dimension, which will be seen to
be uniquely defined. We return to the notation of
Construction~\ref{HO_construction}, as indicated after
Remark~\ref{RNL}.

Let $\langle f \rangle \subseteq \langle F \rangle$ be a principal
$\nu$-kernel and let $\langle f \rangle = \bigcap_{i=1}^{s} K_i$,
where
$$K_i = (L_i \cdot R_i) \cap \langle F \rangle = (L_i \cap \langle F \rangle) \cdot (R_i \cap \langle F \rangle) = L_i' \cdot R_i'$$
is its (full) HO-decomposition; i.e., for each $1 \leq i \leq s$,
\ $R_i \in \PCon(F)$ is a region $\nu$-kernel and $L_i\in
\PCon(F)$ is either an HS-kernel or bounded from below (in which
case  $L_i' = \langle F \rangle$). Then by
Corollary~\ref{cor_subdirect_decomposition_for_semifield_of_fractions},
we have the subdirect decomposition
$$\langle F \rangle/ \langle f \rangle  \hookrightarrow \prod_{i=1}^{t} \langle F \rangle /K_i = \prod_{i=1}^{t} (\langle F \rangle /L_i' \cdot R_i')$$
where $t \leq s$ is the number of $\nu$-kernels $K_i$ for which
$L_i' \neq \langle F \rangle$ (for otherwise $\langle F \rangle
/K_i~=~\{1 \}$ and can be omitted from the subdirect product).

\begin{exmp}\label{exmp_infinite_chain}
Consider the principal $\nu$-kernel $\langle \la_1 \rangle \in
\PCon(F(\la_1,\la_2))$. For $\alpha \in F$ such that $\alpha > 1$,
we have the following infinite strictly descending chain of
principal $\nu$-kernels
$$\langle \la_1 \rangle \supset \langle |\la_1| + |\la_2+1| \rangle \supset \langle |\la_1| + |\alpha^{-1}\la_2 + 1| \rangle \supset \langle |\la_1| + |\alpha^{-2}\la_2 + 1| \rangle \supset \dots$$ $$\supset \langle |\la_1| + |\alpha^{-k}\la_2 + 1| \rangle \supset \dots$$ and the strictly ascending chain of $\onenu$-sets corresponding to it.
$$\skel(\la_1) \subset \skel(|\la_1| + |\la_2+1|) \subset \dots \subset \skel(|\la_1| + |\alpha^{-k}\la_2 + 1|) \subset \dots =$$ $$\skel(\la_1) \subset \skel(\la_1) \cap \skel(\la_2+1) \subset \dots \subset \skel(\la_1) \cap \skel(\alpha^{-k}\la_2 + 1) \subset \dots.$$
\end{exmp}
%
%
%
%
%
%
%
%
%

\begin{exmp}\label{exmp_decomposition}
Again, consider the principal $\nu$-kernel $\langle x \rangle \in
\PCon(F(x,y))$. Then $$\langle x \rangle = \langle |x| + (|y + 1|
\wedge |\frac{1}{y} + 1|) \rangle = \langle  (|x| + |y + 1|)
\wedge (|x| + |\frac{1}{y} + 1|) \rangle = \langle |x| + |y + 1|
\rangle \cap \langle |x| + |\frac{1}{y} + 1| \rangle.$$ So, we
have the nontrivial decomposition of $\skel(x)$ as $\skel(|x| + |y
+ 1|) \cup \skel(|x| + |\frac{1}{y} + 1|)$ (note that $\skel(|x| +
|y+1|) = \skel(x) \cap \skel(y+1)$,  and furthermore $\skel(|x| +
|\frac{1}{y} + 1|) = \skel(x) \cap \skel(\frac{1}{y} + 1)$). In a
similar way, using complementary order $\nu$-kernels,  one can
show that every principal $\nu$-kernel can be nontrivially
decomposed to a pair of principal $\nu$-kernels.
\end{exmp}
%
%
%
%
%
%
%
%
%

\subsection{Irreducible $\nu$-kernels}$ $

Examples \ref{exmp_infinite_chain} and \ref{exmp_decomposition}
demonstrate that the lattice of principal $\nu$-kernels
$\PCon(\overline{F(\Lambda)})$ (resp. $\PCon(\langle F \rangle)$)
is too rich to define reducibility or finite dimension. (See
\cite{Bi} for a discussion of infinite dimension.) Moreover, these
examples suggest that this richness is caused by order
$\nu$-kernels. This motivates us to consider $\Theta$-reducibility
for a suitable sublattice of $\nu$-kernels $\Theta \subset
\PCon(\overline{F(\Lambda)})$ (resp. $\Theta \subset \PCon(\langle
F \rangle)$).

There are various  families of
 $\nu$-kernels that could be utilized to define the notions of reducibility,
dimensionality, and so forth. We take $\Theta$ to be the
sublattice generated by HP-kernels, because of its connection to
the (local) dimension of the linear spaces (in logarithmic scale)
defined by the $\onenu$-set corresponding to a $\nu$-kernel.
Namely, HP-kernels, and more generally   HS-kernels, define affine
subspaces of $F^{(n)}$ (see \cite[\S 9.2]{Kern}).
%
%
%
 We work with Definition~\ref{Omega}.

\begin{defn}\label{irre} A $\nu$-kernel $\langle f \rangle \in \Omega(\overline{F(\Lambda)})$ is
\textbf{reducible} if there are  $ \langle g \rangle, \langle h
\rangle \in \Omega(\overline{F(\Lambda)})$ for which  $ \langle g
\rangle, \langle h \rangle \not \subseteq \langle f \rangle$ but
$\langle g \rangle \cap \langle h \rangle \subseteq \langle f
\rangle$.
\end{defn}


\begin{lem}\label{irred}
 $\langle f \rangle$ is
reducible iff  $\langle f \rangle  = \langle g \rangle \cap \langle
h \rangle$ where $\langle f \rangle \neq \langle g \rangle$ and
$\langle f \rangle \neq \langle h \rangle.$
\end{lem}
\begin{proof}
Assume $\langle f \rangle$ admits the stated condition. If $\langle
f \rangle  \supseteq \langle g \rangle \cap \langle h \rangle$, then
$\langle f \rangle = \langle f \rangle \cdot \langle f \rangle  =
(\langle g \rangle \cdot \langle f \rangle )\cap (\langle h \rangle
\cdot \langle f \rangle)$. Thus $\langle f \rangle = \langle g
\rangle \cdot \langle f \rangle$ or $\langle f \rangle = \langle h
\rangle \cdot \langle f \rangle$, implying $\langle f \rangle
\supseteq \langle g \rangle$ or $\langle f \rangle \supseteq \langle
g \rangle$. The converse is obvious.\end{proof}

\begin{lem}
Let $\langle f \rangle$ be an HP-kernel. Then for any HP-kernels
$\langle g \rangle$ and $\langle h \rangle$ such that $\langle f
\rangle = \langle g \rangle \cap \langle h \rangle$ either $\langle
f \rangle = \langle g \rangle$ or $\langle f \rangle = \langle h
\rangle$. In other words, every HP-kernel is irreducible.
\end{lem}
\begin{proof}
If $\langle f \rangle = \langle g \rangle \cap \langle h \rangle$
then $\langle f \rangle \subseteq  \langle g \rangle$ thus $f \in
\langle g \rangle$. As both  $f$ and $g$ are $\scrL$-monomials (up
to equivalence),   Lemma~\ref{prop_HP_element_is_a_generator} yields
$\langle f \rangle = \langle g \rangle$, which in turn, by
Lemma~\ref{irred}, implies that $\langle f \rangle$ is irreducible.
\end{proof}
\begin{cor}\label{cor_HS-decomposition}
Any HS-kernel $\langle f \rangle$ is irreducible.
\end{cor}
\begin{proof}
If $\langle f \rangle =  \langle g \rangle \cap \langle h \rangle$
for HP-kernels $\langle g \rangle$ and $\langle h \rangle$, then
$\langle f \rangle \subseteq \langle g \rangle$. But $\langle f
\rangle$ is a product $\langle f_1 \rangle \cdots \langle f_t
\rangle$  of finitely many HP-kernels. For each $1 \le j \le t$,
$\langle f_j \rangle \subseteq \langle g \rangle$  yielding $\langle
g \rangle = \langle f_j  \rangle$  by
Lemma~\ref{prop_HP_element_is_a_generator}, and so $\langle f
\rangle = \langle g \rangle$.
\end{proof}

\begin{cor}\label{cor_reducibility_of_HS_kernels}
If $\langle f \rangle = \langle g \rangle \cap \langle h \rangle$
for  HS-kernels $\langle g \rangle$ and $\langle h \rangle$, then
either $\langle f \rangle = \langle g \rangle$ or $\langle f \rangle
= \langle h \rangle$.
\end{cor}
\begin{proof} Otherwise, since
$\langle g \rangle$ and $\langle h \rangle$ are finite products of
HP-kernels, there are HP-kernel $\langle g' \rangle \subseteq
\langle g \rangle$   and  $\langle h' \rangle \subseteq \langle h
\rangle$   such that $\langle g' \rangle \not \subseteq \langle f
\rangle$ and $\langle h' \rangle \not \subseteq \langle f \rangle$.
But $\langle g' \rangle \cap \langle h' \rangle \subseteq \langle g
\rangle \cap \langle h \rangle = \langle f \rangle$, implying
$\langle f \rangle$ is reducible, contradicting Corollary
\ref{cor_HS-decomposition}.
\end{proof}

\begin{prop} The irreducible $\nu$-kernels in the
  lattice generated by HP-kernels are precisely the HS-kernels.\end{prop}
\begin{proof}
 This follows from Corollary \ref{cor_reducibility_of_HS_kernels}, since all proper intersections in the lattice generated by HP-kernels are reducible.
  (Note that HP-kernels are also HS-kernels.)
\end{proof}

\begin{defn}\label{hyperspec}
The \textbf{hyperspace spectrum} of $\overline{F(\Lambda)}$,
denoted $\HSpec(\overline{F(\Lambda)})$, is the family of
irreducible $\nu$-kernels in $\Omega(\overline{F(\Lambda)})$.
\end{defn}

\begin{cor}
$\HSpec(\overline{F(\Lambda)})$ is the family of HS-kernels in
$\Omega(\overline{F(\Lambda)})$, which is precisely the family of
HS-fractions of $\overline{F(\Lambda)}$.
\end{cor}

%

\begin{defn}\label{height0}
A chain $P_0 \subset P_1 \subset \dots \subset P_t$ in
$\HSpec(\overline{F(\Lambda)})$  of HS-kernels of
$\overline{F(\Lambda)}$ is said to have \textbf{length}~$t$. An
HS-kernel $P$ has \textbf{height} $t$ (denoted $\hgt(P)=t$) if there
is a chain of length $t$ in $\HSpec(\overline{F(\Lambda)})$
terminating at $P$, but no chain of length $t+1$ terminates at $P$.
\end{defn}

\begin{rem}\label{rem_correspondence_of_quontient_HSpec}
Let $L$ be a $\nu$-kernel in $\PCon(\overline{F(\Lambda)})$.
Consider the canonical homomorphism $\phi_L :
\overline{F(\Lambda)} \rightarrow \overline{F(\Lambda)}/L$. Since
the image of a principal $\nu$-kernel is generated by the image of
any of its generators, $\phi_L(\langle f \rangle)~=~\langle
\phi_L(f) \rangle$ for any HP-kernel $\langle f \rangle$. Choosing
$f$ to be an $\scrL$-monomial,  $\langle \phi_L(f) \rangle$ is a
nontrivial HP-kernel in $ \overline{F(\Lambda)}/L$ if and only if
$\phi_L(f) \not \in F$. Thus, the set of HP-kernels of
$\overline{F(\Lambda)}$ mapped to HP-kernels of
$\overline{F(\Lambda)}/L$ is
\begin{equation}\label{eq_subset_of_Omega}
\left\{ \langle g \rangle : \langle g \rangle \cdot \langle F
\rangle \supseteq \phi_L^{-1}\left(\langle F \rangle\right) = L
\cdot \langle F \rangle \right\}.
\end{equation}
As $\phi_L$ is an   $F$-homomorphism, it respects $\vee, \wedge$
and $| \cdot |$, and  thus $\phi_{L}( (\Omega
(\overline{F(\Lambda)}), \cap, \cdot)) = (\Omega
(\overline{F(\Lambda)}/L),\cap, \cdot)$.
 In fact Theorem~\ref{thm_kernels_hom_relations} yields a correspondence
identifying $\HSpec(\overline{F(\Lambda)}/L)$ with the subset of
$\HSpec(\overline{F(\Lambda)})$ which consists of all HS-kernels
$P$ of $\overline{F(\Lambda)}$ such that   $P \cdot \langle F
\rangle \supseteq L\cdot\langle F \rangle$.  \end{rem}

\begin{lem} The above correspondence extends to a
correspondence identifying $\Omega(\overline{F(\Lambda)}/L)$ with
the subset~\eqref{eq_subset_of_Omega} of
$\Omega(\overline{F(\Lambda)})$. Under this correspondence, the
maximal HS-kernels of $\overline{F(\Lambda)}/L$ correspond to
maximal HS-kernels of $\overline{F(\Lambda)}$, and reducible
$\nu$-kernels of $\overline{F(\Lambda)}/L$ correspond to reducible
$\nu$-kernels of $\overline{F(\Lambda)}$.
\end{lem}
\begin{proof} The latter assertion is obvious since $\wedge$
is preserved under homomorphisms. For the first assertion,
$(\overline{F(\Lambda)}/L)/(P/L) \cong \overline{F(\Lambda)}/P$ by
Theorem~\ref{thm_nother_1_and_3}, so simplicity of the quotients
is preserved. Hence, so is maximality of  $P/L$ and $P$.
\end{proof}

\begin{defn}\label{defn_Hdim}
The Hyperdimension of $\overline{F(\Lambda)}$, written $\Hdim
\overline{F(\Lambda)}$ (if it exists), is the maximal height of
the HS-kernels in $\overline{F(\Lambda)}$.
\end{defn}

\subsection{Decompositions}$ $

Let us garner some information about reducible $\nu$-kernels, from
rational functions. Suppose $f\in \overline{F(\Lambda)}$. We write
$f = \sum_{i=1}^{k}f_i$ where each $f_i$ is of the form $
{g_i}{h_i}^*$ with $g_i,h_i \in F[\Lambda]$ and $g_i$  a monomial.
(Thus $f_i(\bfa) = \frac{g_i(\bfa)}{h_i(\bfa)}$ when $h_i(\bfa)$
is tangible.) We also assume that this sum is \textbf{irredundant}
in the sense that we cannot remove any of the summands and still
get $f$.
If each time the value $1$ is attained by one of the terms $f_i$ in this expansion and all other terms attain values   $\le 1$, then $\tilde{f} = \bigwedge_{i=1}^{k}|f_i|$  defines the same $\onenu$-set as $f$. Moreover, if $f \in \langle F \rangle$ then $\tilde{f} \wedge |\alpha| \in \langle F \rangle$, for $\alpha \in F \setminus \{1\}$  is also a generator of $\langle f \rangle$. The reason we take $\tilde{f} \wedge |\alpha|$ is that we have no guarantee that each of the $f_i$'s in the above expansion is bounded.\\
We can generalize this idea as follows:

 We call $f\in \overline{F(\Lambda)}$
\textbf{reducible} if we can write $f = \sum_{i=1}^{k}f_i$ as above,
such that for every $1 \leq i \leq k$ the following condition holds:
\begin{equation*}
f_i(\bfa) \nucong 1  \Rightarrow  \ \ f_j(\bfa) \le _{\nu} 1, \
\forall j \neq i.
\end{equation*}

\begin{defn}\label{defn_decomposition1}
Let $f \in \overline{F(\Lambda)}$. A
$\Theta$-\textbf{decomposition} of $f$ is an expression of the
form
\begin{equation}\label{eq_decomposition}
|f| = |u| \wedge |v|
\end{equation}
with $u,v $ $\Theta$-elements in $\overline{F(\Lambda)}$.

The decomposition \eqref{eq_decomposition} is said to be
\emph\textbf{trivial} if  \ $f \sim_{\FF} u$ \ or \ $f \sim_{\FF} v$
(equivalently $|f|~\sim_{\FF}~|u|$ or $|f|~\sim_{\FF}~|v|$).
 Otherwise, the decomposition  is said to be \textbf{nontrivial}.
\end{defn}
%
%
%

\begin{lem}\label{lem_reducibility_and_decomposition}
Suppose $f \in \overline{F(\Lambda)}$ is a $\Theta$-element. Then
$\langle f \rangle$ is reducible if and only if there exists some
generator $f'$ of $\langle f \rangle$  that has a nontrivial
$\Theta$-decomposition.
\end{lem}
\begin{proof}
If $\langle f \rangle$ is reducible , then there exist
 $\nu$-kernels $\langle u \rangle$ and $\langle v \rangle$ in
$\Theta$ such that $\langle f \rangle = \langle u \rangle \cap
\langle v \rangle$ where $\langle f \rangle \neq \langle u \rangle$
and $\langle f \rangle \neq \langle v \rangle$. Since $\langle u
\rangle \cap \langle v \rangle = \langle |u|~\wedge~|v|\rangle$ we
have the nontrivial $\Theta$-decomposition $f' = |u| \wedge |v|$
(which is a generator of $\langle f \rangle$).

Conversely, assume that $f' = |u| \wedge |v|$ is a nontrivial
$\Theta$- decomposition for some $f' \sim_{\FF} f$. Then $\langle f
\rangle = \langle f' \rangle = \langle |u| \wedge |v| \rangle =
\langle u \rangle \cap \langle v \rangle$. Since the decomposition
$f' = |u| \wedge |v|$ is nontrivial, we have that $u \not \sim_{\FF}
f'$ and $v \not \sim_{\FF} f'$, and thus $\langle |u| \rangle =
\langle u \rangle \neq \langle f' \rangle = \langle f \rangle$.
Similarly, $\langle v \rangle \neq \langle f \rangle$. Thus, by
definition, $\langle f \rangle$ is reducible.
\end{proof}

\begin{flushleft} We can equivalently rephrase Lemma \ref{lem_reducibility_and_decomposition} as follows: \end{flushleft}
\begin{rem}
$f$ is reducible if and only if   some $f' \sim_{\FF} f$ has a
nontrivial $\Theta$-decomposition.
\end{rem}

A  question immediately arising from Definition
\ref{defn_decomposition1} and
Lemma~\ref{lem_reducibility_and_decomposition} is:

If $f \in \overline{F(\Lambda)}$ has a nontrivial
$\Theta$-decomposition and $g \sim_{\FF} f$, does $g$ also have a
nontrivial $\Theta$-decomposition? If so, how is this pair of
decompositions related?

In the next few paragraphs we provide an answer to both of these
questions, for  $\Theta = \PCon(\langle \mathscr{R} \rangle)$.

\begin{rem}
$\sum_{i=1}^{k}s_i(a_i \wedge b_i)^{d(i)} =
\left(\sum_{i=1}^{k}s_i a_i^{d(i)} \right) \wedge
\left(\sum_{i=1}^{k}s_i b_i^{d(i)}\right),$ $ \forall s_1, ...,
s_k , a_1,....,a_k  , b_1,....,b_k \in \overline{F(\Lambda)}$,
and  $d(i) \in \mathbb{N}_{\geq 0}$.
\end{rem}

\begin{rem}
If $h_1,...,h_k \in \overline{F(\Lambda)}$   such that each $h_i
\ge_{\nu} 1$, then $\sum_{i=1}^{k}s_ih_i \ge_{\nu}  1 $ for every
$s_1,...,s_k \in \overline{F(\Lambda)}$ such that
$\sum_{i=1}^{k}s_i \nucong 1$.
\end{rem}

\begin{thm}\label{thm_every_gen_is_reducible1}
(For $\Theta = \PCon(\langle \mathscr{R} \rangle)$.) If $\langle f
\rangle$ is a (principal) reducible $\nu$-kernel, then there exist
$\Theta$-elements $g,h \in \overline{F(\Lambda)}$ such that $|f| =
|g| \wedge |h|$ and  $|f| \not \sim_{\FF} |g|,|h|$.
\end{thm}

\begin{proof}
If $\langle f \rangle$ is a principal reducible $\nu$-kernel, then
there exists $f' \sim_{\FF} f$ such that $f' = |u| \wedge |v| =
\min(|u|,|v|)$ for $\Theta$-elements $u,v \in \langle \mathscr{R}
\rangle $ with $f'~\not\sim_{\FF}~|u| , |v|$. Then $|f| \in
\langle f' \rangle$ since $f'$ is a generator of $\langle f
\rangle$,
 so there exist   $s_1,...,s_k \in \overline{F(\Lambda)}$ such that $\sum_{i=1}^{k}s_i = 1$
 and $|f| = \sum_{i=1}^{k}s_i (f')^{d(i)}$ with $d(i) \in \mathbb{N}_{\geq 0}$. ($d(i) \geq 0$ since $|f| \ge_{\nu}
 1$.)
Thus $$f = \sum_{i=1}^{k}s_i (|u| \wedge |v|)^{d(i)} =
\sum_{i=1}^{k}s_i (\min(|u|,|v|))^{d(i)} =
\min\left(\sum_{i=1}^{k}s_i|u|^{d(i)},\sum_{i=1}^{k}s_i|v|^{d(i)}\right)
=|g| \wedge |h|$$ where $g = |g| = \sum_{i=1}^{k}s_i|u|^{d(i)}, h =
|h| = \sum_{i=1}^{k}s_i|v|^{d(i)}$.

Now   $\langle |f| \rangle \subseteq \langle |g| \rangle \subseteq
\langle |u| \rangle$  and $\langle |f| \rangle \subseteq  \langle
|h| \rangle \subseteq \langle |v| \rangle$, implying $\Skel(f)
\supseteq \Skel(g) \supseteq \Skel(u)$ and $\Skel(f) \supseteq
\Skel(h) \supseteq \Skel(v)$.

We claim that $|g|$ and $|h|$ generate $\langle |u| \rangle $ and
$\langle |v| \rangle$, respectively. Indeed, $\Skel(f')
=\Skel(|f|)$, since $f' \sim_{\FF} |f|$ and thus for any $\bfa \in
F^{(n)}$, $f'(\bfa)=1 \Leftrightarrow |f|(\bfa)=1$. Let $s_j
(f')^{d(j)}$ be a dominant term of $|f|$ at $\bfa$, i.e.,
$$|f| \nucong \sum_{i=1}^{k}s_i(\bfa) (f'(\bfa))^{d(i)} \nucong s_j(\bfa) (f'(\bfa))^{d(j)}.$$
Then   $f(\bfa) \nucong 1 \Leftrightarrow s_j(\bfa)
(f'(\bfa))^{d(j)}\nucong1$. If $f'(\bfa) \nucong 1 $, then
$(f'(\bfa))^{d(j)}\nucong1$, so  $s_j(\bfa) \nucong 1$. Now, for
$\bfa \in \Skel(g)$. Then we have $$g(\bfa) \nucong
\sum_{i=1}^{k}s_i|u|^{d(i)} \nucong 1.$$ Let $s_t|u|^{d(t)}$ be a
dominant term of $g$ at $\bfa$. If $s_t(\bfa) \nucong1$ then
$|u|^{d(t)} \nucong 1$ and thus $u = 1$, and $\bfa \in \Skel(u)$.
Otherwise $s_t(\bfa) <_\nu 1$ (since $\sum_{i=1}^{k}s_i \nucong 1$)
and so, by the above, $s_t (f')^{d(t)}$ is not a dominant term of
$|f|$ at $\bfa$. Thus,  for every index j of a dominant term of
$|f|$ at $\bfa$, we have $j \ne t$ and
$$|u(\bfa)|^{d(j)} \nucong s_j(\bfa)|u(\bfa)|^{d(j)} < _\nu s_t(\bfa)|u(\bfa)|^{d(t)} \nucong g(\bfa) \nucong 1.$$

\begin{equation}\label{eq_1_thm_every_gen_is_reducible}
s_j(\bfa)(f'(\bfa))^{d(j)}\nucong s_j(\bfa) (|u|(\bfa) \wedge
|v|(\bfa))^{d(j)} \leq s_j(\bfa)|u(\bfa)|^{d(j)} <_\nu 1.
\end{equation}
On the other hand, $f'(\bfa)\nucong1$ since $\Skel(f) \supseteq
\Skel(g)$, implying $s_j(\bfa)(f'(\bfa))^{d(j)}\nucong1$,
contradicting \eqref{eq_1_thm_every_gen_is_reducible}. Hence,
$\Skel(g) \subseteq \Skel(u)$, yielding $\Skel(g) = \Skel(u)$, which
implies that $g$ is a generator of $\langle |u| \rangle = \langle u
\rangle$. The proofs for $h$ and $|v|$ are analogous.

Consequently,   $g \sim_{\FF} |g| \sim_{\FF} |u|$ and $h \sim_{\FF}
|h| \sim_{\FF} |v|$. Since $|f|~\sim_{\FF}~f' \not \sim_{\FF} |u| ,
|v|$ we conclude that $|f| \not \sim_{\FF} |g|,|h|$.
\end{proof}

\begin{cor}\label{cor_generator_structure}
For $f \in \langle F \rangle$, if $|f| = \bigwedge_{i=1}^{s} |f_i|$
for   $f_i \in \langle F \rangle$, then for any $g \sim_{\FF} f$, we
have $|g| = \bigwedge_{i=1}^{s} |g_i|$, with $g_i \sim_{\FF}  f_i$
for $i = 1,...,s$.
\end{cor}
\begin{proof}
Iterate Theorem \ref{thm_every_gen_is_reducible1}.
\end{proof}
%

%

\ \\

\begin{cor}\label{cor_norm_decomposition}
If $\langle f \rangle$ is a $\nu$-kernel in $\Theta$, then
$\langle f \rangle$ has a nontrivial decomposition $\langle f
\rangle = \langle g \rangle \cap \langle h \rangle$ if and only if
$|f|$ has a nontrivial decomposition $|f|~=~|g'|~\wedge~|h'|$ with
$|g'| \sim_{\FF} g$ and $|h'| \sim_{\FF} h$.
\end{cor}
\begin{proof}
If $|f| = |g'| \wedge |h'|$  then, since $|g'| \sim_{\FF} g$ and
$|h'| \sim_{\FF} h$ we have $$\langle f \rangle  = \langle |f|
\rangle =  \langle |g'| \wedge |h'| \rangle = \langle |g'| \rangle
\cap \langle  |h'| \rangle = \langle g \rangle \cap \langle h
\rangle.$$ The converse is seen as in  the proof of
Theorem~\ref{thm_every_gen_is_reducible1}.
\end{proof}

Corollary \ref{cor_norm_decomposition} provides a
$\Theta$-decomposition of $|f|$, for every generator $f$ of a
reducible $\nu$-kernel in $\Theta$.

\begin{rem}\label{rem_wedge_equivalence}
By \cite[Corollary~4.1.25]{Kern},
$$\langle f \rangle \cap \langle g \rangle = \langle (f + f^*) \wedge (g + g^*)\rangle  = \langle |f| \wedge |g| \rangle.$$
But, in fact, $\langle f \rangle \cap \langle g \rangle = \langle f'
\rangle \cap \langle g' \rangle$  for any  $g' \sim_{\FF} g$ and $h'
\sim_{\FF} h$, so we could   take  $|g'| \wedge |f'|$ instead of
$|g| \wedge |f|$ on the righthand side of the equality, e.g.,
$\langle |f^k| \wedge |g^m| \rangle$ for any $m,k \in \mathbb{Z}
\setminus \{ 0 \}$.
\end{rem}

\begin{defn}
Let $\mathbb{S}$ be a semifield and let $a,b \in \mathbb{S}$. We say
that $a$ and $b$ are $\FF$-\textbf{comparable} if  there exist some
$a' \sim_\FF a$ and $b' \sim_\FF b$ such that  $|a'| \leq | b'| $ or
$|b'| \leq |a'|$.
\end{defn}

Since $|g| \wedge |h| = \min(|g|,|h|)$ we can utilize Remark
\ref{rem_wedge_equivalence} to get the following observation:

\begin{prop}
A $\Theta$-decomposition $f  \sim_{\FF} |g| \wedge |h| \in
\overline{F(\Lambda)}$ is nontrivial if and only if the
$\Theta$-elements $g$ and $h$ are not $\FF$-comparable.
\end{prop}
\begin{proof}
If $g$ and $h$ are $\FF$-comparable, then there exist some $g'
\sim_{\FF} g$ and $h' \sim_{\FF} h$ such that $|g'| \ge_{\nu} |h'|$
or $|h'| \ge_{\nu}  |g'|$. Without loss of generality, assume that
$|g'| \ge_{\nu}  |h'|$. Then $\langle  |g| \wedge |h| \rangle  =
\langle |g| \rangle \cap \langle |h| \rangle = \langle |g'| \rangle
\cap \langle |h'| \rangle = \langle  |g'| \wedge |h'| \rangle =
\langle \min(|g'|,|h')| \rangle = \langle  |g'| \rangle = \langle g'
\rangle = \langle  g \rangle$. Thus $\langle f \rangle = \langle g
\rangle$ so $f \sim_{\FF} g$ yielding that the decomposition is
trivial.

Conversely, if  $g$ and $h$ are not $\FF$-comparable then we claim
that $f \not \sim_{\FF} g$ and $f \not \sim_{\FF} h$. We must show
that $\langle h \rangle \not \subseteq \langle g \rangle$ and
$\langle g \rangle \not \subseteq \langle h \rangle$ respectively.
So assume that $\langle h \rangle  \supseteq \langle g \rangle$. In
view of Lemma~\ref{prop_HP_element_is_a_generator} and
\cite[Proposition~4.1.13]{Kern}, there exists some $f' \sim f$ such
that $f' = |h|^{k} \wedge g$. Note that $|h|^{k} \geq 1$ and $g \geq
1$ so $f' = |h|^{k} \wedge g \geq 1$ and thus $|f'| = f'$. Finally,
$$|f'| = |h|^{k} \wedge g \Leftrightarrow |f'| \leq |h|^{k}
\Leftrightarrow |f'| \in \langle |h| \rangle \Leftrightarrow f' \in
\langle h \rangle \Leftrightarrow f \in \langle h \rangle.$$
\end{proof}

\subsection{Convex dependence}$ $


\begin{defn}\label{defn_convex_dep}
An  HS-fraction  $f$ of $\overline{F(\Lambda)}$ is
\textbf{$F$-convexly dependent} on a set $A$   of HS-fractions  if
\begin{equation}\label{eq_defn_convex_dep}
f \in \left\langle \{g : g \in A \} \right\rangle \cdot \langle F
\rangle;
\end{equation}
otherwise $f$ is said to be \textbf{$F$-convexly independent} of
$A$. The set $A$ is said to be $F$-\textbf{convexly independent} if
$f$  is $F$-convexly independent of $A \setminus \{ f \}$, for every
$f \in A$. If $\{a_1,...,a_n \}$ is $F$-convexly dependent, then we
also say that $a_1,...,a_n$ are $F$-convexly dependent.

\end{defn}

Note that under the assumption that $g \in  \langle F \rangle
\setminus \{1\}$ for some $g \in A$, the condition in
\eqref{eq_defn_convex_dep} simplifies to $f \in \left\langle \{g : g
\in A \} \right\rangle $.

\begin{rem}
By  definition, an HS-fraction $f$ is $F$-convexly   dependent on
HS-fractions $\{g_1,...,g_t\} $  if and only if
$$\langle |f|  \rangle = \langle f  \rangle  \subseteq  \langle g_1,..., g_t \rangle \cdot
\langle F \rangle = \Big\langle \sum_{i=1}^{t} |g_i| \Big\rangle
\cdot \langle F \rangle
 = \Big\langle \sum_{i=1}^{t} |g_i|  +  |\alpha| \Big\rangle,$$
  for any element $\alpha$ of $F$ for which $\a^\nu \ne \onenu.$
\end{rem}

\begin{exmp}
For any $\alpha \in F$ and any $f \in \overline{F(\Lambda)}$,
$$|\alpha f| \leq |f|^2  + |\alpha|^2 =(|f|  +  |\alpha|)^2.$$
Thus $\alpha f \in \left\langle (|f|  +  |\alpha|)^2 \right\rangle =
\langle |f|  +  |\alpha| \rangle = \langle f \rangle \cdot \langle F
\rangle$. In particular, if $f$ is an HS-fraction, then $\alpha f$
is $F$-convexly dependent on $f$.
\end{exmp}

As a consequence of Lemma~\ref{prop_HP_element_is_a_generator}, if
two $\scrL$-monomials $f,g$, satisfy $g \in \langle f \rangle$,
then $\langle g \rangle = \langle f \rangle$. In other words,
either $\langle g \rangle = \langle f \rangle$ or $\langle g
\rangle \not \subseteq \langle f \rangle$ and $\langle f \rangle
\not \subseteq \langle g \rangle$. This motivates us to restrict
the convex dependence relation to the set of $\scrL$-monomials.
This will be justified later by showing that for each $F$-convexly
independent subset of HS-fractions of order $t$ in
$\overline{F(\Lambda)}$, there exists an $F$-convexly independent
subset of $\scrL$-monomials having order $ \geq t$ in
$\overline{F(\Lambda)}$. Let us see that convex-dependence is an
abstract dependence relation.

\begin{prop}\label{prop_abstract_dependence_properties}

Let $A, A_1 \subset  \overline{F(\Lambda)}$ be  sets of
HS-fractions, and let $f$ be an HS-fraction.
\begin{enumerate}
  \item If $f \in A$, then $f$ is $F$-convexly dependent on $A$.
  \item If $f$ is $F$-convexly dependent on $A$ and   each $a \in A$ is
  $F$-convexly dependent on $A_1$, then $f$ is $F$-convexly dependent on $A_1$.
  \item If $f$ is $F$-convexly dependent on $A$, then $f$ is $F$-convexly dependent on $A_0$ for some finite subset $A_0$ of $A$.
\end{enumerate}
\end{prop}
\begin{proof}$ $
(1)  $f \in \langle A \rangle \subseteq \langle A \rangle \cdot
\langle F \rangle $.

(2)  $\langle A \rangle \subseteq \langle A_1 \rangle \cdot \langle
F \rangle$ since $a$ is convexly dependent on $A_1$  for each $a \in
A$. If $f$ is $F$-convexly dependent on~$A$, then $f \in \langle A
\rangle \cdot \langle F \rangle \subseteq \langle A_1 \rangle \cdot
\langle F \rangle$, so, $f$ is $F$-convexly dependent on $A_1$.

 (3) $a \in
\langle A \rangle \cdot \langle F \rangle$, so by Proposition
\ref{prop_ker_stracture_by group} there exist some $s_1,...,s_k
\in \overline{F(\Lambda)} $ and  $g_1,...,g_k \in G(A \cup F)
\subset \langle A \rangle \cdot \langle F \rangle$, where $G(A
\cup F)$ is the group generated by $A \cup F$, such that
$\sum_{i=1}^{k}s_i = 1$ and $a = \sum_{i=1}^{k}s_i g_i^{d(i)}$
with $d(i) \in \mathbb{Z}$. Thus $a \in \langle g_1,...,g_k
\rangle$ and  $A_0  = \{g_1,...,g_k\}$.
\end{proof}

From now on, we assume that the \vsemifield0 $F$ is divisible.

\begin{prop}[Steinitz exchange axiom]\label{prop_convex_dependence_of_HP_property}
 Let $S =
\{b_1,...,b_t\} \subset \operatorname{HP}(\overline{F(\Lambda)})$
and let $f$ and $b$ be elements of
$\operatorname{HP}(\overline{F(\Lambda)})$. If $f$ is $F$-convexly
dependent on $S \cup \{ b \}$ and $f$ is $F$-convexly independent
of $S$, then $b$ is $F$-convexly dependent on $S \cup \{ f \}$.
\end{prop}
\begin{proof}
We may assume that $\alpha \in S$ for some $\alpha \in F$. Since
$f$ is $F$-convexly independent of $S$, by definition $f \not \in
\langle S \rangle$ this implies that $\langle S \rangle \subset
\langle S \rangle \cdot \langle f \rangle$ (for otherwise $\langle
f \rangle \subseteq \langle S \rangle$ yielding that $f$ is
$F$-convexly dependent on $S$).  Since $f$ is $F$-convexly
dependent on $S \cup \{ b \}$, we have that  $f \in \langle S \cup
\{ b \} \rangle = \langle S \rangle \cdot \langle b \rangle$. In
particular, we get that $b \not \in \langle S \rangle \cdot
\langle F \rangle$ for otherwise $f$ would be dependent on $S$.
Consider the quotient map $\phi : \overline{F(\Lambda)}
\rightarrow \overline{F(\Lambda)}/\langle S \rangle$. Since $\phi$
is a semifield epimorphism and $f,b  \not \in \langle S \rangle
\cdot \langle F \rangle = \phi^{-1}(\langle F \rangle)$, we have
that $\phi(f)$ and $\phi(b)$ are not in $F$ thus are
$\scrL$-monomials in the semifield $\Im(\phi)
=\overline{F(\Lambda)}/\langle S \rangle$. By the above, $\phi(f)
\neq 1$ and $\phi(f) \in \phi(\langle b \rangle) = \langle \phi(b)
\rangle$. Thus, $\langle \phi(f) \rangle =\langle \phi(b) \rangle$
by Lemma~\ref{prop_HP_element_is_a_generator}. So $\langle S
\rangle \cdot \langle f \rangle = \phi^{-1}(\langle \phi(f)
\rangle) = \phi^{-1}(\langle \phi(b) \rangle) = \langle S \rangle
\cdot \langle b \rangle$, consequently $b \in \langle S \rangle
\cdot \langle b \rangle = \langle S \rangle \cdot \langle f
\rangle = \langle S \cup \{ f \} \rangle $, i.e., $b$ is
$F$-convexly dependent on $S \cup \{ f \}$.
\end{proof}

\begin{defn}\label{defn_convex_span_in_semifield_of_fractions}
Let $A \subseteq \operatorname{HP}(\overline{F(\Lambda)})$. The
\textbf{convex span} of $A$ over $F$ is the set
\begin{equation}
\ConSpan_{F}(A) = \{ a \in
\operatorname{HP}(\overline{F(\Lambda)}) : a \text{ is
$F$-convexly dependent on $A$} \}.
\end{equation} For a \vsemifield0 $\mathrm{K} \subseteq \overline{F(\Lambda)}$  such that
$F \subseteq \mathrm{K}$.
 a set  $A \subseteq
\operatorname{HP}(\overline{F(\Lambda)})$  is said to
\textbf{convexly span} $\mathrm{K}$ over $F$ if
$$\operatorname{HP}(\mathrm{K}) = \ConSpan_{F}(A).$$
\end{defn}

\begin{rem}
$\ConSpan(\{ f_1,...,f_m\}) =  \langle f_1,...,f_m \rangle \cdot
\langle F \rangle.$
\end{rem}

In view of Propositions \ref{prop_abstract_dependence_properties}
and \ref{prop_convex_dependence_of_HP_property}, convex dependence
on  $\operatorname{HP}(\overline{F(\Lambda)})$ is an abstract
dependence relation. Then by \cite[Chapter 6]{Ro}, we have:

\begin{cor}\label{cor_basis_for_HP}
Let $V \subset \operatorname{HP}(\overline{F(\Lambda)})$. Then any
set convexly spanning $V$ contains a basis $B_V \subset V$, which
is a maximal convexly independent subset of unique cardinality
such that
$$\ConSpan(B_V) = \ConSpan(V).$$
\end{cor}

\begin{exmp}\label{exmp_basis_for_semifield_of_fractions}
By
Lemma~\ref{prop_maximal_kernels_in_semifield_of_fractions_part2},
the maximal $\nu$-kernels in $\PCon(\overline{F(\Lambda)})$ are
HS-fractions of the form $L_{(\alpha_1,...,\alpha_n)}=\langle
\alpha_1 x_1, \dots, \alpha_n x_n \rangle$ for any
$\alpha_1,...,\alpha_n \in F$. In view of
Corollary~\ref{cor_basis_for_HP},
$$\overline{F(\Lambda)} = \ConSpan(\{ \alpha_1 x_1, \dots, \alpha_n x_n \}),$$
i.e., $\{ \alpha_1 x_1, \dots, \alpha_n x_n \}$ convexly spans
$\overline{F(\Lambda)}$ over $F$.  Now $\alpha_k x_k \not \in
\langle \bigcup_{j \neq k}\alpha_j x_j \rangle \cdot~\langle
F\rangle$, since there are no order relations between $\alpha_i
x_i$ and the elements of  $\{ \alpha_j x_j : j \neq i \} \cup
\{\alpha\ : \alpha \in F \}$.  Thus, for arbitrary
$\alpha_1,...,\alpha_n \in F$, $\{ \alpha_1 x_1, \dots, \alpha_n
x_n \}$ is $F$-convexly independent, constituting a basis for
$\overline{F(\Lambda)}$.

\begin{defn}\label{condim}
Let $V \subset \operatorname{HP}(\overline{F(\Lambda)})$ be a set
of $\scrL$-monomials. We define the \textbf{convex  dimension} of
$V$, $\condeg(V)$, to be $|B|$ where $B$ is a basis for $V$.
\end{defn}

\begin{exmp}\label{rem_convexity_degree_of_semifield_of_fractions}
$\condeg\left(\overline{F(\Lambda)}\right) = n$, by Example
\ref{exmp_basis_for_semifield_of_fractions}.
\end{exmp}

\begin{rem}\label{rem_dependence_is_a_lattice}
If $S \subset \operatorname{HP}(\overline{F(\Lambda)})$, then for
any $f, g \in \overline{F(\Lambda)}$ such that $f,g \in
\ConSpan(S)$
$$|f|  +  |g| \in \ConSpan(S) \ \ \text{and} \ \ |f| \wedge |g| \in \ConSpan(S).$$
\end{rem}
\begin{proof}
First we prove that $|f|  +  |g| \in \ConSpan(S)$. Since $\langle f
\rangle \subseteq \langle S \rangle \cdot \langle F \rangle$ and
$\langle g \rangle \subseteq \langle S \rangle \cdot \langle F
\rangle$, we have $\langle f, g \rangle =  \langle |f|  +  |g|
\rangle = \langle f \rangle \cdot \langle g \rangle \subseteq
\langle S \rangle \cdot \langle F \rangle$.  $|f| \wedge |g| \in
\ConSpan(S)$,  since $\langle |f| \wedge|g| \rangle = \langle f
\rangle \cap \langle g \rangle \subseteq \langle g \rangle \subseteq
\langle  S \rangle \cdot \langle F \rangle$.
\end{proof}

\begin{rem}\label{rem_HS_kernel_finitely_spanned}
If $K$ is an HS-kernel, then $K$ is generated by an HS-fraction
$f~\in~\overline{F(\Lambda)}$ of the form $f = \sum_{i=1}^{t}
|f_i|$ where $f_1,...,f_t$ are $\scrL$-monomials. So,
$$\ConSpan(K) = \langle F \rangle \cdot K =\langle F \rangle \cdot \langle f \rangle = \langle F \rangle \cdot \Big\langle \sum_{i=1}^{t} |f_i| \Big\rangle = \langle F \rangle \cdot \prod_{i=1}^{t} \langle f_i \rangle$$ $$= \langle F \rangle \cdot \langle f_1,...,f_t \rangle$$ and so, $\{f_1,..., f_t \}$ convexly spans $\langle F \rangle \cdot K$.
\end{rem}

\begin{rem}\label{rem_HS_depend_on_HP}
Let $f$ be an HS-fraction. Then $f \sim_{\FF} \sum_{i=1}^{t}
|f_{i}|$ where $f_{i}$ are   $\scrL$-monomials. Hence $f$ is
$F$-convexly dependent on $\{ f_1,...,f_t \}$, since $\langle f
\rangle = \prod_{i=1}^{t} \langle f_i \rangle = \langle \{
f_1,...,f_t \} \rangle$.
\end{rem}

\begin{lem}\label{lem_HS_HP_convexity_relations}
 Suppose $\{ b_1, ...,b_m \}$ is a set of HS-fractions,
 such that $b_i\sim_{\FF}\sum_{j=1}^{t_i} |f_{i,j}|$,  where $f_{i,j}$ are $\scrL$-monomials.
 Then $b_1$ is $F$-convexly dependent on $\{ b_2, ...,b_m \}$ if and only if all of its
 summands $f_{1,r}$ for $1 \leq r \leq t_1$ are $F$-convexly dependent on $\{ b_2, ...,b_m \}$.
\end{lem}
\begin{proof}
If $b_1$ is $F$-convexly dependent on $\{ b_2, ...,b_m \}$, then
$$ \prod_{j=1}^{t_1} \langle f_{1,j} \rangle = \Big\langle \sum_{j=1}^{t_1} |f_{1,j}|
\Big\rangle = \langle b_1 \rangle \subseteq \big\langle \{ b_1,
...,b_m \} \big\rangle .$$ Hence $ f_{1,r} \in \prod_{j=1}^{t_1}
\langle f_{1,j} \rangle$ is $F$-convexly   dependent on $\{ b_1,
...,b_m \}$ and by
 Remark \ref{rem_HS_depend_on_HP} $f_{1,r}$ is $F$-convexly dependent on  $ \left\{ f_{i,j} : 2 \leq i \leq m ;\ 1 \leq j \leq t_i \right\}$.
Conversely, if each~$f_{1,r}$ is $F$-convexly dependent on $\{ b_2,
...,b_m \}$ for $1 \leq r \leq t_1$, then there exist some
$k_1,...,k_{t_1}$ such that   $|f_{1,r}| \leq
\sum_{i=2}^{m}\sum_{j=1}^{t_i}  |f_{i,j}|^{k_r}$. Hence $b_1$ is
$F$-convexly dependent on $\{ b_2, ...,b_m \}$, by
Corollary~\ref{cor_principal_ker_by_order}.
\end{proof}

\begin{lem}\label{lem_HS_base_gives_rise_to_HP_base}
Let $V = \{ f_1, ...,f_m \}$ be a $F$-convexly independent set of
HS-fractions, with $f_i \sim_{\FF} \sum_{j=1}^{t_i} |f_{i,j}|$ for
$\scrL$-monomials $f_{i,j}$. Then there exists an $F$-convexly
independent subset $$S_0 \subseteq S =\left\{ f_{i,j} : 1 \leq i
\leq m ; 1 \leq j \leq t_i \right\}$$ such that $|S_0| \geq |V|$ and
$\ConSpan(S_0) = \ConSpan(V)$.
\end{lem}
\begin{proof}
By Remark \ref{rem_HS_depend_on_HP}, \  $f_i$ is dependent on the
set of $\scrL$-monomials $\{ f_{i,j} :  1 \leq j \leq t_i \} \subset
S$ for each    $1 \leq i \leq m$, implying $\ConSpan(S) =
\ConSpan(V)$. By Corollary \ref{cor_basis_for_HP},  $S$ contains a
maximal $F$-convexly independent subset $S_0$ such that
$\ConSpan(S_0) = \ConSpan(S)$ which, by
Lemma~\ref{lem_HS_HP_convexity_relations}, we can shrink down to a
base.
\end{proof}

\begin{lem}\label{rem_HP_not_contains_contained_H}
The following hold for an $\scrL$-monomial $f$:
\begin{enumerate}
  \item $\langle F \rangle \not \subseteq \langle f \rangle$.
  \item If $\overline{F(\Lambda)}$ is not bounded, then $\langle f  \rangle \not \subseteq \langle F \rangle$.
\end{enumerate}
\end{lem}
\begin{proof} By definition, an $\scrL$-monomial is not bounded from below. Thus
$\langle f \rangle \cap F = \{ 1 \}$, yielding $\langle F \rangle
\not \subseteq \langle f \rangle$. For the second assertion,   an
HP-kernel is not bounded when $\overline{F(\Lambda)}$ is not
bounded, so $\langle f \rangle \not \subseteq \langle F \rangle$.
\end{proof}

A direct consequence of Lemma~\ref{rem_HP_not_contains_contained_H}
is:
\begin{lem}
If $\overline{F(\Lambda)}$ is not bounded, then any nontrivcial
HS-kernel (i.e., $\ne \langle 1 \rangle$) is   $F$-convexly
independent.
\end{lem}
\begin{proof}
By Lemma~\ref{rem_HP_not_contains_contained_H} the assertion is true
for HP-kernels, and thus for   HS-kernels, since every HS-kernel
contains some HP-kernel.
\end{proof}

\subsection{Computing convex dimension}\label{compcon}$ $

Having justified our restriction to $\scrL$-monomials, we move ahead
with computing lengths of chains.

\begin{rem}\label{rem_basis_for_HS_kernel}
Let $K$ be an HS-kernel of $\overline{F(\Lambda)}$.
By definition there are   $\scrL$-monomials $f_1,...,f_t\in
\operatorname{HP}(K)$ such that $K = \langle \sum_{i=1}^{t} |f_i|
\rangle$. By Remark \ref{rem_HS_kernel_finitely_spanned},
 $\ConSpan(K)$ is convexly spanned by ${f_1,...,f_t}$. Now,
 since $\ConSpan(K) = \ConSpan(f_1,...,f_t)$ and $\{f_1,...,f_t\} \subset \operatorname{HP}(K) \subset \operatorname{HP}(\overline{F(\Lambda)})$,
 by Corollary~\ref{cor_basis_for_HP}, $\{f_1,...,f_t\}$ contains a basis $B=\{b_1,...,b_s\} \subset \{f_1,...,f_t\}$ of $F$-convexly independent elements, where $s = \condeg(K),$ such that $\ConSpan(B) = \ConSpan(f_1,...,f_t)=\ConSpan(K)$.
\end{rem}

%

\begin{prop}\label{prop_order_independence_1}
For  any order $\nu$-kernel $o$ of $\overline{F(\Lambda)}$, if
$\scrL$-monomials $h_1,...,h_t $  are $F$-convexly dependent, then
the images of $ h_1,...,h_t $  are   $F$-convexly dependent (in
the quotient \vsemifield0 $\overline{F(\Lambda)}/o$).
\end{prop}
\begin{proof}
Denote by $\phi_o: \overline{F(\Lambda)} \rightarrow
\overline{F(\Lambda)}/o$ the quotient $F$-homomorphism. Then
$\phi_o( \langle F \rangle) = \langle \phi_o(F) \rangle = \langle
F \rangle_{\overline{F(\Lambda)}/ o}$. Now, if $ h_1,...,h_t  $
are $F$-convexly dependent then there exist some $j$, say without
loss of generality $j=1$, such that $ h_1 \in \langle h_2,...,h_t
\rangle \cdot \langle F \rangle$.
 By assumption and Proposition~\ref{prop_principal_ker},
\begin{align*}
\phi_o( h_1) \ &   \in \phi_o \left(\langle h_2,...,h_t, \alpha \rangle \right)\\
& = \left\langle \phi_o(h_2),...,\phi_o(h_t),\phi_o(\alpha) \right\rangle = \langle\phi_o(h_2),...,\phi_o(h_t), \alpha \rangle\\
& = \langle \phi_o(h_2),...,\phi_o(h_t)  \rangle \cdot \langle F
\rangle
\end{align*}
Thus $\phi_o( h_1)  $ is $F$-convexly dependent on $\{
\phi_o(h_2),...,\phi_o(h_t) \}$.
\end{proof}

\begin{flushleft}Conversely, we have:\end{flushleft}

\begin{lem}\label{lem_order_independence_2}
For any order $\nu$-kernel   $o$  of $\overline{F(\Lambda)}$, and
any set $\{ h_1,...,h_t \}$  of $\scrL$-monomials, if
$\phi_o(h_1),...,\phi_o(h_t)$ are $F$-convexly dependent in the
quotient \vsemifield0   $\overline{F(\Lambda)}/o$  and
$\sum_{i=1}^{t}\phi_o(|h_i|)\cap F = \{1\}$, then
 $  h_1,...,h_t  $   are $F$-convexly dependent   in
$\overline{F(\Lambda)}$.
\end{lem}
\begin{proof}
Note that $\sum_{i=1}^{t}\phi_o(|h_i|) \cap F = \{1\}$ if and only
if $\bigcap_{i=1}^{t}\skel(h_1) \cap \skel(o) \neq \emptyset$.
Translating the variables by a point $a \in
\bigcap_{i=1}^{t}\skel(h_1) \cap \skel(o)$, we may assume that the
constant coefficient of each $\scrL$-monomial $h_i$ is~$1$. Assume
that $\phi_o(h_1),...,\phi_o(h_t)$ are $F$-convexly dependent. We
may assume that $\phi_o(h_1)$ is $F$-convexly dependent on
$\phi_o(h_2),...,\phi_o(h_t)$.  This means by
Definition~\ref{defn_convex_dep} that we can take $h_{t+1} \in F$
for which $\phi_o(h_1)\in \langle
\phi_o(h_2),...,\phi_o(h_t),\phi_o(h_ {t+1}) \rangle$. Taking the
pre-images of the quotient map yields $$\langle h_1 \rangle \cdot o
\subseteq \langle h_2,...,h_t, h_{t+1} \rangle \cdot o.$$ Take an
$\scrL$-monomial $g$ such that   $1 +  g$ generates $o$. By
Corollary~\ref{cor_principal_ker_by_order}, there exists some $k \in
\mathbb{N}$ such that
\begin{equation}\label{eq_order_independence_2_1}
|h_1|  +  | 1  +  g| \leq_{\nu} (|h_2|  +  \dots  +  |h_{t+1}|  + |1
+  g|)^{k} = |h_2|^{k}  +  \dots  +  |h_{t+1}|^{k}  +  |1  + g|^{k}.
\end{equation}
 As $1  +  g \geq_{\nu} 1$ we have that $|1  +  g | \nucong 1  +  g$, and the
 right hand side of Equation~ \eqref{eq_order_independence_2_1} equals $$|h_2|^{k}  +  \dots  +  |h_{t+1}|^{k}  +  (1  +  g)^{k} \nucong |h_2|^{k}  +  \dots  +  |h_{t+1}|^{k}  +  1  +  g^{k} \nucong |h_2|^{k}  +  \dots +  |h_{t+1}|^{k}  +  g^{k}.$$ The last equality is due to the fact that $\sum|h_i|^k \geq_{\nu} 1$ so that $1$ is absorbed. The same argument,
  applied to the left hand side of Equation \eqref{eq_order_independence_2_1}, yields that
\begin{equation}\label{eq_order_independence_2_2}
|h_1|   +  g \leq_{\nu} |h_2|^{k}  +  \dots  +  |h_{t+1}|^{k}  +
g^{k}.
\end{equation}
Assume on the contrary that $h_1$ is $F$-convexly independent of $\{
h_2 ,...,h_t  \}$. Then $$\langle h_1 \rangle \not \subseteq \langle
h_2,...,h_{t+1} \rangle \nucong \bigg\langle \sum_{i=2}^{t+1}|h_i|
\bigg\rangle.$$ Thus for any $m \in \mathbb{N}$ there exists some
$\bfa_m \in F^{(n)}$ such that
$$|h_1(\bfa_m)| >_{\nu} \bigg|\sum_{i=2}^{t+1}|h_i(\bfa_m)|\bigg|^{m}
\nucong \sum_{i=2}^{t+1}|h_i(\bfa_m)|^m .$$ Thus by equation
\eqref{eq_order_independence_2_2} and the last observation we get
that $$\sum_{i=2}^{t}|h_i(\bfa_m)|^m   +  g(\bfa_m) <_{\nu}
|h_1(\bfa_m)| + g(\bfa_m) \leq_{\nu} \sum_{i=2}^{t}|h_i(\bfa_m)|^k +
g(\bfa_m)^k,$$ i.e., there exists some fixed $k \in \mathbb{N}$ such
that for any $m \in \mathbb{N}$,
\begin{equation}\label{eq_order_independence_2_3}
\sum_{i=2}^{t}|h_i(\bfa_m)|^m  <_{\nu} \sum_{i=2}^{t}|h_i(\bfa_m)|^k
+ g(\bfa_m)^k.
\end{equation}
For $m > k$, since  $|\gamma|^k \leq_{\nu} |\gamma|^m$ for any
$\gamma \in F$, we get that $\sum_{i=2}^{t}|h_i(\bfa_m)|^m
\geq_{\nu} \sum_{i=2}^{t}|h_i(\bfa_m)|^k$. Write $$g^k  = g(1)g' .$$
Since $g^k$ is an HP-kernel, $g(1)$ is the constant coefficient of
$g$ and $g'$ is a Laurent monomial with coefficient $1$.

According to the way the $\bfa_m$ were chosen,
$\sum_{i=2}^{t}|h_i(\bfa_m)|
> 1$ and $\sum_{i=2}^{t}|h_i(\bfa_m)|^m <_{\nu} g(\bfa_m)^k$, and thus
  $g(1) <_{\nu} g'(\bfa_{m_0})$ for  large enough $m_0$. But
$g'(\bfa_{m}^{-1}) \nucong g'(\bfa_{m})^{-1}$ so
$$g^k(\bfa_{m_0}^{-1}) \nucong g(1)g'(\bfa_{m_0}^{-1}) \nucong
g(1)g'(\bfa_{m_0})^{-1} <_{\nu} 1.$$ Thus
\eqref{eq_order_independence_2_3} yields
$\sum_{i=2}^{t}|h_i(\bfa_{m_0}^{-1})|^m <_{\nu}
\sum_{i=2}^{t}|h_i(\bfa_{m_0}^{-1})|^k$, a contradiction.
\end{proof}

%

\begin{prop}\label{prop_order_independence_2}
Let $R$ be a region $\nu$-kernel of $\overline{F(\Lambda)}$.  Let
$\{ h_1,...,h_t \}$ be a set of $\scrL$-monomials such that $(R
\cdot \langle h_1, ...,h_t \rangle) \cap F = \{1 \}$. Then $h_1
\cdot R,...,h_t \cdot R$ are   $F$-convexly dependent in the
quotient \vsemifield0 $\overline{F(\Lambda)}/R$ if and only if
$h_1,...,h_t$ are $F$-convexly dependent in
$\overline{F(\Lambda)}$.
\end{prop}
\begin{proof}
The `if' part of the assertion follows from Proposition
\ref{prop_order_independence_1}. Since $R = \prod_{i=1}^{m}o_i$
for suitable order $\nu$-kernels $\{ o_i \}_{i=1}^{m}$, the `only
if' part follows from Lemma \ref{lem_order_independence_2} applied
repeatedly to each of these $o_i$'s.
\end{proof}
%

\begin{prop}\label{prop_convexity_degree_invariance_for_HS_kernels}
Let $R \in \PCon(\overline{F(\Lambda)})$ be a region $\nu$-kernel.
Then, for any set $L$ of~HS-fractions, we have $$\condeg(L) =
\condeg(L \cdot R),$$  the right side taken in
$\overline{F(\Lambda)}/R$.
\end{prop}

\begin{proof}   $\condeg(L) \leq  \condeg(R \cdot L)$ since $L \subseteq  R
\cdot L$. For the reverse inequality, let $\phi_{R} :
\overline{F(\Lambda)} \rightarrow \overline{F(\Lambda)}/R$ be the
quotient map. Since $L$ is a sub-\vsemifield0 of
$\overline{F(\Lambda)}$, $\phi_{R}^{-1}(\phi_R(L))=R \cdot \langle
L \rangle$ by Theorem \ref{thm_nother_1_and_3}, and $\condeg(L)
\geq \condeg(\phi_R(L))$ by Proposition
\ref{prop_order_independence_1}, while $\condeg(\phi_R(L)) \geq
\condeg(\phi_{R}^{-1}(\phi_R(L)))$ by Lemma
\ref{lem_order_independence_2}. Thus
  $\condeg(L) \geq  \condeg(R \cdot L)$.
\end{proof}

\end{exmp}

%
%
%

In this way we see that $K \mapsto \Omega(K)$ yields a
homomorphism of $\nu$-kernels. Hence $\Omega$ is a natural map in
the sense of Definition~\ref{natural}, and we can apply
Theorem~\ref{Schreier}.

\begin{rem}\label{rem_condeg_not_effected_by_region}
Let $R$ be a region $\nu$-kernel
and let
$$A = \{ \langle g \rangle : \langle g \rangle \cdot \langle F
\rangle \supseteq  R \cdot \langle F \rangle \}.$$ Then
$\condeg(\overline{F(\Lambda)}/R) = \condeg(A)$, in view of
Remark \ref{rem_correspondence_of_quontient_HSpec} and
Proposition \ref{prop_order_independence_2}. As $L_{a} \in A$ for
any $a \in \skel(R) \neq \emptyset$
and $\condeg(L_a) = \condeg(\overline{F(\Lambda)}),$ we conclude that $\condeg(A) = \condeg(\overline{F(\Lambda)})$.\\
\end{rem}

We are ready to prove    catenarity for  $\condeg.$

\begin{thm}\label{caten}
If $R$ is a region $\nu$-kernel and $L$ is an HS-kernel  of
$\overline{F(\Lambda)}$, then
$$\condeg(\overline{F(\Lambda)}/LR) = \condeg(\overline{F(\Lambda)}) - \condeg(L).$$
In particular, $$\condeg(\overline{F(\Lambda)}/LR) = n -
\condeg(L).$$
\end{thm}
\begin{proof}
 $\overline{F(\Lambda)}/LR \cong (\overline{F(\Lambda)}/R)/ (L\cdot R/R)$, by the third isomorphism theorem.
 Choose a basis for HP$(\overline{F(\Lambda)}/R)$ containing a basis for HP$(L\cdot R/R)$.
 Then  $\condeg(L \cdot R/R) = \condeg(\phi_R(L))$, by Remark \ref{rem_condeg_not_effected_by_region}.
But $L \cdot R \cap F = \{ 1 \}$. Hence, by Proposition
\ref{prop_order_independence_2},
 $\condeg(\phi_R(L)) = \condeg(L)$. So $$\condeg(\overline{F(\Lambda)}/LR) = \condeg(A)-\condeg(L) = \condeg(\overline{F(\Lambda)}) -
 \condeg(L).$$

Thus, $$\condeg(\overline{F(\Lambda)}/LR) = \condeg(A)-\condeg(L)
= n - \condeg(L).$$
\end{proof}

\subsection{Recapitulation for dimensions}$ $

In conclusion, for every principal regular $\nu$-kernel $\langle f
\rangle \in P(\overline{F(\Lambda)})$, we have  obtained explicit
region $\nu$-kernels
$$\{R_{1,1},...,R_{1,s},R_{2,1},...,R_{2,t} \}$$ having trivial
intersection, such that
$$\langle f \rangle = \bigcap_{i=1}^{s}K_i \cap \bigcap_{j=1}^{t}N_j$$
where $K_i = L_i \cdot R_{1,i}$ for $i =1,...,s$ and appropriate
HS-kernels $L_i$ and $N_j = B_j \cdot R_{2,j}$ for $j =1,...,t$
and appropriate bounded from below $\nu$-kernels $B_j$. If
$\langle f \rangle \in \PCon(\langle F \rangle)$,  then, in view
of Theorem~\ref{thm_HP_expansion} we can take $B_j = \langle F
\rangle$ for every $j = 1,...,t$. Note that over the various
regions in $F^{(n)}$ corresponding to the region
$\nu$-kernels~$R_{i,j}$, $f$ is locally represented by distinct
HS-fractions in $\HSpec(\overline{F(\Lambda)})$. In fact each
region is defined so
 that the local HS-representation of $f$ is given over the entire region.
 Thus the    $R_{i,j}$'s  defining the partition of the space
 can be obtained as a minimal set of regions over each of which $\langle f \rangle$
takes the form of an HS-kernel.

For each $j = 1,...,t$,  $\condeg(N_j) = \Hdim(N_j) = 0$,  since
$N_j$ contains no elements of
$\operatorname{HP}(\overline{F(\Lambda)})$, implying
$$\condeg(\overline{F(\Lambda)}/N_j) = \Hdim(\overline{F(\Lambda)}/N_j) = n.$$ For each $i
= 1,...,s$,
 $\condeg(K_i)= \condeg(L_i) = \Hdim(L_i) \geq 1$, implying
$$\condeg(\overline{F(\Lambda)}/K_i) =
\Hdim(\overline{F(\Lambda)}/K_i) = n - \Hdim(L_i) < n.$$

%
\begin{rem}
In view of the discussion in \cite[\S 9.2]{Kern}, each term
$\overline{F(\Lambda)}/L_i$ corresponds to the linear subspace
of~$F^{(n)}$ (in logarithmic scale) defined by the linear
constraints endowed on the quotient $\overline{F(\Lambda)}/L_i$ by
the HS-kernel~$L_i$. One can think of these terms as an algebraic
description of the affine subspaces
 locally comprising  $\skel(f)$.
\end{rem}

\section{The Jordan–-H\"{o}lder theorem}$ $

Our final goal is to obtain a Jordan-H\"{o}lder theorem for
$\nu$-kernels. But there are too many $\nu$-kernels for a viable
theory in general, as discussed in \cite{Bi}, and anyway, as
pointed out in the introduction, the order $\nu$-kernels need not
alter the geometric dimension,so we want to ignore them in this
section. If we limit our set of $\nu$-kernels to a given
sublattice, of HS-kernels, one can use the Schreier refinement
theorem \cite{Rot} to obtain a version of the Jordan--H\"{o}lder
Theorem.

\begin{defn}\label{natural}
 $\mathcal L(\mathcal S)$ denotes the lattice
of $\nu$-kernels of a \vsemifield0 $\mathcal S.$

$\Psi$ is a \textbf{natural map} if for each \vsemifield0
$\mathcal S,$ there is a lattice homomorphism $\Psi_{\mathcal S}:
\mathcal L(\mathcal S)\to \mathcal L(\mathcal S) $ such that $K
\mapsto \Psi_{\mathcal S}(K)$ is a homomorphism of $\nu$-kernels.
We write $\Psi(\mathcal S) $ for $\Psi_{\mathcal S}(\mathcal
L(S)),$ and call the $\nu$-kernels in $\Psi({\mathcal S})$
$\Psi$-\textbf{kernels}. (We delete $\mathcal S$ when it is
unambiguous.)

 A $\Psi(\mathcal S)$-\textbf{simple} $\nu$-kernel is a minimal
$\Psi$-kernel $\ne \{1 \}$.  A $\Psi(\mathcal
S)$-\textbf{composition series}
 $\mathcal C(K,L)$ in $\Psi(\mathcal S)$ from  a $\nu$-kernel~$K$ to a $\nu$-subkernel $L$ is
  a chain $$K = K_0 \supset K_1 \supset \dots K_t = L$$  in $\Psi$
such that each factor is $\Psi$-simple. \end{defn}

By Theorem~\ref{thm_kernels_hom_relations} $\mathcal C(K,L)$ is equivalent to the
$\Psi(\mathcal S/L)$-composition series
$$K/L \supset K_1/L \supset \dots \supset K_t/L = 0$$ of $K/L.$

Given a $\Psi$-kernel $K$, we define its {\bf composition length}
$\ell (K)$ to be the length of a $\Psi$-composition series for $K$
(presuming $K$ has one). By definition, $\{ 1\}$ is the only
$\Psi$-kernel of composition length 0. A nonzero $\Psi$-kernel $K$
is simple iff $\ell(K) = 1.$ The next theorem is a standard
lattice-theoretic result of Schreier and Zassenhaus, yielding the
Schreier-Jordan-H\"{o}lder Theorem, cf.~\cite[Theorem~3.11]{Ro}.

\begin{thm}\label{Schreier}  Suppose $K$ has a
composition series  $$K = K_0 \supset K_1 \supset \dots \supset K_t
= 0,$$  which we denote as $\mathcal C$. Then:

(i) Any arbitrary finite chain of $\nu$-subkernels $$K = N_0
\supset N_1 \supset \dots \supset N_k\supset 0$$ (denoted as
$\mathcal D$),  can be refined to a composition series equivalent
to $\mathcal C$. In particular, $k\le t$.

(ii)  Any two composition series of $K$ are equivalent.

(iii) $\ell (K) = \ell(N) + \ell (K/N)$ for every $\nu$-subkernel
$N$ of $K$. In particular, every $\nu$-subkernel and every
homomorphic image of a $\nu$-kernel with composition series has a
composition series.
\end{thm}

To apply this, we need the following observation.

\begin{prop}\label{prop_kernel_descending_chain}
Let $L$ be an HS-kernel in $\overline{F(\Lambda)}$ with
$\skel(L)\neq\emptyset$. Let $\{ h_1, ...,h_t \}$ be a set of
$\scrL$-monomials in $\HSpec(\overline{F(\Lambda)})$ such that
$\ConSpan(h_1, ...,h_t) = \ConSpan(L)$ and let $L_i = \langle h_i
\rangle$. Then the chain

\begin{equation}\label{eq_desc_chain_kernels}
L = \prod_{i=1}^{u}L_i  \supseteq \prod_{i=1}^{u-1}L_i  \supseteq
\dots \supseteq L_1  \supseteq \langle 1 \rangle .
\end{equation}
of HS-kernels is  strictly descending   if and only if
$h_1,....,h_u$ are $F$-convexly independent.
\end{prop}
\begin{proof}
 $(\Rightarrow)$ If $h_u$ is
$F$-convexly dependent on $\{h_1,....,h_{u-1}\}$, then   $L_u =
\langle h_u \rangle  \subseteq \prod_{i=1}^{u-1}L_i \cdot \langle F
\rangle$. Assume that $L_u = \langle h_u \rangle  \not \subseteq
\prod_{i=1}^{u-1}L_i$. Then $\langle F \rangle \subseteq
\prod_{i=1}^{u}L_i$, implying that $\prod_{i=1}^{u}L_i$ is not an
HS-kernel. Thus $L_u = \langle h_u \rangle   \subseteq
\prod_{i=1}^{u-1}L_i$, and the chain is not strictly descending.

  $(\Leftarrow)$  $\skel(\prod_{i=1}^{t}L_i )
\subseteq \skel(L) \neq \emptyset$ for every $0 \leq t \leq u$,
implying  that $(\prod_{i=1}^{t}L_i ) \cap F = \{ 1 \}$ for every $0
\leq t \leq u$ (for otherwise
$\skel(\prod_{i=1}^{t}L_i)~=~\emptyset$). If $\{h_1,....,h_u\}$ is
$F$-convexly independent then  $L_u = \langle h_u \rangle \not
\subseteq \prod_{i=1}^{u-1}L_i \cdot \langle F \rangle$. By
induction, the
 chain \eqref{eq_desc_chain_kernels}  is strictly descending.

\end{proof}

\begin{thm}\label{prop_kernel_descending_chain_properties}
If $L \in \HSpec(\overline{F(\Lambda)})$, then $\hgt(L) =
\condeg(L)$, cf.~Definition~\ref{height0}.
 Moreover, every factor of a  descending chain of HS-kernels of maximal length is an  HP-kernel.
\end{thm}
\begin{proof}
By Proposition \ref{prop_kernel_descending_chain},   the maximal
length of a chain of   HS-kernels descending from an HS-kernel $L$
equals the number of elements in a basis of $\ConSpan(L)$; thus  the
chain is of unique length $\condeg(L)$, i.e., $\hgt(L) =
\condeg(L)$. Moreover, by Theorem \ref{thm_nother_1_and_3}(2),
$$\prod_{i=1}^{j}L_i / \prod_{i=1}^{j-1}L_i   \cong L_j/\bigg(L_j \cap \prod_{i=1}^{j-1}L_i \bigg).$$
Furthermore $$\left(L_j \cdot (L_j \cap (\prod_{i=1}^{j-1}L_i
))\right) \cap F = \{1\},$$ since $L_j \cdot (L_j \cap
\prod_{i=1}^{j-1}L_i  ) = L_j \cap \prod_{i=1}^{j}L_i   \subset
\prod_{i=1}^{j}L_i $
 and $(\prod_{i=1}^{j}L_i ) \cap F = \{1\}$.
 So the  image of the HP-kernel~$L_j$ in $\overline{F(\Lambda)}/(L_j \cap (\prod_{i=1}^{j-1}L_i))$  is an HP-kernel. Thus, every factor of the chain is an HP-kernel.
\end{proof}

\begin{cor}\label{correctdim}
$\Hdim(\overline{F(\Lambda)}) = \condeg(\overline{F(\Lambda)})=
n.$
\end{cor}

\begin{rem}\label{cover_of_regions_sub_direct_product}
If $R_1 \cap R_2 \cap \dots \cap R_t = \{1 \}$, then
$\overline{F(\Lambda)}$ is a subdirect product
$$\overline{F(\Lambda)} = \overline{F(\Lambda)}/(R_1 \cap R_2 \cap \dots \cap R_t) \hookrightarrow \prod_{i=1}^{t} \overline{F(\Lambda)}/R_i.$$
Then for  any $\nu$-kernel $K$ of $\overline{F(\Lambda)}$, $R_1
\cap R_2 \cap \dots \cap R_t \cap K = \bigcap_{i=1}^t (R_i \cap K)
= \{1 \}$ and, since $K$ itself is an idempotent \vsemifield0,
$$K = K/\bigcap_{i=1}^t (R_i \cap K) \cong \prod_{i=1}^{t} K/(R_i \cap K) \cong \prod_{i=1}^{t} R_i K/R_i.$$
\end{rem}


\bibliographystyle{amsplain}

\providecommand{\bysame}{\leavevmode\hbox
to3em{\hrulefill}\thinspace}
\providecommand{\MR}{\relax\ifhmode\unskip\space\fi MR }
\providecommand{\MRhref}[2]{%
  \href{http://www.ams.org/mathscinet-getitem?mr=#1}{#2}
} \providecommand{\href}[2]{#2}

\end{document}